\documentclass[a4paper,12pt,twoside]{amsart}

\usepackage{amsfonts,amssymb,amscd,bbold}
\usepackage{enumerate,nicefrac}
\theoremstyle{plain}
\newtheorem{thm}{Theorem}[section]
\newtheorem{lem}[thm]{Lemma}
\newtheorem{pro}[thm]{Proposition}
\newtheorem{cor}[thm]{Corollary}
\newtheorem{thmABC}{Theorem}

\newtheorem{proABC}[thmABC]{Proposition}

\theoremstyle{definition}
\newtheorem{dfn}[thm]{Definition}
\newtheorem{exa}[thm]{Example}
\newtheorem{assumption}[thm]{Assumption}
\theoremstyle{remark}
\newtheorem{rmk}[thm]{Remark}
\newtheorem{prob}[thm]{Problem}

\numberwithin{equation}{section}

\makeatletter
\renewcommand*\env@matrix[1][*\c@MaxMatrixCols c]{%
  \hskip -\arraycolsep
  \let\@ifnextchar\new@ifnextchar
  \array{#1}}
\makeatother
\newcommand{\C}{\mathbb{C}}
\newcommand{\F}{\mathbb{F}}
\newcommand{\N}{\mathbb{N}}
\newcommand{\Q}{\mathbb{Q}}
\newcommand{\R}{\mathbb{R}}
\newcommand{\Z}{\mathbb{Z}}
\newcommand{\calB}{\mathcal{B}}
\newcommand{\calC}{\mathcal{C}}
\newcommand{\calE}{\mathcal{E}}

\newcommand{\calI}{\mathcal{I}}
\newcommand{\calJ}{\mathcal{J}}
\newcommand{\calL}{\mathcal{L}}
\newcommand{\calM}{\mathcal{M}}
\newcommand{\calO}{\mathcal{O}}
\newcommand{\calR}{\mathcal{R}}
\newcommand{\calT}{\mathcal{T}}
\newcommand{\calV}{\mathcal{V}}
\newcommand{\fb}{\mathfrak{b}}
\newcommand{\fd}{\mathfrak{d}}
\newcommand{\fe}{\mathfrak{e}}
\newcommand{\ff}{\mathfrak{f}}
\newcommand{\fg}{\mathfrak{g}}
\newcommand{\fh}{\mathfrak{h}}
\newcommand{\fk}{\mathfrak{k}}
\newcommand{\fn}{\mathfrak{n}}
\newcommand{\fo}{\mathfrak{o}}
\newcommand{\fu}{\mathfrak{u}}
\newcommand{\fV}{\mathfrak{V}}
\newcommand{\fX}{\mathfrak{X}}
\newcommand{\schX}{\mathfrak{X}}
\newcommand{\schY}{\mathfrak{Y}}

\newcommand{\Ip}{\mathfrak{p}}
\newcommand{\IP}{\mathfrak{P}}
\newcommand{\GL}{\mathsf{GL}}
\newcommand{\gl}{\mathsf{gl}}
\newcommand{\Un}{\mathsf{U}}
\DeclareMathOperator{\Stab}{Stab}
\DeclareMathOperator{\Ind}{Ind}
\DeclareMathOperator{\Res}{Res}
\DeclareMathOperator{\Hom}{Hom}
\DeclareMathOperator{\Irr}{Irr}
\DeclareMathOperator{\IN}{in}
\DeclareMathOperator{\stab}{stab}
\DeclareMathOperator{\rad}{rad}
\DeclareMathOperator{\Spec}{Spec}
\DeclareMathOperator{\Mat}{Mat}
\DeclareMathOperator{\Pfaff}{Pfaff}
\DeclareMathOperator{\tr}{tr}
\DeclareMathOperator{\Tr}{Tr}
\DeclareMathOperator{\spec}{Spec}
\DeclareMathOperator{\Aut}{Aut}
\DeclareMathOperator{\End}{End}
\DeclareMathOperator{\comp}{Comp}
\DeclareMathOperator{\vol}{vol}
\newcommand{\di}{\mathrm{d}}
\newcommand{\one}{\mathbb{1}}
\newcommand{\disc}{D^{\rm{disc}}}

\newcommand{\pre}{\mathrm{pre}}
\newcommand{\iso}{\mathrm{iso}}
\newcommand{\LS}[1]{(\!( #1 )\!)}
\newcommand{\PS}[1]{[\![ #1 ]\!]}
\newcommand{\coloneqq}{\mathrel{\mathop:}=}
\renewcommand{\epsilon}{\varepsilon}
\renewcommand{\phi}{\varphi}
\renewcommand{\rho}{\varrho}
\renewcommand{\theta}{\vartheta}

\begin{document}

\title[Zeta functions for representations of $p$-adic Lie groups]{Zeta
  functions associated to admissible representations of compact
  $p$-adic Lie groups}

\author{Steffen Kionke} %
\address{Karlsruhe Institute of Technology, Institute for Algebra and
  Geometry, Englerstr. 2, 76131 Karlsruhe, Germany} %
\email{steffen.kionke@kit.edu}

\author{Benjamin Klopsch} %
\address{Mathematisches Institut, Heinrich-Heine-Universit\"at,
  Universit\"atsstr.\ 1, 40225 D\"usseldorf, Germany} %
\email{klopsch@math.uni-duesseldorf.de}

\begin{abstract}
  Let $G$ be a profinite group.  A strongly admissible smooth
  representation $\rho$ of $G$ over $\C$ decomposes as a direct sum
  $\rho \cong \bigoplus_{\pi \in \Irr(G)} m_\pi(\rho) \, \pi$ of
  irreducible representations with finite multiplicities~$m_\pi(\rho)$
  such that for every positive integer $n$ the number $r_n(\rho)$ of
  irreducible constituents of dimension $n$ is finite.  Examples arise
  naturally in the representation theory of reductive groups over
  non-archimedean local fields. In this article we initiate an
  investigation of the Dirichlet generating function
  \[ \zeta_\rho (s) = \sum\nolimits_{n=1}^\infty r_n(\rho) n^{-s}
    = \sum\nolimits_{\pi \in \Irr(G)} \frac{m_\pi(\rho)}{(\dim \pi)^s}
   \]
   associated to such a representation $\rho$.
   
   Our primary focus is on representations $\rho = \Ind_H^G(\sigma)$
   of compact $p$-adic Lie groups $G$ that arise from
   finite-dimensional representations $\sigma$ of closed subgroups $H$
   via the induction functor.  In addition to a series of foundational
   results -- including a description in terms of $p$-adic integrals
   -- we establish rationality results and functional equations for
   zeta functions of globally defined families of induced
   representations of potent pro-$p$ groups.  A key ingredient of our
   proof is Hironaka's resolution of singularities, which yields
   formulae of Denef-type for the relevant zeta functions.

   In some detail, we consider representations of open compact
   subgroups of reductive $p$-adic groups that are induced from
   parabolic subgroups.  Explicit computations are carried out by
   means of complementing techniques: (i) geometric methods that are
   applicable via distance-transitive actions on spherically
   homogeneous rooted trees and (ii) the $p$-adic Kirillov orbit
   method.  Approach (i) is closely related to the notion of Gelfand
   pairs and works equally well in positive defining characteristic.
\end{abstract}

\thanks{We acknowledge support by DFG grant KL 2162/1-1.}

\maketitle

% set counter to {1} to avoid listing of subsections:
% set to (3) for full details
\setcounter{tocdepth}{1}
\tableofcontents
\thispagestyle{empty}

%%%%%

\section{Introduction} \label{sec:introdu} In recent years the subject
of representation growth has advanced with a primary focus on zeta
functions enumerating (i) irreducible representations of arithmetic
lattices and compact $p$-adic Lie groups, (ii) twist-isoclasses of
irreducible representations of finitely generated nilpotent groups;
for instance, see
\cite{LaLu08,AvKlOnVo13,AvKlOnVo16a,AvKlOnVo16b,AiAv16}
and~\cite{Vo10, StVo14, Ro16, DuVo17, HrMaRi18, Ro18}, or the relevant
surveys~\cite{Kl13,Vo15}.  The aim of this paper is to introduce and
study a new, more general zeta function that can be associated to any
`suitably tame' infinite-dimensional representation of a group.  Our
focus is on admissible smooth representations of compact $p$-adic Lie
groups that arise from finite-dimensional representations of closed
subgroups via the induction functor.  In addition to a series of
foundational results that include a description in terms of $p$-adic
integrals and provide a springboard for further investigations, we
establish in Theorem~\ref{thm:D} rationality results and functional
equations for zeta functions of globally defined families of induced
representations of potent pro-$p$ groups.

%%%

\subsection{Background on representations zeta functions}
A group $G$ is said to be representation rigid if, for each positive
integer~$n$, its number of (isomorphism classes of) irreducible
complex representations of degree~$n$, denoted $r_n(G)$, is finite.
%In the spirit of~\cite{GrSeSm88} 
The sequence $r_n(G)$, $n \in \mathbb{N}$, is encoded in a Dirichlet
generating function $\zeta_G(s) = \sum_{n=1}^\infty r_n(G) n^{-s}$,
yielding the conventional \emph{representation zeta function} of~$G$.
If $G$ has polynomially bounded representation growth, then
$\zeta_G(s)$ converges and defines a holomorphic function on a right
half-plane $\{ s \in \mathbb{C} \mid \mathrm{Re}(s) > \alpha(G) \}$,
where the abscissa of convergence $\alpha(G)$ reflects the polynomial
degree of representation growth.  In favourable circumstances the
function extends meromorphically to a larger domain, possibly the
entire complex plane.

For a semisimple complex algebraic group $G = \mathbf{G}(\mathbb{C})$
the representation zeta function~$\zeta_G(s)$, encoding irreducible
rational representations, is also known as the Witten zeta function;
see~\cite{Wi91}.  The analytic properties of such zeta functions have
been studied thoroughly, using the available classification of
irreducible representations in terms of highest weights; e.g.,
see~\cite{KoMaTs10,LaLu08}.  We are concerned with groups where a
similar classification is either impracticable or way out of reach.

Recent advances in \cite{AvKlOnVo13,AvKlOnVo16a,AvKlOnVo16b} concern
the representation zeta functions of arithmetic lattices in semisimple
locally compact groups (mostly in characteristic $0$) and their local
Euler factors.  The latter are, in fact, representation zeta functions
enumerating irreducible continuous representations of compact $p$-adic
Lie groups.  The main focus has been on the following aspects:
abscissae of convergence, possible meromorphic continuations, pole
spectra, and local functional equations.  There are fundamental
connections to Margulis super-rigidity and the classical Congruence
Subgroup Problem; for instance, a quantitative refinement of the
Congruence Subgroup Conjecture, due to Larsen and
Lubotzky~\cite{LaLu08}, has been reduced to the original conjecture
in~\cite{AvKlOnVo16b}.

In these investigations the main tools are: Lie-theoretic techniques,
the Kirillov orbit method, the character theory of finite groups, and
Clifford theory, all in parallel with algebro-geometric,
model-theoretic, and combinatorial methods from $p$-adic integration.
This includes, among others, Deligne--Lusztig theory for
representations of finite groups of Lie type, Hironaka's resolution of
singularities in characteristic~$0$, and aspects of the Weil
conjectures regarding zeta functions of smooth projective varieties
over finite fields.

Most relevant for our purposes are the results of Avni, Klopsch, Onn,
and Voll in~\cite{AvKlOnVo12,AvKlOnVo13}, including a `Denef formula'
for the representation zeta functions of principal congruence
subgroups of compact $p$-adic Lie groups that arise from a global Lie
lattice over the ring of integers of a number field;
compare~\cite{De87}.

\enlargethispage{1\baselineskip}
%%%

\subsection{Zeta functions associated to admissible representations}
Our motivation is to analyse more general enumeration problems than
the one underlying past research on representation growth.  Let $G$ be
a profinite group; in due course we specialise to the case where $G$
is a compact $p$-adic Lie group.  Key examples are the completions
$G = \widehat{\mathbf{G}(O_S)}$ and $G = \mathbf{G}(O_\Ip)$ of
arithmetic groups $\mathbf{G}(O_S)$ with respect to the profinite
topology or the $p$-adic topology associated to a non-archimedean
prime~$\Ip$; here $\mathbf{G}$ denotes a semisimple affine group
scheme over the ring of $S$-integers $O_S$ of a number field.  In
addition we are interested in the principal congruence subgroups of
the groups~$\mathbf{G}(O_\Ip)$.  In order to develop flexible methods
for enumerating irreducible representations by their degrees according
to natural weights, we shift emphasis and attach a zeta function to
every `well behaved' infinite-dimensional representation of~$G$.  This
point of view is motivated also by geometric applications, for
instance, in the context of cohomology growth, where such
representations are supported on direct limits of
cohomology groups arising from systems of finite coverings of a fixed
topological space; see~\cite{SaAd94}.  In particular, in the realm of
number theory the relevant profinite groups acting on such direct
limits are completions of arithmetic groups as described above; e.g.,
see~\cite{Ha87}.

Technically, we consider admissible smooth representations of~$G$
over~$\C$, a concept which arises in number theory mainly through the
study of automorphic representations, e.g., in the context of the
Langlands program.  Admissible smooth representations of reductive
groups over $p$-adic fields have been studied thoroughly since the
1970s, starting from the work of Casselman as well as Bernstein and
Zelevinsky; compare~\cite{BuHe06,Re10}.  Let $\Irr(G)$ denote the set
of all irreducible smooth representations of the profinite group~$G$,
up to isomorphism, and note that each $\pi \in \Irr(G)$ is
finite-dimensional.  An admissible smooth representation $\rho$ of $G$
decomposes as a direct sum
$\rho \cong \bigoplus_{\pi \in \Irr(G)} m_\pi(\rho) \, \pi$ of
irreducible representations with finite multiplicities~$m_\pi(\rho)$.
We say that $\rho$ is \emph{strongly admissible} if, in addition, for
every positive integer $n$ the number $r_n(\rho)$ of irreducible
constituents of dimension $n$ is finite; in this case we define the
\emph{zeta function} of the representation $\rho$ as the Dirichlet
generating function
\begin{equation} \label{equ:zeta-rho} \zeta_\rho (s) =
  \sum_{n=1}^\infty r_n(\rho) n^{-s} = \sum_{\pi \in \Irr(G)}
  \frac{m_\pi(\rho)}{(\dim \pi)^s} \qquad (s \in \C).
\end{equation}

This raises the fundamental question which infinite-dimensional $\rho$
are moreover \emph{polynomially strongly admissible}, i.e., for which
$\rho$ the zeta function $\zeta_\rho$ has finite abscissa of
convergence $\alpha(\rho)$ and thus converges in the non-empty right
half-plane
$\{ z \in \mathbb{C} \mid \mathrm{Re}(z) > \alpha(\rho) \}$.  Of
course, the property can easily be formulated in terms of conditions
on the multiplicities $m_\pi(\rho)$; however, the latter are typically
difficult to access.  We remark that the definition gives a natural
generalisation of the conventional representation zeta function of a
profinite group~$G$: indeed,
$\zeta_G(s) = \zeta_{\rho_\mathrm{reg}} (s+1)$ for the regular
representation~$\rho_\mathrm{reg} = \Ind_1^G(\one)$ of~$G$; see
Example~\ref{exa:conn-repr-zeta-functions}.

Generally, we are interested in relations between the algebraic
properties of a polynomially strongly admissible representation $\rho$
of the profinite group $G$ and the analytic properties of the
associated zeta function~$\zeta_\rho$.  Depending on the context,
different types of properties become relevant.  Analytic properties of
$\zeta_\rho$ comprise, for instance: its abscissa of convergence,
meromorphic extension, the location of zeros and poles, and possibly
functional equations.  Algebraic properties of $\rho$, on the other
hand, have to be defined without direct reference to the
multiplicities $m_\pi(\rho)$ of irreducible constituents.

The new zeta function defined in~\eqref{equ:zeta-rho} transforms well
with respect to basic operations on admissible representations, e.g.,
passing to the smooth dual, forming direct sums and tensor products.
This suggests that other operations, such as restricting to open
subgroups or taking invariants under normal subgroups, may also
transform the zeta functions in a controllable way.  In this paper we
study in detail representations $\rho = \Ind_H^G(\sigma)$ of $G$ that
are obtained by \emph{induction} from representations $\sigma$ of a
closed subgroup $H \le_\mathrm{c} G$.  The induction functor is a
central tool for producing infinite-dimensional representations from
finite-dimensional ones.
% The guiding question is: how does the zeta function $\zeta_\rho(s)$ of
% $\rho = \Ind_H^G(\sigma)$ relate to the zeta function
% $\zeta_\sigma(s)$ for polynomially strongly admissible~$\sigma$?
 
%%%

\subsection{Summary of main results}
In dealing with profinite groups, it is often convenient to pass to
open subgroups.  In Section~\ref{sec:twist-sim-prop-A} we establish
that various admissibility properties for a smooth representation
$\rho$ of a profinite group $G$ are common properties of the
twist-similarity class of $\rho$ defined there.  In particular, if
$\rho$ is polynomially strongly admissible, the abscissa of
convergence $\alpha(\rho)$ is an invariant of the corresponding
twist-similarity class.

We recall that a finitely generated profinite group $G$ is
representation rigid if and only if it is FAb, i.e., if every open
subgroup $K \le_\mathrm{o} G$ has
finite abelianisation~$K/[K,K]$; see \cite[Prop.~2]{BaLuMaMo02}.  We
introduce the following refined notion: $G$ is \emph{FAb relative to}
a closed subgroup $H \le_\mathrm{c} G$ if $K/(H \cap K)[K,K]$ is
finite for every open subgroup~$K \le_\mathrm{o} G$.  In
Theorem~\ref{thm:rel-FAb-equiv} we establish the following natural
characterisation.

\begin{thmABC} \label{thm:A}
  Let $H \le_\mathrm{c} G$ be a closed subgroup of a finitely
  generated profinite group $G$.  The following statements are
  equivalent.
  \begin{enumerate}[\rm (a)]
  \item The group $G$ is FAb relative to $H$.
  \item The functor $\Ind_H^G$ preserves strong admissibility.
  \item The induced representation $\Ind_{H}^G(\one_H)$ is strongly
    admissible.
  \end{enumerate}
\end{thmABC}

It is an open problem to characterise finitely generated profinite
groups with polynomial representation growth.  Theorem~\ref{thm:A}
leads to the following refined question: Under what conditions on
$H \le_\mathrm{c} G$ does the induction functor $\Ind_H^G$ preserve
polynomially strong admissibility?  In
Proposition~\ref{pro:ind-finite-dim-adm-rel-FAb}, we provide a partial
answer in the special setting where $G$ is a compact $p$-adic Lie
group for some prime~$p$; our result can be regarded as a
generalisation of~\cite[Prop.~2.7]{LuMa14}.

\begin{proABC} \label{proABC:B}
  Let $H \le_\mathrm{c} G$ be a closed subgroup of a compact $p$-adic
  Lie group~$G$.  If $G$ is FAb relative to~$H$ then
  $\Ind_H^G(\sigma)$ is polynomially strongly admissible for every
  finite-dimensional smooth representation $\sigma$ of $H$.
\end{proABC}

For uniformly powerful -- more generally, for finitely generated
torsion-free potent -- pro-$p$ groups the Kirillov orbit method
provides a powerful tool for handling the characters of smooth
irreducible representations; see~\cite{Go09}.  A crucial ingredient in
this context is the Lie correspondence between such groups and certain
$\fo$-Lie lattices, for suitable discrete valuation rings~$\fo$; we
recall that an $\fo$-Lie lattice is an $\fo$-Lie ring whose underlying
$\fo$-module is free of finite rank.  In
Section~\ref{sec:ind-uniform-groups} we generalise the approach used,
for instance, in \cite{AvKlOnVo12,AvKlOnVo13,AvKlOnVo16a} to describe
zeta functions of induced representations for potent pro-$p$ groups.
One of our key results is Proposition~\ref{pro:zeta-integral-formula};
it provides a formula for the relevant zeta function in terms of a
$p$-adic integral involving Pfaffian determinants.  We refer to
Sections~\ref{sec:LieLattices} and \ref{sec:p-adic-integrals} for an
explanation of the notation, in particular $\| \!\cdot\! \|_\Ip$,
appearing in the integrands; the canonical $p$-adic measure is
recalled in Remark~\ref{rmk:canonicalmeasure}.

\begin{proABC} \label{propABC:integrals} Let $\fo$ be a compact
  discrete valuation ring of characteristic~$0$ and residue
  characteristic~$p$, with uniformiser $\pi$. Put $\Ip = \pi \fo$ and
  $q = \lvert \fo/\Ip \rvert$.  Let $\fg$ be an $\fo$-Lie lattice and
  $\fh$ an $\fo$-Lie sublattice of $\fg$ such that
  $\lvert \fg : \fh + [\fg,\fg] \rvert < \infty$.  Write
  $n = \dim_\fo \fg$ and $m+1 = \dim_\fo \fg - \dim_\fo \fh$.  Let
  $r \in \N_0$ be such that $G = \exp(\pi^r \fg)$ is a potent pro-$p$
  group with potent subgroup $H = \exp(\pi^r \fh) \le_\mathrm{c} G$.
  Then the zeta function of $\rho = \Ind_H^G(\mathbb{1}_H)$ is given
  by the following integral formulae.

  \smallskip

  \textup{(1)} Writing
  $W = \{ w \in \Hom_\fo(\fg,\ff) \mid w(\fh) \subseteq \fo \}$, where
  $\ff$ is the fraction field of~$\fo$, we
  have
  \[
  \zeta_\rho(s) = q^{r(m+1)} \int_{w \in W} \left( \left\| \bigcup \{
      \Pfaff_k(w) \mid 0 \le k \le \lfloor \nicefrac{n}{2} \rfloor \}
    \right\|_\Ip \right)^{-1-s} \,\di\mu(w),
  \]
  where $\mu$ denotes the normalised Haar measure satisfying
  $\mu(\Hom_\fo(\fg,\fo)) = 1$

  \smallskip

  \textup{(2)} For simplicity, suppose further that the
  $\fo$-module $\fg$ decomposes as a direct sum
  $\fg = \fk \oplus \fh$.  Interpreting
  $\Hom_\fo(\fk,\fo)$ as the
  $\fo$-points $\fV(\fo)$ of the
  $m+1$-dimensional affine space $\fV$ over
  $\Spec(\fo)$, let $\fX$ denote the
  projectivisation $\mathbb{P} \fV$ over
  $\Spec(\fo)$.  For
  $0 \le k \le \lfloor \nicefrac{n}{2} \rfloor$, the map
  $w \mapsto \Pfaff_k(w)$ induces a sheaf of ideals $\calI_k$ on
  $\fX$, and
  \begin{equation*}%\label{eq:example-zeta-function}
  \zeta_\rho(s) = (1- q^{-1}) q^{r(m+1)} \sum_{\ell \in \Z}
  q^{-\ell(m+1)} \int_{\fX(\fo)} \left( \max_{0 \le
      k \le \lfloor n/2 \rfloor} \left\| \pi^{k \ell} \calI_k
    \right\|_\Ip \right)^{-1-s} \,\di\mu_{\fX,\Ip},
  \end{equation*}
  where $\mu_{\fX,\Ip}$ denotes the canonical $p$-adic measure on~$\fX(\fo)$.
\end{proABC}

In Proposition~\ref{pro:zeta-integral-formula-coord} we obtain a
concrete version of this formula, based on a particular choice of
coordinates.  We illustrate the usefulness of the explicit formula in
Section~\ref{sec:examples-orbit-method} by computing the zeta
functions of various representations induced from Borel or parabolic
subgroups.  In this context we discuss how our formulae relate to the
more common methodology used in~\cite{AvKlOnVo13} and elsewhere.

In Theorem~\ref{thm:main-summary} we obtain results on the rationality
of globally induced representations, their abscissae of convergence,
and local functional equations.  This generalises similar results for
conventional representation zeta functions, for instance, in
\cite{Ja06} and~\cite{AvKlOnVo13}.  Our result is a consequence of a
general discussion of certain zeta functions which are obtained as an
infinite series of Igusa integrals, similar to the one appearing in
part~(2) of Proposition~\ref{propABC:integrals}.  We think that this
approach is of independent interest, since it is very flexible and
provides a new perspective on the abscissa of convergence.

Let $K_0$ be a number field with ring of integers~$O_{K_0}$.  Fix a
finite set $S$ of maximal ideals of $O_{K_0}$ and let
$O_{K_0,S} = \{ a \in K_0 \mid \lvert a \rvert_\Ip \le 1 \text{ for } \Ip \not\in S \}$
denote the ring of $S$-integers in~$K_0$.  We consider number fields
$K$ that arise as finite extensions of~$K_0$.  Let $O_S = O_{K,S}$
denote the integral closure of $O_{K_0,S}$ in~$K$.  For a maximal
ideal $\Ip \trianglelefteq O_S$ we write $\kappa_\Ip = O_S/\Ip$ for
the residue field and denote its cardinality by $q_\Ip$.  Furthermore,
$O_\Ip = O_{K,S,\Ip}$ denotes the completion of $O_S$ with respect to
$\Ip$ and we fix a uniformiser $\pi_\Ip$ so that $O_\Ip$ has the
valuation ideal~$\pi_\Ip O_\Ip$.

Let $\fg$ be an $O_{K_0,S}$-Lie lattice.  For every finite extension
$K$ of~$K_0$ and every maximal ideal $\Ip \trianglelefteq O_S$ we
consider the $O_\Ip$-Lie lattice
$\fg_{\Ip} = O_\Ip \otimes_{O_{K_0,S}} \fg$ and, for $r \in \N_0$, its
principal congruence
sublattices~$\fg_{\Ip,r} = \pi_\Ip^{\, r} \fg_{\Ip}$.  For any given
$K$ and $\Ip$, the Lie lattice $\fg_{\Ip,r}$ is potent for all
sufficiently large integers~$r$ so that
$G_{\Ip,r} = \exp(\fg_{\Ip,r})$ is a potent pro-$p$ group; we say that
such $r$ are \emph{permissible} for~$\fg_\Ip$; compare
\cite[Prop.~2.3]{AvKlOnVo13}.

\begin{thmABC} \label{thm:D} As in the set-up described above, let
  $\fg$ be an $O_{K_0,S}$-Lie lattice and let $\fh \subseteq \fg$ be a
  Lie sublattice such that
  $\lvert \fg : \fh + [\fg,\fg] \rvert < \infty$.  Suppose that $\fh$
  is a direct summand as a submodule of the $O_{K_0,S}$-module~$\fg$,
  and put $m + 1= \dim_{O_{K_0,S}} \fg - \dim_{O_{K_0,S}} \fh$.

  For maximal ideals $\Ip \trianglelefteq O_{K,S}$, where $K$ ranges
  over finite extensions of~$K_0$, and for positive integers~$r$ that
  are permissible for $\fg_\Ip$ and $\fh_\Ip$, we consider the induced
  representation
  \[
  \rho_{\Ip,r} = \Ind_{H_{\Ip,r}}^{G_{\Ip,r}}(\one_{H_{\Ip,r}})
  \]
  associated to the pro-$p$ groups $G_{\Ip,r} = \exp(\fg_{\Ip,r})$ and
  $H_{\Ip,r}= \exp(\fh_{\Ip,r})$.
  
  \begin{enumerate}[\rm (1)]
  \item For each $\Ip$ there is a complex-valued function $Z_\Ip$ of a
    complex variable $s$ that is rational in $q_\Ip^{-s}$ with integer
    coefficients so that for all permissible $r$,
   \begin{equation*}
     \zeta_{\rho_{\Ip,r}}(s) = (1-q_\Ip^{-1})q^{r(m+1)}_\Ip Z_\Ip(s).
   \end{equation*}
 \item The real parts of the poles of the functions $Z_{\Ip}$, for all
   $\Ip$, form a finite subset $P_{\fg,\fh} \subseteq \Q$.
 \item There is a finite extension $K_1$ of $K_0$ such that, for all
   maximal ideals $\Ip \trianglelefteq O_{K,S}$, arising for extensions
   $K_1 \subseteq K$, and for all permissible~$r$, the abscissa of
   convergence of the zeta function $\zeta_{\rho_{\Ip,r}}$ satisfies
   \[
   \alpha(\zeta_{\rho_{\Ip,r}}) = \max P_{\fg,\fh}.
   \]
 \item There are a rational
   function $F \in \Q(Y_1, Y_2, X_1,\dots, X_g)$  and a finite set $T$
   of maximal ideals of $O_{K_0}$ with $S \subseteq T$ such that the
   following holds:

   For every maximal ideal $\Ip_0 \trianglelefteq O_{K_0}$ not
   contained in $T$ there are algebraic integers
   $\lambda_1 = \lambda_1(\Ip_0), \dots, \lambda_g = \lambda_g(\Ip_0)
   \in \C^\times$
   so that for every finite extension $K$ of~$K_0$ and every maximal
   ideal $\Ip$ dividing $\Ip_0$ ,
   \begin{equation*}
     Z_\Ip(s) = F(q^{-f} ,q^{-fs},\lambda_1^{\, f},\dots,\lambda_g^{\, f}),
   \end{equation*}
   where $q = q_{\Ip_0}$ and $f = [ \kappa_\Ip : \kappa_{\Ip_0}]$
   denotes the inertia degree of $\Ip$ over~$K_0$.

   Furthermore the following functional equation holds:
   \[
   \zeta_{\rho_{\Ip,r}}(s) \vert_{\substack{q \to q^{-1} \\ \lambda_j
       \to \lambda_j^{-1}}} =  q^{f (m+1) (1-2r)} \zeta_{\rho_{\Ip,r}}(s). 
   \]
 \end{enumerate}
\end{thmABC}

The theorem is a consequence of our general discussion of certain zeta
functions in Section \ref{sec:rationality-etc}; these functions are
defined as infinite series of Igusa integrals, similar to the one in
part (2) of Proposition~\ref{propABC:integrals}.  This class of zeta
functions is comparatively general and it is well suited to our
applications.  The proof of the rationality relies on Hironaka's
resolution of singularities, which yields a formula of Denef-type.
Moreover, the functional equation is obtained along the lines of
Denef--Meuser~\cite{DeMe91} and Voll~\cite{Vo10}.  Since the zeta
functions treated here differ somewhat from those considered by
previous authors, we include a concise but nearly complete discussion.
Our account is geared towards applications to representation zeta
functions; this has the positive side-effect that the underlying
arguments are less technical than in the still more general situation
treated in~\cite{Vo10}.  In addition, the designed shape of our zeta
functions allows us to give a new perspective on the abscissa of
convergence, leading naturally to the notion of local abscissae of
convergence.  In Corollary~\ref{cor:abscissaDensity} we deduce from
statement (3) in Theorem~\ref{thm:D} that the abscissa of convergence
is attained on a set of primes with positive Dirichlet density; in
this way, statement (3) can be seen as a strengthening
of~\cite[Theorem~B]{AvKlOnVo13}.  Retrospectively, the stronger
assertion could also be derived within the framework
of~\cite{AvKlOnVo13}, by a careful analysis of the proofs given there;
this was pointed out to us by the referee.

In Theorem~\ref{thm:distanceTransitiveFormula} we record a class of
examples, arising geometrically from distance-transitive actions of
profinite groups on spherically homogeneous rooted trees. This result
is based on ideas of Bekka, de la Harpe, and Grigorchuk discussed in
the appendix of \cite{BeHa03}.

\begin{thmABC} \label{thmABC:E} Let $\calT_m$ be a spherically
  homogeneous rooted tree with branching sequence
  $m = (m_n)_{n \in \N}$ in $\N_{\ge 2}$.  Let
  $G \le_\mathrm{c} \Aut(\calT_m)$ act distance-transitively on the
  boundary $\partial\calT_m$, and let $P = P_\xi \le_\mathrm{c} G$ be
  the stabiliser of a point~$\xi \in \partial\calT_m$.

  Then the induced representation $\rho_\partial = \Ind_P^G(\one_P)$
  decomposes as a direct sum of the trivial representation and a
  unique irreducible constituent $\pi_n$ of dimension
  $(m_n-1)\prod_{j=1}^{n-1} m_j$ for each $n\ge 1$.  In particular,
  the representation is multiplicity-free and the zeta function of
  $\rho_\partial$ is
  \begin{equation*}
    \zeta_{\rho_\partial}(s) = 1 + \sum_{i=1}^\infty (m_i-1)^{-s}
    \prod_{j=1}^{i-1} m_j^{\, -s} 
   \end{equation*}
   with abscissa of convergence $\alpha(\rho_\partial) =  0$.
\end{thmABC}

A pair $(G,P)$ such that the induced representation $\Ind_P^G(\one_P)$
is multiplicity-free is called a \emph{Gelfand pair}; for an
introduction to Gelfand pairs we refer the reader to~\cite{Gr91}.  The
zeta function of an induced representation of a Gelfand pair $(G,P)$
enumerates a subset of all irreducible representations of $G$, which
suggests that Gelfand pairs are of special interest in our setting.
As an application of these new methods we obtain in
Proposition~\ref{prop:formulaInductionTrees} an explicit formula for
the zeta functions associated to representations induced from maximal
$(1,n)$-parabolic subgroups to $\GL_{n+1}(\fo)$, where $\fo$ is a
compact discrete valuation ring.  This approach works in arbitrary
characteristic and complements the Kirillov orbit method described
above.

%%%

\subsection{Further discussion and open problems}

\subsubsection{} Our investigations are naturally related to the
representation theory of reductive groups over $p$-adic fields or
non-archimedean local fields in general; compare
Remark~\ref{rem:red-grps-char-p}.  For simplicity, consider the
reductive group $G = \GL_n(\Q_p)$.  We would like to gain a detailed
understanding of the smooth irreducible representations of~$G$, which
match via the local Langlands correspondence, established by
Harris--Taylor~\cite{HaTa01} and Henniart~\cite{He00}, with certain
representations of the Weil--Deligne group of~$\Q_p$.  Consider a
smooth irreducible representation $\rho$ of $G$ and its restriction
$\rho_{|K}$ to the maximal compact subgroup $K = \GL_n(\Z_p)$.  The
following question is fundamental: describe the decomposition of
$\rho_{|K}$ into irreducible components.  For $n=2$,
Casselman~\cite{Ca73} obtain detailed results, but for larger degrees
an answer seems to be unknown.

Already specific cases are of interest.  For instance, the unramified
principal series representations constitute a large class of smooth
irreducible representations of~$G$.  These representations are induced
to $G$ from $1$-dimensional representations on a Borel subgroup~$B$.
As $KB = G$, the restriction of such a representation to $K$ is
induced from $K\cap B$.  The study of zeta functions of induced
representations thus provides a quantitative approach to the
decomposition problem for unramified principal series representations.
The decomposition into irreducible $\GL_3(\Z_p)$-constituents of
unramified principal series representations of $\GL_3(\Q_p)$ was
determined by Onn and Singla~\cite{OnSi14}, completing
results of Campbell and Nevins~\cite{CaNe09}.  They used direct
representation-theoretic considerations and, as a consequence, deduced
a formula for the associated zeta function.  It is interesting
to conduct similar investigations for supercuspidal representations of
$G$; Nevins~\cite{Ne14} has made steps in this direction.

\subsubsection{} \label{rmk:other-fields} In this article we consider,
for simplicity, only smooth representations of profinite groups~$G$
over the complex field~$\C$.  More generally, following unpublished
work of Gonz\'alez-S\'anchez, Jaikin-Zapirain, and Klopsch, one could
derive analogues of several of our results also for representations
over fields $\F$ of characteristic $0$ that are not necessarily
algebraically closed.  Indeed, in that situation $\F$-irreducible
representations correspond to Galois orbits of irreducible complex
representations and the action of the absolute Galois group of $\F$
can be incorporated into the treatment; for general profinite
groups~$G$ it is not clear how to deal with the relevant Schur indices
of irreducible complex representations, but the Schur indices are
known to be trivial in the important case where $G$ is a pro-$p$ group
for an odd prime~$p$; compare \cite[Sec.~10]{Is94}. Finally, it would
be interesting to consider also the situation where $\F$ has positive
characteristic $\ell > 0$; compare Remark~\ref{rem:red-grps-char-p}.
Provided that $G$ has a trivial pro-$\ell$-Sylow subgroup, we expect
that many of our results can be suitably generalised.

\subsubsection{} Dirichlet generating functions have also been
employed to study the distribution of finite-dimensional irreducible
representations of finitely generated nilpotent groups.  For such a
group $\Gamma$ one defines and studies the zeta function enumerating
twist-isoclasses of irreducible representations of~$\Gamma$; for
instance, see~\cite{HrMaRi18,StVo14,Ro16,DuVo17}.  Many theorems on
representation zeta functions, e.g., regarding rationality, pole
spectra, and functional equations, have analogues in the
twist-isoclass setting; the Kirillov orbit method can be adjusted to
take into account twist-isoclasses and thus yields a basic tool.  It
is natural to investigate -- in analogy to our approach in this paper
-- zeta functions $\zeta_\rho$ associated to infinite-dimensional
twist-invariant representations $\rho$ of a finitely generated
nilpotent group~$\Gamma$ that are completely reducible into
finite-dimensional irreducible constituents: the Dirichlet series
$\zeta_\rho$ encodes the finite multiplicities of entire
twist-isoclasses rather than individual constituents.  Natural
examples occur again in the form of arithmetic groups: the
representation $\rho$ spanned by all finite-dimensional
subrepresentations of the (co-)induced representation
$\Ind_\Delta^\Gamma(\one_\Delta)$, where $\Gamma = \mathbf{G}(O)$ for
a unipotent affine group scheme $\mathbf{G}$ over the ring of integers
$O$ of a number field, and $\Delta = \mathbf{H}(O)$ for a subgroup
$\mathbf{H} \subseteq [\mathbf{G},\mathbf{G}]$.  In ongoing joint work
with Rossmann we pursue this line of research and study zeta functions
associated to (co-)induced representations of nilpotent groups; the
precise results are to appear elsewhere.

\subsubsection{} Our results on induced representations of compact
$p$-adic Lie groups can be seen as a starting point for further
investigations. In particular, the general setting of induced
representations provides many new tractable examples and significantly
more flexibility to vary input parameters.  To indicate possible
future directions we formulate the following concrete open questions.

\begin{prob}\label{problem:function-fields}
  Jaikin-Zapirain~\cite[Section~7]{Ja06} computed the conventional
  representation zeta function of $\mathsf{SL}_2(R)$, where $R$ is a
  compact discrete valuation ring of odd residue characteristic, and
  observed that it only depends on the residue field cardinality and
  not on the characteristic or isomorphism type of $R$ itself.  The
  same phenomenon occurs in our setting, for zeta functions associated
  to induced representations; see~\cite[Theorem~6.5]{OnSi14} and
  Proposition~\ref{prop:formulaInductionTrees} below.

  Find further examples or, possibly, counter-examples of this
  phenomenon in the context of induced representations.  Give an
  explanation of the invariance of the zeta function where it holds.
\end{prob}

\begin{prob}
  One of the key invariants of a representation zeta function is its
  abscissa of convergence.  The abscissa of the conventional
  representation zeta function is explicitly known for FAb compact
  $p$-adic Lie groups of small dimension and fully understood for
  norm-$1$ groups of $p$-adic division algebras; see
  \cite{Ja06,AvKlOnVo13,ZoXX} and
  \cite[Theorem~7.1]{LaLu08}.  However, there is no general
  interpretation, even at the conjectural level.  In our new setting,
  Propositions~\ref{pro:ind-Br-Gr}, \ref{pro:ind-Br-Ur}, and
  Theorem~\ref{thm:InductionMaxPara} provide the abscissae of
  convergence of zeta functions associated to some families of induced
  representations.

  Determine the abscissae for further families of induced
  representations.  For instance, consider representations of open
  compact subgroups of reductive $p$-adic groups that are induced from
  (maximal) parabolic subgroups.  Can the results be explained in
  terms of the root systems?
\end{prob}

\begin{prob}
  The conventional representation zeta functions of
  $\mathsf{SL}_3(\fo)$ and $\mathsf{SU}_3(\fo)$ over a compact
  discrete valuation ring $\fo$ of characteristic $0$ display an
  Ennola-type duality that can be realised concretely by simple
  sign-changes; see~\cite{AvKlOnVo13,AvKlOnVo16a}.  The examples of
  induced representations in Propositions~\ref{pro:ind-Br-Gr} and
  \ref{pro:ind-Br-Ur} (and the way we derive them) do not suggest
  unmistakably a similarly prominent form of duality.  A priori the
  induced representations we consider are sensitive to the choices of
  Borel subgroups; thus it is unclear whether some form of duality
  should actually be expected for the \emph{specific choices} made in
  Propositions~\ref{pro:ind-Br-Gr} and \ref{pro:ind-Br-Ur}.  To reveal
  and explain an Ennola-type duality in our setting would require a
  more conceptual treatment.

  Study systematically the zeta functions of induced representations
  in groups of Type~$\mathsf{A}_n$, initially for $n=2$, and determine
  when an Ennola-type duality holds.  Find an explanation for this
  duality.
\end{prob}

%%% 

\subsection{Organisation} \label{sec:org-not} In
Section~\ref{sec:basic-concepts} we recall some standard notions from
the representation theory of totally disconnected locally compact
groups.  Based on suitable refinements, we introduce zeta functions
associated to strongly admissible smooth representations.  In
Section~\ref{sec:twist-sim-prop-A} we study the robustness of strong
admissibility under the induction functor and introduce in this
context the notion of relative FAb-ness.  In
Section~\ref{sec:ind-uniform-groups} we develop a $p$-adic formalism
to compute the zeta functions of induced representations between
uniform pro-$p$ groups.  Our main tool is the Kirillov orbit method,
with special care taken for $p=2$.  In
Section~\ref{sec:rationality-etc} we establish the rationality of the
local factors of globally defined families of induced representations;
we also study their abscissae of convergence and prove functional
equations.  In Section~\ref{sec:action-on-trees} we produce, with
surprisingly little effort, instructive examples of zeta functions of
induced representations, arising geometrically for groups acting on
rooted trees.  In Section~\ref{sec:examples-orbit-method} we give a
number of concrete examples of zeta functions of induced
representations of potent pro-$p$ groups.

%%%%%

\section{Basic concepts and preliminaries} \label{sec:basic-concepts}

In this section we recall some standard notions from the
representation theory of totally disconnected locally compact groups;
see~\cite{BuHe06,Re10}.  Based on suitable refinements, we introduce
zeta functions associated to strongly admissible smooth
representations.
 
%%%

\subsection{Admissible smooth representations}
Let $G$ be a totally disconnected locally compact topological group;
for instance, $G$ could be a reductive $p$-adic group such
as~$G = \GL_n(\Q_p)$.  Observe that $G$ is automatically Hausdorff, as
$\{1\} \le_\mathrm{c} G$.  A complex representation of $G$ is given by
a homomorphism $\rho \colon G \rightarrow \GL(V_\rho)$, where $V_\rho$
is a vector space over~$\C$.  All representations we consider are
over~$\C$, and we usually drop the specification `complex'; compare
Section~\ref{rmk:other-fields}.  Depending on the situation, it is
convenient to denote a representation of $G$ either by the vector
space $V_\rho$ acted upon, or by the homomorphism
$\rho \colon G \rightarrow \GL(V_\rho)$, or by the
pair~$(\rho, V_\rho)$.

A vector $v \in V_\rho$ is said to be \emph{smooth} with respect to
$\rho$ if its stabiliser
$\Stab_G(v) = \{ g \in G \mid \rho(g).v = v \}$ is open in~$G$.  The
vector subspace
\[
V_\rho^\infty = \{ v \in V_\rho \mid \text{$v$ is smooth with respect to
  $\rho$} \}
\]
is $\rho$-invariant and gives rise to a sub-representation $\rho^\infty
= (\rho^\infty, V_\rho^\infty)$.  The representation $\rho$ is said to
be \emph{smooth} if $V_\rho = V_\rho^\infty$; equivalently, $\rho$ is
smooth if the map $G \times V_\rho \rightarrow V_\rho$ is continuous
when $G$ is equipped with its natural topology and $V_\rho$ is
equipped with the discrete topology.

We recall that the totally disconnected compact topological groups are
precisely the profinite groups.  The assertions of the following lemma
are easy to prove and well known; see~\cite[Sec.~2.2]{BuHe06}.

\begin{lem}
  Let $G$ be a profinite group.

  \textup{(1)} Every smooth representation of~$G$ is semisimple, i.e.,
  it decomposes as a direct sum of (smooth) irreducible constituents.

  \textup{(2)} The smooth irreducible representations of $G$ are
  precisely the finite-di\-men\-sional irreducible continuous
  representations of~$G$; in particular, each of these factors through
  a finite continuous quotient of~$G$.
\end{lem}

\begin{rmk} \label{rmk:decomp-lctd-reps} The lemma has the following
  consequence for a totally disconnected locally compact topological
  group~$G$. If $(\rho,V_\rho)$ is a smooth representation of $G$ and
  if $K \le_\mathrm{o} G$ is a compact open subgroup of $G$ then the
  restriction $\rho_{\vert K}$ is semisimple and all its irreducible
  constituents are finite-dimensional.  If, in addition,
  $(\rho,V_\rho)$ is irreducible and $\lvert G : K \rvert$ countable
  then $\dim V_\rho$ is countable.  A typical example of this
  situation is: $G = \GL_n(\Q_p)$ and $K = \GL_n(\Z_p)$.
\end{rmk}

A smooth representation $(\rho,V_\rho)$ of a totally disconnected
locally compact topological group $G$ is said to be \emph{admissible}
if for every compact open subgroup $K \le_\mathrm{o} G$ the space of
fixed vectors
\[
V_\rho^{\, K} = \{ v \in V_\rho \mid \forall g \in K \colon \rho(g).v
= v \}
\]
has finite dimension.  

\begin{rmk} \label{rem:red-grps-char-p}
  According to a classical result in the representation theory of
  reductive $p$-adic groups, every smooth irreducible representation
  of a reductive $p$-adic group $G$ is admissible; see~\cite[Ch.~II,
  2.8]{Vi96}.  It is worth noting that this result holds in
  great generality.  Firstly, it applies to irreducible
  representations over any -- not necessarily algebraically
  closed~\cite[Prop.~2]{Bl05} -- field of characteristic
  different from~$p$.  Secondly, the statement remains valid, when $G$
  is the group of rational points of a reductive group over a local
  field of \emph{positive} characteristic.  In view of
  Remark~\ref{rmk:decomp-lctd-reps}, this provides a wide range of
  interesting admissible smooth representations of profinite groups
  where it is natural to study the multiplicities of irreducible
  components.
\end{rmk}

%%%

\subsection{Polynomially strongly admissible representations} Let $G$
be a profinite group.  We denote by $\Irr(G)$ the set of (isomorphism
classes of) smooth irreducible -- i.e.\ finite-dimensional irreducible
continuous -- representations of~$G$.  For simplicity, we usually do
not distinguish notationally between representations and isomorphism
classes of representations.  It is easy to see that a representation
$(\rho, V_\rho)$ of $G$ is admissible smooth if and only if it
decomposes as a direct sum
\begin{equation}\label{equ:rho-decomp}
  V_\rho \cong \bigoplus_{\pi \in \Irr(G)} m(\pi,\rho) \cdot
  V_\pi,
\end{equation}
where the multiplicity $m(\pi,\rho)$ of $\pi$ in $\rho$ is finite for
each $\pi \in \Irr(G)$.  

We say that an admissible smooth representation $(\rho,V_\rho)$ of $G$
is \emph{strongly admissible} if, for every $d \in \N$, the number
\[
R_d(\rho) = \sum_{\substack{\pi \in \Irr(G) \\ \dim \pi \le d}}
m(\pi,\rho)
\]  
of its irreducible constituents of dimension at most $d$ is finite.
Furthermore, we define the \emph{zeta function} of a strongly
admissible smooth representation $(\rho,V_\rho)$ of $G$ to be the
formal Dirichlet series
\[
\zeta_\rho(s) = \sum_{\pi \in \Irr(G)} m(\pi,\rho) (\dim \pi)^{-s},
\]
where $s$ denotes a complex variable.  

\begin{rmk}
  If $G$ is finitely generated as a profinite group, it is known that
  $G$ has only finitely many irreducible representations of any given
  finite dimension if and only if $G$ is \emph{FAb}, meaning that
  every open subgroup $H \le_\mathrm{o} G$ has finite abelianisation
  $H/[H,H]$; compare~\cite{BaLuMaMo02}.  Consequently, for a FAb
  finitely generated profinite group every admissible smooth
  representation is strongly admissible.
\end{rmk}

We say that $(\rho,V_\rho)$ is \emph{polynomially strongly admissible}
if there exists $r \in \R_{>0}$ such that $R_d(\rho) = O(d^r)$, i.e.,
if there are $r,C \in \mathbb{R}_{>0}$ such that for every $d \in \N$,
\[
\sum_{\substack{\pi \in \Irr(G) \\ \dim \pi \le d}}
m(\pi,\rho) = R_d(\rho) \le C d^r.
\]
Equivalently, $(\rho,V_\rho)$ is polynomially strongly admissible if
its (polynomial) \emph{degree of irreducible constituent growth},
defined as
\[
\deg(\rho) = \inf \{ r \in \R_{>0} \mid R_d(\rho) = O(d^r) \} \in
\R_{\ge 0} \cup \{ \infty \},
\]
is finite.  

For a polynomially strongly admissible smooth representation
$(\rho,V_\rho)$ of $G$ the formal Dirichlet series $\zeta_\rho(s)$
converges (absolutely) and defines an analytic function on a right
half-plane $\{ s \in \C \mid \mathrm{Re}(s) > \alpha(\rho) \}$, where
$\alpha(\rho) \in \mathbb{R} \cup \{-\infty\}$ denotes the abscissa of
convergence.  In fact, if $\rho$ is infinite-dimensional then
$\alpha(\rho) = \deg(\rho)$.

\begin{exa} \label{exa:conn-repr-zeta-functions} The profinite group
  $G$ acts via right translations on the space
  $V = \calC^\infty(G,\C)$ of all locally constant functions,
  i.e.\ continuous functions when $\C$ is equipped with the discrete
  topology, from $G$ to $\C$:
  \[
  ({}^g f)(x) = f(xg) \qquad \text{for $g \in G$, $f \in
    \calC^\infty(G,\C)$, $x \in G$.}
  \]
  The resulting `regular' representation $\rho_\text{reg}$ is strongly
  admissible if $G$ has only finitely many irreducible representations
  of any given finite dimension.  As remarked above, this is for
  instance the case if $G$ is finitely generated and FAb.  In this
  case the formal Dirichlet series
  \[
  \zeta_{\rho_\text{reg}}(s) = \sum_{\pi \in \Irr(G)} (\dim
  \pi)^{1-s}
  \]
  is equal to $\zeta_G(s-1)$, a shift of the conventional
  representation zeta function $\zeta_G(s)$ of~$G$ that enumerates 
  irreducible complex representations of~$G$ and has been the sole focus
  of study until now; compare Section~\ref{sec:introdu}. \hfill $\diamond$
\end{exa}

Finally, we take interest in yet another finiteness condition.  An
admissible smooth representation $\rho$ of the profinite group $G$ is
said to have \emph{bounded multiplicities}, if there exists $M \in \N$
such that $m(\pi,\rho) \le M$ for all~$\pi \in \Irr(G)$.  

\begin{rmk}
  This concept arises naturally in two ways.  Firstly, there are
  prominent examples, such as multiplicity-free representations
  associated to Gelfand pairs; compare
  Section~\ref{sec:action-on-trees}.  Secondly, the abscissa of
  convergence of the zeta function associated to any smooth
  representation $\rho$ of $G$ with bounded multiplicities yields a
  lower bound for the abscissa of convergence of the conventional
  representation zeta function of~$G$.
\end{rmk}

\begin{dfn}[Admissibility properties]
  We refer to the properties of being admissible, strongly admissible,
  polynomially strongly admissible or having bounded multiplicities
  collectively as \emph{properties of type~$\mathrm{(A)}$}.
\end{dfn}

%%%

\subsection{Tensor products and induced representations} Let $G$ be a
profinite group.  The \emph{contragredient representation}
$\rho^\vee$ of a smooth representation $\rho$ of $G$
is the the smooth part $(\rho^*)^\infty $ of the abstract dual
representation~$\rho^*$.  The \emph{tensor product}
$(\rho \otimes \theta, V \otimes_\C U)$ of two smooth representations
$(\rho,V)$ and $(\theta,U)$ of~$G$, defined via
\[
(\rho \otimes \theta)(g).(v \otimes u) = \rho(g).v \otimes \theta(g).u
\qquad \text{for $g \in G$, $v \in V$, $u \in U$,}
\]
is a smooth representation of~$G$.  By considering the tensor product
$\rho \otimes \rho^\vee$ of an infinite-dimensional admissible smooth
representation $\rho$ with its contragredient
representation~$\rho^\vee$, one sees that admissibility need not be
preserved under tensor products.

Let $H \le_\mathrm{c} G$ be a closed subgroup.  The \emph{restriction}
$\Res_H^G(\theta)$ of a smooth representation $\theta$ of $G$ is a smooth
representation of~$H$, but, clearly, admissibility need not be preserved.
Conversely, the \emph{induced representation}
$\rho = \Ind_H^G(\sigma)$ of a smooth representation $(\sigma,W)$ of
$H$ is constructed as follows: $G$ acts on
\[
V_\rho = \{ f \in \calC^\infty(G,W) \mid \forall h \in H \,
\forall x \in G \colon f(hx) = \sigma(h).f(x) \}
\]
via right translation
\[
(\rho(g).f)(x) = ({}^g f)(x) = f(xg) \qquad \text{for $g, x \in G$, $f
  \in V_\rho$.}
\]   
In calling the representation `induced' rather than `co-induced', we
follow \cite[I.\S 2.5]{Se02} rather than \cite[VII.\S 6]{Se79}.  We note
that in the context of profinite groups the notions of induced,
compactly induced, and co-induced representations are in fact all the
same.  The next proposition is a consequence of Frobenius reciprocity;
compare~\cite[Sec.~2.4]{BuHe06}.

\begin{pro} \label{pro:admi-pres-ind}
  Let $G$ be a profinite group and $H \le_\mathrm{c} G$ a closed
  subgroup.  If $\sigma$ is an admissible smooth representation of $H$,
  then $\Ind_H^G(\sigma)$ is an admissible smooth representation
  of~$G$.
\end{pro}

The induction functor is one of the key tools to construct new
representations from known ones and, in particular, interesting
infinite-dimensional representations from finite-dimensional ones.

%%%%%

\section{Twist-similarity classes and admissibility
  properties} \label{sec:twist-sim-prop-A}

Throughout this section $G$ denotes a profinite group and
$H \le_\mathrm{c} G$ a closed subgroup.  Given an admissible smooth
representation $(\sigma, W)$ of~$H$, we are keen to study the zeta
function $\zeta_\rho$ attached to the induced representation
$\rho = \Ind_H^G(\sigma)$.  With this aim we develop conditions which
ensure that $\rho$ is (polynomially) strongly admissible.
Furthermore, we establish that the abscissa of convergence
$\alpha(\rho)$ of $\zeta_\rho$ is rather robust: it is, in fact, an
invariant of the twist-similarity class of~$\sigma$, defined below.

%%%

\subsection{Tensor products and twist-similarity}
Let $G$ be a profinite group.  We say that two smooth representations
$(\rho, V)$ and $(\rho', V')$ of $G$ are \emph{twist-similar} to one
another
%, in symbols $\rho \twistsim \rho'$, 
if there are non-zero finite-dimensional smooth representations
$(\theta,U)$ and $(\theta',U')$ and a homomorphism
$f \colon V\otimes_\C U \to V'\otimes_\C U'$ of $G$-representations
with finite-dimensional kernel and cokernel.  Note that,
in particular, all finite-dimensional smooth representations of $G$
are twist-similar to one another.

\begin{pro}\label{prop:TensorInvariance}
  Let $\rho, \theta$ be a smooth representations of a profinite
  group~$G$, and  suppose that
  $1 \le \dim \theta < \infty$.  Then
  $\rho$ has a property of type $\mathrm{(A)}$ if and only if
  $\rho\otimes\theta$ has the same property.

  Moreover, if $\rho$ and $\rho \otimes \theta$ are polynomially
  strongly admissible then $\zeta_\rho$ and
  $\zeta_{\rho \otimes \theta}$ have the same abscissa of
  convergence: $\alpha(\rho) = \alpha(\rho\otimes\theta)$.
\end{pro}

\begin{proof}
  Note that each property of type $\mathrm{(A)}$ is inherited by
  sub-representations.  Denoting by $\theta^\vee$ the contragredient
  representation of~$\theta$ acting on
  $V_{\theta^\vee} = V_\theta^\vee$, we observe that the trivial
  representation $\one_G$ occurs in $\theta \otimes \theta^\vee$.
  Hence $\rho$ is a sub-representation of
  $\rho \otimes \theta \otimes \theta^\vee$.  Consequently, it
  suffices to show:
  \begin{enumerate}[\rm (i)]
  \item if $\rho$ has a property of type $\mathrm{(A)}$ then
    $\rho\otimes\theta$ has the same property;
  \item if $\rho$ and $\rho \otimes \theta$ are polynomially strongly
    admissible then $\alpha(\rho\otimes\theta) \le \alpha(\rho)$.
  \end{enumerate}
 
  For smooth irreducible representations $\sigma, \pi \in \Irr(G)$, we recall that
  \[
  \Hom_G(V_\pi \otimes_\C V_\theta,V_\sigma) \cong
  \Hom_G(V_\pi,\Hom(V_\theta,V_\sigma)) \cong \Hom_G(V_\pi,,
  V_{\theta^\vee} \otimes_\C V_\sigma)
  \]
  as $\C$-vector spaces.  Taking dimensions, we deduce from
  Schur's lemma that the multiplicity of $\sigma$ in
  $\pi \otimes \theta$ and the multiplicity of $\pi$ in
  $\sigma \otimes \theta^\vee$ are the same:
  \begin{equation}\label{eq:multiplicityTwist}
    m(\sigma,\pi\otimes\theta) 
    = m(\pi, \sigma\otimes\theta^\vee). 
  \end{equation}
 
  Fix $\sigma \in \Irr(G)$ and, for $\pi \in \Irr(G)$, write
  $\pi \mid \sigma \otimes \theta^\vee$ to indicate that $\pi$ is a
  constituent of $\sigma \otimes \theta^\vee$, equivalently, that
  $m(\pi,\sigma \otimes \theta^\vee) > 0$ .  If $\rho$ is admissible,
  then \eqref{equ:rho-decomp} and \eqref{eq:multiplicityTwist} yield:
  \begin{equation} \label{equ:m-sigma-rhotau}
    m(\sigma,\rho\otimes\theta) = \sum_{\substack{\pi \in \Irr(G) \\
        \sigma \mid \pi \otimes \theta}} m(\pi, \rho) \,
    m(\sigma,\pi\otimes\theta) = \sum_{\substack{\pi \in \Irr(G) \\
        \pi \mid \sigma \otimes \theta^\vee}} m(\pi, \rho) \,
    m(\sigma,\pi\otimes\theta).
  \end{equation}
  Since $\sigma \otimes \theta^\vee$ and all $\pi \otimes \theta$ are
  finite-dimensional, this shows: if $\rho$ is admissible
  (respectively strongly admissible) then $\rho \otimes \theta$ is
  admissible (respectively strongly admissible).

  For all $\sigma, \pi \in \Irr(G)$ with
  $\sigma \mid \pi \otimes \theta$ we infer from
  \eqref{eq:multiplicityTwist} that
 \begin{equation} \label{equ:dim-inequ}
   \dim \sigma \le (\dim \pi) (\dim \theta) \qquad \text{and} \qquad
   \dim \pi \le (\dim \sigma) (\dim \theta). 
 \end{equation}
 If $\rho$ has multiplicities bounded by~$M$, then
 \eqref{equ:m-sigma-rhotau}, \eqref{eq:multiplicityTwist}, and the
 first inequality in \eqref{equ:dim-inequ} yield
  \begin{align*}
    m(\sigma,\rho\otimes\theta)
    & = \sum_{\substack{\pi \in \Irr(G) \\ \pi \mid \sigma \otimes
    \theta^\vee}} m(\pi, \rho) \, m(\sigma,\pi\otimes\theta) \\
    & \le M \sum_{\substack{\pi \in \Irr(G) \\ \pi \mid \sigma \otimes
    \theta^\vee}} m(\pi,\sigma \otimes \theta^\vee) (\dim
    \pi) \frac{\dim \theta}{\dim \sigma} \\ 
    & \le M (\dim \theta)^2,
  \end{align*}
  so that $\rho \otimes \theta$ has multiplicities bounded
  by~$M (\dim \theta)^2$.

  Finally, suppose that $\rho$ is polynomially strongly admissible, so
  that $\zeta_\rho(s)$ converges absolutely for
  $\mathrm{Re}(s) > \alpha(\rho)$.  Without loss of generality we may
  assume that $\rho$ is infinite-dimensional and consequently
  $\alpha(\rho) \ge 0$.  For all real $s > \alpha(\rho)$, we use
  \eqref{equ:m-sigma-rhotau} and the second inequality in
  \eqref{equ:dim-inequ} to obtain
 \begin{align*}
   \zeta_{\rho\otimes\theta}(s) & = \sum_{\sigma \in \Irr(G)}
                                  m(\sigma,\rho \otimes \theta)
                                  (\dim \sigma)^{-s} \\
                                & = \sum_{\pi\in \Irr(G)} m(\pi,\rho)
                                  \sum_{\substack{\sigma \in \Irr(G)
   \\ \sigma \mid \pi \otimes  \theta}} m(\sigma,\pi\otimes\theta)
   (\dim \sigma)^{-s} \\
                                & \le \sum_{\pi\in \Irr(G)}
                                  m(\pi,\rho) \sum_{\substack{\sigma \in \Irr(G)
   \\ \sigma \mid \pi \otimes  \theta}}
                                  m(\sigma,\pi\otimes\theta) (\dim \sigma)
                                  \frac{(\dim \theta)^{s+1}}{(\dim \pi)^{s+1}}\\ 
                                & = \sum_{\pi\in \Irr(G)} m(\pi,\rho)
                                  \frac{(\dim \theta)^{s+2}}{(\dim \pi)^{s}} \\
                                & = (\dim \theta)^{s+2} \zeta_\rho(s).
 \end{align*}
 Consequently, $\rho\otimes\theta$ is polynomially admissible and
 $\alpha(\rho\otimes\theta) \le \alpha(\rho)$.
\end{proof}

\begin{cor}
  Let $\rho$ be a smooth representation of a profinite group~$G$.  The
  properties of type $\mathrm{(A)}$ that hold for $\rho$ are common
  properties of the twist-similarity class of~$\rho$.  If $\rho$ is
  polynomially strongly admissible, then also the abscissa of convergence
  $\alpha(\rho)$ is an invariant of the twist-similarity class
  of~$\rho$.
\end{cor}

%%%

\subsection{Induction and restriction of smooth representations}
Let $H \le_\mathrm{c} G$ be a closed subgroup of a profinite
group~$G$.  Clearly, the restriction functor $\Res_H^G$ preserves
twist-similarity classes.  Hence, for any smooth representation $\rho$
of~$G$, the properties of type $\mathrm{(A)}$ of $\Res_H^G(\rho)$ only
depend on the twist-similarity class of~$\rho$.  The next proposition
shows that this conclusion also holds for the induction functor.

\begin{pro}\label{pro:InductionAndTwistSimilarity}
  Let $H \le_\mathrm{c} G$ be a closed subgroup of the profinite group
  $G$, and let $\sigma, \sigma'$ be twist-similar non-zero smooth
  representations of $H$.  If $\Ind_H^G(\sigma)$ has a property of
  type $\mathrm{(A)}$ then $\Ind_H^G(\sigma')$ has the same property.
  
  Furthermore, if $\Ind_H^G(\sigma)$ is polynomially strongly
  admissible then the abscissa of convergence
  $\alpha(\Ind_H^G(\sigma))$ depends on $\sigma$ only up to twist-similarity.
\end{pro}

\begin{proof}
  Properties of type $\mathrm{(A)}$ are inherited by
  sub-representations.  Hence, by an argument similar to the one
  starting the proof of Proposition~\ref{prop:TensorInvariance}, it
  suffices to consider the special cases (i)
  $\sigma' = \sigma \otimes \tau$ and (ii)
  $\sigma' = \sigma \oplus \tau$ for a finite-dimensional smooth
  representation $\tau$ of~$H$; furthermore, it is enough to show
  that, if $\Ind_H^G(\sigma)$ is polynomially strongly admissible in
  these cases, then
  $\alpha(\Ind_H^G(\sigma')) \le \alpha(\Ind_H^G(\sigma))$.

  First suppose that $\sigma' = \sigma \otimes \tau$.  We choose a
  finite-dimensional smooth representation $\theta$ of $G$
  such that $\tau$ injects into the restriction
  $\Res_H^G(\theta)$.  Then
  $\Ind_H^G(\sigma') = \Ind_H^G(\sigma \otimes \tau)$ injects into
  $\Ind_H^G(\sigma \otimes \Res_H^G(\theta)) \cong
  \Ind_H^G(\sigma) \otimes \theta$,
  and Proposition~\ref{prop:TensorInvariance} implies that, if
  $\Ind_H^G(\sigma)$ has a property of type $\mathrm{(A)}$, then
  $\Ind_H^G(\sigma \otimes \tau)$ has the same property.  Furthermore,
  if $\Ind_H^G(\sigma)$ is polynomially strongly admissible then
  $\alpha(\Ind_H^G(\sigma')) \le \alpha(\Ind_H^G(\sigma) \otimes
  \theta) = \alpha(\Ind_H^G(\sigma))$.

  Now suppose that $\sigma' = \sigma \oplus \tau$.  Choose a smooth
  irreducible constituent $\pi$ of~$\sigma$ and observe that $\tau$ is
  a sub-representation of $\sigma \otimes \pi^\vee \otimes \tau$.  If
  $\Ind_H^G(\sigma)$ has a property of type $\mathrm{(A)}$ then, by
  the argument above, the same property holds for
  $\Ind_H^G(\sigma \otimes \pi^\vee \otimes \tau)$ and hence for the
  sub-representation~$\Ind_H^G(\tau)$.  Furthermore, if
  $\Ind_H^G(\sigma)$ is polynomially strongly admissible then we
  obtain
  \begin{multline*}
    \alpha(\Ind_H^G(\sigma')) = \alpha(\Ind_H^G(\sigma) \oplus
    \Ind_H^G(\tau)) = \max \{ \alpha(\Ind_H^G(\sigma)),
    \alpha(\Ind_H^G(\tau)) \} \\
    \le \max \{ \alpha(\Ind_H^G(\sigma)), \alpha(\Ind_H^G(\sigma
    \otimes \pi^\vee \otimes \tau)) \} = \alpha(\Ind_H^G(\sigma)). \qedhere
  \end{multline*}
\end{proof}

The following result generalises~\cite[Cor.~2.3]{LuMa14}
and~\cite[Cor.~4.5]{LaLu08}.

\begin{pro}\label{pro:IndResFiniteIndex}
  Let $H \le_\mathrm{o} G$ be an open subgroup of the profinite
  group~$G$.  The functors $\Ind_H^G$ and $\Res_H^G$ preserve all the
  properties of type $\mathrm{(A)}$.  In addition, the functors
  preserve the abscissa of convergence for polynomially strongly
  admissible representations.
\end{pro}

\begin{proof}
  We only discuss the polynomially strongly admissible case; the other
  properties follow from similar arguments based on Frobenius
  reciprocity.
   
  Let $\sigma \in \Irr(H)$ and $\pi \in \Irr(G)$ be smooth irreducible
  representations of $H$ and $G$ respectively.  If $\pi$ occurs in
  $\Ind_H^G(\sigma)$, or equivalently $\sigma$ occurs in
  $\Res_H^G(\pi)$, then
   \begin{equation*}
       \dim \sigma \le \dim \pi  \le  \lvert G : H \rvert  \dim \sigma.
   \end{equation*}
   For all $s \in \R_{\ge 0}$ this yields the inequalities
   \begin{align}
     \frac{1}{(\lvert G:H \rvert \dim \sigma)^{s}}
     \le & \sum_{\tilde\pi \in \Irr(G) }
           \frac{m(\tilde\pi,\Ind_H^G(\sigma))}{(\dim \tilde\pi)^s}  
           \le  \frac{\lvert G : H \rvert}{(\dim \sigma)^{s}},  \label{equ:inequ1}\\
     \frac{1}{(\dim \pi)^{s}}
     \le & \sum_{\tilde\sigma \in \Irr(H) }
           \frac{m(\tilde\sigma,\Res_H^G(\pi))}{(\dim \tilde\sigma)^s} 
           \le  \frac{\lvert G : H \rvert^{s+1}}{(\dim \pi)^{s}} \label{equ:inequ2}
   \end{align}

   Let $\theta$ be a polynomially strongly admissible
   representation of~$H$.  If $\theta$ is finite-dimensional so is
   $\Ind_H^G(\theta)$, and there is nothing further to show.  Suppose
   that $\theta$ has infinite dimension so that
   $\alpha(\theta) \ge 0$.  Using \eqref{equ:inequ1}, we deduce that
   \begin{equation*}
     \lvert G : H \vert^{-s} \, \zeta_\theta(s) \le
     \zeta_{\Ind_H^G(\theta)}(s) \le \lvert G : H \rvert \,
     \zeta_\theta(s) \qquad \text{ for all $s \in \R_{\ge 0}$ }
   \end{equation*}
   so that $\alpha(\theta) = \alpha(\Ind_H^G(\theta))$.  

   Similarly, the claim for the restriction functor follows
   from~\eqref{equ:inequ2}.
\end{proof}

%%%

\subsection{Strong admissibility of induced representations}
Let $H \le_\mathrm{c} G$ be a closed subgroup of a profinite
group~$G$.  As recorded in Proposition~\ref{pro:admi-pres-ind}, the
functor $\Ind_H^G$ preserves admissibility.  However, strong
admissibility is in general not preserved by induction.  For example,
the regular representation of a finitely generated profinite group $G$
is strongly admissible if and only if $G$ is FAb.  The latter means
that every open subgroup $K \le_\mathrm{o} G$ has finite
abelianisation; see Section~\ref{sec:basic-concepts}.  We introduce a
relative FAb-condition to deal with induced representations in
general.

\begin{dfn}
  We say that $G$ is \emph{FAb relative to} $H$ if for every open
  subgroup $K \le_\mathrm{o} G$ the abelian quotient  $K/(H\cap K) [K,K]$
  is finite.
\end{dfn}

\begin{rmk}
  \noindent (a) Note that $G$ is FAb if and only if it is FAb relative
  to the trivial subgroup.  In general, $G$ is FAb relative to a
  closed normal subgroup $N \trianglelefteq_\mathrm{c} G$ exactly if the group
  $G/N$ is FAb.

  \noindent(b) The group $G$ is FAb relative to $H$ if and only if
  for every open normal subgroup $K \trianglelefteq_\mathrm{o} G$ the
  index of $H[K,K]$ in $G$ is finite. 

  \noindent (c) Suppose that $H_1 \le_\mathrm{c} H_2 \le_\mathrm{c} G$
  are closed subgroups of~$G$.  Obviously, if $G$ is FAb relative to
  $H_1$ then $G$ is FAb relative to~$H_2$.  The converse holds if $H_1$
  is open in $H_2$, because $\lvert H_2[K,K] : H_1[K,K] \rvert \le
  \lvert H_2 : H_1 \rvert$ for every $K\trianglelefteq_\mathrm{o} G$.
\end{rmk}

We restate and prove Theorem~\ref{thm:A} from the introduction.

\begin{thm} \label{thm:rel-FAb-equiv}
  Let $H \le_\mathrm{c} G$ be a closed subgroup of a finitely
  generated profinite group $G$.  The following statements are
  equivalent.
  \begin{enumerate}[\rm (a)]
  \item\label{enu:Fab-rel} The group $G$ is FAb relative to $H$.
  \item\label{enu:ind-pres} The functor $\Ind_H^G$ preserves strong
    admissibility.
  \item\label{enu:ind-pres-1} The induced representation
    $\Ind_{H}^G(\one_H)$ is strongly admissible.
  \end{enumerate}
\end{thm}

\begin{proof}
  In order to prove that \eqref{enu:ind-pres-1}
  implies~\eqref{enu:ind-pres}, we suppose that $\Ind_{H}^G(\one_H)$
  is strongly admissible.  All finite-dimensional smooth
  representations of $H$ are twist-similar to one another, thus
  Proposition~\ref{pro:InductionAndTwistSimilarity} shows that
  $\Ind_{H}^G(\sigma)$ is strongly admissible for every
  finite-dimensional smooth representation $\sigma$ of~$H$.  Let
  $\theta$ be any strongly admissible representation of $H$.  For
  every $d\in \N$ we obtain
   \begin{equation*}
     R_d(\Ind_H^G(\theta)) = \sum_{\sigma \in \Irr(H)}
     m(\sigma,\theta) R_d(\Ind_H^G(\sigma)) < \infty, 
   \end{equation*}
   because $\theta$ and $\Ind_H^G(\sigma)$ are strongly admissible
   and, furthermore, $\dim(\sigma) >d$ implies
   $R_d(\Ind_H^G(\sigma)) = 0$.
  
   We prove that \eqref{enu:ind-pres} implies~\eqref{enu:Fab-rel} by
   contraposition.  Suppose that $G$ is not FAb relative to $H$.  Take
   an open normal $K\trianglelefteq_\mathrm{o} G$ such that for
   $L = (H\cap K) [K,K]$ the abelian quotient $K / L$ is infinite. The
   representation $\Ind_L^K(\one_L)$ is infinite-dimensional and
   decomposes into $1$-dimensional representations of $K$, each
   occurring with multiplicity~$1$.  In particular, using the
   embedding
   $\Ind_L^K(\one_L) \hookrightarrow \Ind_{H\cap K}^K(\one_{H\cap
     K})$,
   we conclude that $\Ind_{H\cap K}^K(\one_{H\cap K})$ contains an
   infinite number of distinct $1$-dimensional representations.  For
   the finite-dimensional representation
   $\sigma = \Ind_{K\cap H}^H(\one_{K \cap H})$, we conclude that
   $\Ind_{H}^G(\sigma) = \Ind_{H\cap K}^G(\one_{H\cap K})$ is not
   strongly admissible as it contains infinitely many irreducible
   constituents of dimension at most~$\lvert G:K \rvert$.

   Finally, we prove that \eqref{enu:Fab-rel}
   implies~\eqref{enu:ind-pres-1}, again by contraposition.  Suppose
   that $\Ind_{H}^G(\one_H)$ is not strongly admissible.  We find a
   positive integer $d$ and an infinite sequence
   $(\pi_i,V_i)_{i=1}^\infty$ of distinct $d$-dimensional smooth
   irreducible representations of $G$ which occur in
   $\Ind_H^G(\one_H)$.  For every $i \in \mathbb{N}$ there is a
   non-zero vector $v_i$ in the space $V_{i}$ which is fixed by $H$.
   Setting $N_i = \mathrm{ker}(\pi_i)$, we observe that $G/N_i$ is a
   finite subgroup of $\GL_d(\C)$.  By a classical theorem of Jordan
   (see \cite[(36.13)]{CuRe88} or
   \cite[Thm.~5.7]{Di71}) there exist $m = m(d) \in \mathbb{N}$
   and open normal subgroups $A_i \trianglelefteq_\mathrm{o} G$ of
   index at most $m$ so that $N_i \subseteq A_i$ and $A_i/N_i$ is
   abelian for all $i \in \mathbb{N}$.  As $G$ is finitely generated,
   it contains only a finite number of open subgroups of index at
   most~$m$.  Hence we find $A \trianglelefteq_\mathrm{o} G$ such that
   $I = \{ i \in \mathbb{N} \mid A = A_i \}$ is infinite.  For each
   $i\in I$ the stabiliser of $v_i$ in $G$ contains the group
   $H[A,A]$.  Therefore the distinct irreducible representations
   $(\pi_i)_{i \in I}$ all occur in the induced representation
   $\Ind_{H[A,A]}^G(\one_{H[A,A]})$.  In particular, this
   representation is not finite-dimensional and the index of
   $\lvert G : H[A,A] \rvert$ is infinite. We conclude that $G$ is not
   FAb relative to~$H$.
\end{proof}

%%%

\subsection{Polynomially strong admissibility and compact $p$-adic Lie groups}
A profinite group $G$ has polynomial representation growth, as defined
in~\cite{LaLu08}, if and only if the regular representation of $G$ is
polynomially strongly admissible;
see~Example~\ref{exa:conn-repr-zeta-functions}.  As yet no simple
characterisation of profinite groups of polynomial representation
growth is known, not even at a conjectural level.  The regular
representation is obtained by inducing the trivial representation from
the trivial subgroup; on that account we formulate the following more
general problem.

\begin{prob}
  Under what conditions on $H \le_\mathrm{c} G$ does the induction
  functor $\Ind_H^G$ preserve polynomially strong admissibility?
\end{prob}

The next result, which is Proposition~\ref{proABC:B} in the
introduction, and its proof generalise \cite[Prop.~2.7]{LuMa14}.  Let
$p$ be a prime.  We refer to \cite{DidSMaSe99} for the relevant
structure theory of compact $p$-adic analytic groups.

\begin{pro} \label{pro:ind-finite-dim-adm-rel-FAb}
  Let $H \le_\mathrm{c} G$ be a closed subgroup of a compact
  $p$-adic Lie group~$G$.  If $G$ is FAb relative to~$H$ then
  $\Ind_H^G(\sigma)$ is polynomially strongly admissible for every
  finite-dimensional smooth representation $\sigma$ of $H$.
\end{pro}

\begin{proof}
  We may assume that $H$ and $G$ are uniformly powerful pro-$p$ groups
  and that $\sigma = \one_H$ is the trivial representation.  Indeed,
  there are open uniformly powerful pro-$p$ subgroups
  $H^* \le_\mathrm{o} H$ and $G^* \le_\mathrm{o} G$ such that
  $H^* \le G^*$.  By
  Proposition~\ref{pro:InductionAndTwistSimilarity}, it suffices to
  consider $\sigma = \Ind_{H^*}^H(\one_{H^*})$, thus we may assume
  that $H =H^*$ and $\sigma = \one_H$.  By
  Proposition~\ref{pro:IndResFiniteIndex}, we may assume further that
  $G = G^*$.  Clearly, the property of being relatively FAb is also
  inherited.
 
  Let $\fg = \log(G)$ and $\fh = \log(H)$ denote the powerful
  $\Z_p$-Lie lattices associated to $G$ and $H$.  The lower $p$-series
  of $G$, given by $G_1 = G$ and $G_n = (G_{n-1})^p [G_{n-1},G]$ for
  $n \ge 2$, satisfies $G_n = G^{p^{n-1}} = \exp(p^{n-1} \fg)$ for all
  $n \in \mathbb{N}$.  Since $G$ is FAb relative to~$H$, we find
  $r \in \N_0$ such that $G_{r+1} \subseteq H[G,G]$.  Observe that
  $[G,G]$ is powerfully embedded in~$G$, hence $H[G,G]$ is uniformly
  powerful and $\log(H[G,G]) = \fh + [\fg,\fg]$.  This yields
  $p^r \fg \subseteq \fh + [\fg, \fg]$.  For $n \in \N$ this implies
  that $p^{2n+r} \fg \subseteq \fh + [p^n\fg,p^n\fg]$, equivalently
  $G_{2n+r+1} \subseteq H [G_{n+1},G_{n+1}]$, using again that
  $[G_{n+1},G_{n+1}]$ is powerfully embedded in~$G$.
 
  For $n\in \N$ let $\psi_n$ denote the $G$-sub-representation of
  $\Ind_H^G(\one_H)$ spanned by all irreducible sub-representations of
  dimension at most $p^n$.  Let $(\eta,V_\eta) \in \Irr(G)$ be an
  irreducible constituent of~$\psi_n$.  As $G$ is monomial, there are
  an open subgroup $K \le_\mathrm{o} G$ and a linear character
  $\chi\colon K \to \C^\times$ such that $\eta = \Ind_K^G(\chi)$.  Note
  that $\lvert G : K \rvert = \dim \eta \le p^n$ and so
  $G_{n+1} = G^{p^{n}} \subseteq K$.
 
  The commutator group $[G_{n+1},G_{n+1}]$ lies in the kernel of
  $\eta$ and so every $H$-fixed vector in $V_\eta$ is also fixed by
  $[G_{n+1},G_{n+1}]$.  We conclude that $\eta$ occurs in
  $\Ind_{H [G_{n+1},G_{n+1}]}^G(\one_H)$ with the same multiplicity as
  in $\psi_n$.  Denoting by
  $\iso_\fg(\fh) = \fg \cap (\Q_p \otimes_{\Z_p} \fh)$ the isolator of
  $\fh$ in~$\fg$ (compare~\cite[\S 3]{GoKl09}), we obtain in total
 \begin{multline*}
   R_{p^n}(\Ind_H^G(\sigma)) \le \dim \psi_n  \le \lvert G : H
   G_{2n+r+1} \rvert \\ = \lvert \fg : \fh + p^{2n+r} \fg \rvert \le
   C p^{2(\dim(G)-\dim(H)) n}, 
 \end{multline*}
 where $C= p^{r(\dim(G) - \dim(H))} \lvert \iso_{\fg}(\fh) :
 \fh \rvert \in \mathbb{R}_{> 0}$ is independent of~$n$.
\end{proof}

%%%%%

\subsection{Rationality}
This section contains some basic observations concerning the
rationality of zeta functions associated to induced representations.
Following~\cite{GoJaKl14}, we say that a Dirichlet series $\zeta(s)$
is \emph{rational with respect to a prime $p$} if it has non-empty
domain of convergence and admits a meromorphic continuation of the
form
\begin{equation*} 
  \zeta(s) = \sum^r_{i=1} m_i^{-s}
  F_i(p^{-s}), 
\end{equation*} 
for finitely many suitable positive integers $m_1,\dots, m_r$ and
rational functions $F_1,\dots, F_r \in \Q(X)$.  In
Section~\ref{sec:rationality-etc} we show that the zeta functions of
induced representations of potent pro-$p$ groups are rational with
respect to~$p$.  Further examples for rationality appear in
Proposition~\ref{prop:formulaInductionTrees}.  We do not know whether
this is a general phenomenon for compact $p$-adic Lie groups; it
remains an open problem to generalise the results of
Jaikin-Zapirain~\cite{Ja06}.

\begin{pro}\label{pro:vanishing-result}
  Let $G$ be virtually a pro-$p$ group and let $H \leq_\mathrm{c} G$
  be a closed subgroup of infinite index.  Let $\theta$ be a
  finite-dimensional smooth representation of $H$ and suppose that the
  induced representation $\rho = \Ind_H^G(\theta)$ is polynomially
  strongly admissible.  If the zeta function $\zeta_\rho$ is rational
  with respect to $p$, then
  \[
  \zeta_\rho(-1) = 0.
  \]
\end{pro}

\begin{proof}
  Suppose that $\zeta_\rho$ is rational with respect to $p$, and
  recall that, for $n \in \N$, we denote by $r_n(\rho)$ the number of
  $n$-dimensional irreducible constituents of~$\rho$.  By the general
  argument used~\cite[Proof of Theorem~1]{GoJaKl14} it is sufficient
  to show that the series $\sum_{n=1}^\infty r_n(\rho) n$ converges to
  $0$ in~$\Z_p$.

  Since $G$ is virtually pro-$p$, there are finitely many positive
  integers $m_1,\dots, m_r$ such that every irreducible representation
  of $G$ has dimension $m_ip^k$ for some $i \in \{1,\dots,r\}$ and
  some $k\in \N_0$.  Indeed, every irreducible representation factors
  through some finite continuous quotient $G/N$ and the dimension
  divides the order of $G/N$; see~\cite[Theorem~3.11]{Is94}.  We
  deduce that the series
  \begin{equation}\label{eq:series-p-adic}
    \sum_{n=1}^\infty r_n(\rho) n  = \sum_{\pi \in \Irr(G)} m(\pi,\rho) \dim \pi
  \end{equation}  
  converges in~$\Z_p$; in addition, an elementary argument from
  $p$-adic analysis implies the limit is independent of the order of
  summation.

  The kernel $K \trianglelefteq_\mathrm{c} G$ of $\rho$ is the
  intersection of the open kernels of the countably many irreducible
  constituents of $\rho$.  In particular, there is a decreasing
  sequence $(N_i)_{i\in\N}$ of open normal subgroups
  $N_i \trianglelefteq_\mathrm{o} G$ with
  $\bigcap_{i \in \N} N_i = K$.  From the definition of
  $\Ind_H^G(\theta)$ we observe that $K \subseteq H$ and thus
  $H = H (\bigcap_{i\in\N} N_i) = \bigcap_{i \in \N} H N_i$.  By
  shrinking the subgroups $N_i$ further, if necessary, we may assume
  that $\theta$ factors through each $H\cap N_i$.
  
  Consider the series~\eqref{eq:series-p-adic} and its partial sums
  \[
  S_i = \sum_{\substack{\pi \in \Irr(G)\\ N_i \subseteq \ker(\pi)}}
  m(\pi,\rho) \dim\pi, \qquad i \in \N,
  \]
  taken over the irreducible representations of $G$ which factor
  through $G/N_i$.  The right-hand side equals the dimension of the
  space $V_\rho^{\, N_i}$ of $N_i$-invariants which, as a
  representation of $G/N_i$, is canonically isomorphic to
  $\Ind_{HN_i/ N_i}^{G/N_i}(\theta)$.  In particular, we see that
  $S_i = \dim(\theta) \lvert G/N_i : HN_i/N_i \rvert = \dim(\theta)
  \lvert G:HN_i \rvert$.
  Since $\lvert G:H \rvert$ is infinite and $G$ is virtually pro-$p$,
  we conclude that
  $\sum_{n=1}^\infty r_n(\rho) n = \lim_{i\to \infty} S_i = 0$ in the
  $p$-adic integers~$\Z_p$.
\end{proof}

\begin{exa} \label{exa:no-rationality} While induction from and
  restriction to an open subgroup preserve properties of type
  $\mathrm{(A)}$, these functors can substantially alter rationality
  properties of the associated zeta functions.  Therefore our
  rationality results for zeta functions of induced representations of
  potent pro-$p$ groups in Section~\ref{sec:rationality-etc} (compare
  Remark~\ref{rmk:not-global}) do not automatically extend to general
  compact $p$-adic Lie groups.

  To illustrate the underlying issue we construct, for every prime $p$,
  \emph{a compact $p$-adic Lie group $G$, an open uniformly powerful
    pro-$p$ subgroup $H \leq_o G$ of index $2$, and a polynomially
    strongly admissible representation $\rho$ of $G$ such that
    $\zeta_{\Res_H^G(\rho)}(s)$ is a rational function in
    $\mathbb{Q}(p^{-s})$, whereas $\zeta_\rho(s)$ cannot be expressed
    as a rational function in
    $\mathbb{Q}(\{n^{-s} \mid n \in
    \mathbb{N}\})$}.

  Let $U$ be a uniformly powerful pro-$p$ group which has an
  irreducible representation $\eta_n$ of degree $p^n$ for every
  integer $n\geq 0$.  For instance, take $U$ to be the $3$-dimensional
  $p$-adic Heisenberg group.  Define $H = U \times \mathbb{Z}_p$ and
  $G = H \rtimes C_2$, where the generator $\tau$ of $C_2$ acts as
  the identity on $U$ and by inversion on~$\mathbb{Z}_p$.  Fix a
  $1$-dimensional irreducible representation $\chi$ of $\mathbb{Z}_p$
  such that $\chi(-t) \neq \chi(t)$ for some $t\in\mathbb{Z}_p$.
  
  Consider the irreducible $p^n$-dimensional representations
  $\pi_n = \eta_n \otimes \one$ and $\sigma_n = \eta_n \otimes \chi$
  of $H$.  Observe that the representation $\pi_n$ is
  $\tau$-invariant, thus $\pi_n$ extends to an irreducible
  representation $\alpha_n$ of $G$.  However, the construction gives
  $\sigma_n \neq \sigma_n^\tau$ and thus yields an irreducible
  representation $\beta_n$ of $G$ of degree $2p^n$ such that
  $\Res_H^G(\beta_n) = \sigma_n \oplus \sigma^{\tau}_n$.
  
  Choose a sequence $(m_n)_{n\in \mathbb{N}} \in \{0,1\}^{\mathbb{N}}$ which is not eventually
  periodic.  The smooth representation
  \begin{equation*}
    \rho = \bigoplus_{n=0}^\infty 2(1-m_n) \alpha_n \oplus
    \bigoplus_{n=0}^\infty m_n \beta_n  
  \end{equation*}
  of $G$ is polynomially strongly admissible.
  By construction, we have
  \begin{equation*}
    \Res_H^G(\rho) = \bigoplus_{n=0}^\infty 2(1-m_n) \pi_n \oplus
    \bigoplus_{n=0}^\infty m_n \sigma_n \oplus \bigoplus_{n=0}^\infty
    m_n \sigma^\tau_n, 
  \end{equation*}
   and a short calculation yields
   \begin{equation*}
      \zeta_{\Res_H^G(\rho)}(s) = \sum_{n=0}^\infty 2 p^{-ns} = \frac{2}{1-p^{-s}}.
   \end{equation*}
   However, the Dirichlet series defining the zeta function of the
   representation $\rho$ converges at $1$ and evaluates to
   \begin{equation*}
     \zeta_\rho(1) = \sum_{n=0}^\infty (2-2m_n)p^{-n} + m_n
     (2p^{n})^{-1} = \frac{2}{1-p^{-1}} - \frac{3}{2}
     \sum_{n=0}^\infty m_n p^{-n}. 
   \end{equation*}
   The sequence $(m_n)_{n\in\mathbb{N}}$ was chosen so that the number
   $\sum_{n=0}^\infty m_n p^{-n}$ is irrational, and we conclude that
   $\zeta_\rho(s)$ cannot be expressed by as a rational function in
   $\{n^{-s} \mid n \in \mathbb{N}\}$ with coefficients in
   $\mathbb{Q}$.
\end{exa}

\begin{rmk}
  The construction in Example~\ref{exa:no-rationality} can easily be
  adapted to produce, for each odd prime~$p$, a representation of a
  $p$-adic analytic pro-$p$ group with similar properties.  Indeed,
  choose $U$ as before and consider $G = H \rtimes C_p$, where
  $H = U \times \mathbb{Z}_p[z] \cong U \times \mathbb{Z}_p^{\, p-1}$
  for a primitive $p$th root of unity $z$ and the generator $\tau$ of
  $C_p$ acts as the identity on $U$ and by multiplication by $z$
  on~$\mathbb{Z}_p[z]$.  In a similar fashion as before, one can
  define a polynomially strongly admissible representation $\rho$ of
  $G$, encoding a sequence
  $(m_n)_{n \in \mathbb{N}} \in \{0,1\}^\mathbb{N}$ which is not
  eventually periodic.  The representation $\rho$ has the same
  features as its counterpart in Example~\ref{exa:no-rationality}, by
  the same reasoning.
\end{rmk}

%%%%%

\section{Induced representations of potent
  pro-$p$ groups via~the~orbit~method} \label{sec:ind-uniform-groups}

Throughout this section $p$ denotes a prime.  Every $p$-adic Lie group
contains a compact open pro-$p$ group that is uniformly powerful; we
refer to~\cite{DidSMaSe99} for the general theory of $p$-adic analytic
groups.  A pro-$p$ group $G$ is called \emph{potent} if
$[G,G] \subseteq G^4$ for $p=2$ and $\gamma_{p-1}(G) \subseteq G^p$
for $p>2$; there is an analogous definition for $\Z_p$-Lie lattices.
Finitely generated torsion-free potent pro-$p$ groups are a natural
generalisation of uniformly powerful pro-$p$ groups, and for all such
groups the Kirillov orbit method provides a powerful tool for handling
the characters of smooth irreducible representations;
see~\cite{GoJa04,Go09}.

In this section we describe the representation zeta functions of
induced representations for finitely generated torsion-free potent
pro-$p$ groups generalising the approach used, for instance, in
\cite{AvKlOnVo12,AvKlOnVo13,AvKlOnVo16a}.  There is a bijective
correspondence between isomorphism classes of smooth irreducible
representations of a profinite group $G$ and the corresponding
irreducible complex characters.  We will interpret elements of
$\Irr(G)$ in a flexible way as isomorphism classes of representations
or characters, as befits the situation.
  
%%%

\subsection{Potent pro-$p$ groups} Let $G$ be a finitely generated
torsion-free potent pro-$p$ group. Then $G$ is saturable (in the sense
of Lazard) and we denote by $\fg = \log(G)$ the associated potent
$\Z_p$-Lie lattice; compare~\cite{Kl05,Go07}.  We make repeatedly use
of the following basic lemma.

\begin{lem} \label{lem:Haar-measure} Let $G$ be a finitely generated
  torsion-free potent pro-$p$ group with associated potent
  $\Z_p$-Lie lattice~$\fg$.  Then the logarithm map
  $\log \colon G \rightarrow \fg$ transforms the multiplicative Haar
  measure on $G$ to the additive Haar measure on~$\fg$.
\end{lem}

\begin{proof}
  It suffices to verify that the measures of cosets of open subgroups
  forming a base of neighbourhoods of $1$ in $G$ are preserved under
  the logarithm map.  This follows from the fact that the
  multiplicative cosets $xN$ of any open powerfully embedded normal 
  subgroup $N \trianglelefteq_\mathrm{o} G$ are mapped to the additive
  cosets $\log(x) + \fn$ of the associated $\Z_p$-Lie
  sublattice $\fn \le \fg$; see the argument in
  \cite[Cor.~6.38]{DidSMaSe99} or \cite[Lem.~4.4]{Go07}.
\end{proof}

The \emph{adjoint action} of $G$ on
$\fg$ is related to conjugation in $G$ via
\[ {}^g X = \log( g \exp(X) g^{-1}) \qquad \text{for $g \in G$ and $X
  \in \fg$.}
\]
We fix an isomorphism
\begin{equation}\label{equ:Qp/Zp-iso}
  \Q_p/\Z_p \rightarrow \mu_{p^\infty}(\C), \quad
  \overline{a} = a + \Z_p \mapsto \overline{a}_\C,
\end{equation}
for instance, by decreeing that
$\overline{p^{-m}}_\C = e^{2 \pi i p^{-m}}$ for $m \in \N$.  In this
way the Pontryagin dual of the additive compact group $\fg$ can be
realised as $\fg^\vee = \fg_{\Z_p}^\vee = \Hom(\fg,\Q_p/\Z_p)$.  We
denote the neutral element of $\fg^\vee$, i.e.\ the zero map, by~$0$.

The \emph{co-adjoint action} of $G$ on $\fg^\vee$ is given by
\[
({}^g \omega)(X) = \omega({}^{g^{-1}} X) \qquad \text{for $g \in G$,
  $\omega \in \fg^\vee$ and $X \in \fg$.}
\]
At the level of Lie lattices there is a corresponding \emph{co-adjoint
  action} of $\fg$ on $\fg^\vee$ given by
\[
(Y.\omega)(X) = \omega([X,Y]) \qquad \text{for $X, Y \in \fg$ and
  $\omega \in \fg^\vee$.}
\]

The \emph{Kirillov orbit method} for $p$-adic analytic pro-$p$ groups
yields a bijective correspondence
\[
G \backslash \fg^\vee \rightarrow \Irr(G), \quad G.\omega
\mapsto \chi_\omega
\]
between the collection
$G \backslash \fg^\vee = \{ G.\omega \mid \omega \in \fg^\vee \}$ of
co-adjoint orbits and the set $\Irr(G)$ of irreducible complex
characters of~$G$.  If $p$ is odd, the correspondence is canonical and
can be made explicit via the formula
\begin{equation}\label{eq:formulaKirillovOrbitMethod}
  \chi_\omega(x) = \lvert G.\omega \rvert^{-\nicefrac{1}{2}}
  \sum_{\widetilde{\omega} \in G.\omega} \widetilde{\omega}(\log(x))_\C
  \qquad \text{for $\omega \in \fg^\vee$ and $x \in G$,}
\end{equation}
where we make use of the isomorphism~\eqref{equ:Qp/Zp-iso};
see~\cite{Go09}.  If $p=2$, the
expression~\eqref{eq:formulaKirillovOrbitMethod} for $\chi_\omega(x)$
remains valid whenever $x$ lies in the open subgroup~$G^2$, but does
not hold in general; see~\cite[Thm.~2.12]{Ja06}.  For
$\omega \in \fg^\vee$ we write $\pi_\omega$ for a representation of
$G$ affording $\chi_\omega$ so that
\[
\dim \pi_\omega = \chi_\omega(1) = \lvert G.\omega \rvert^{\nicefrac{1}{2}} =
\lvert G : \Stab_G(\omega) \rvert^{\nicefrac{1}{2}}.
\]
There is a useful Lie-theoretic description of the stabiliser
$\Stab_G(\omega)$.  The \emph{stabiliser} of $\omega$ in
$\fg$ under the co-adjoint action is
\[
\stab_\fg(\omega) = \{ Y \in \fg \mid Y.\omega = 0 \},
\]
which can also be interpreted as the radical of an alternating
bilinear form associated to~$\omega$.  It is a fact that
$\log(\Stab_G(\omega)) = \stab_\fg(\omega)$ (see
\cite[Lem.~2.3]{Ja06}) and hence
$\lvert G : \Stab_G(\omega) \rvert = \lvert \fg : \stab_\fg(\omega)
\rvert$.

Now consider a (finitely generated torsion-free) potent subgroup
$H \le_\mathrm{c} G$ and its associated potent $\Z_p$-Lie
lattice~$\fh \le \fg$.  The inclusion map
$i_\fh^\fg \colon \fh \rightarrow \fg$ induces a surjective
restriction map $r_\fh^\fg \colon \fg^\vee \rightarrow \fh^\vee$.  For
$\eta \in \fh^\vee$ we consider the $r_\fh^\fg$-fibres over elements
of the co-adjoint orbit $H.\eta$.  For $\omega \in \fg^\vee$ we define
the \emph{orbit intersection number}
\[
\IN(\omega,\eta) = \lvert G.\omega \cap
(r_\fh^\fg)^{-1}(H.\eta) \rvert.
\]
With this terminology we obtain the following application of the
Kirillov orbit method to induced representations for all odd primes.

\begin{pro} \label{pro:multiplicities} Suppose that~$p>2$.  Let $G$ be
  a finitely generated torison-free potent pro-$p$ group, with
  associated $\Z_p$\nobreakdash-Lie lattice $\fg$, and
  $H \le_\mathrm{c} G$ a potent subgroup, with associated
  $\Z_p$\nobreakdash-Lie lattice $\fh$.  Let $\omega \in \fg^\vee$ and
  $\eta \in \fh^\vee$.  Then the multiplicity of $\pi_\omega$ in the
  induced representation $\Ind_H^G(\pi_\eta)$ is given by the formula
  \[
  m \left( \pi_\omega,\Ind_H^G(\pi_\eta) \right) =
  \frac{\IN(\omega,\eta)}{\lvert G.\omega \rvert^{\nicefrac{1}{2}} \, \lvert
    H.\eta \rvert^{\nicefrac{1}{2}}}.
  \]
\end{pro}

\begin{proof}
  The orthogonality relations for irreducible characters of $\fh$
  yield
  \[
  \IN(\omega,\eta) = \sum_{\widetilde{\omega} \in G.\omega}
  \sum_{\widetilde{\eta} \in H.\eta} \int_\fh \widetilde{\omega}(Y)_\C
  \, \overline{\widetilde{\eta}(Y)_\C} \,\di\mu_\fh(Y).
  \]
  Frobenius reciprocity and Lemma~\ref{lem:Haar-measure} thus give
  \begin{align*}
    m \left( \pi_\omega, \Ind_H^G(\pi_\eta) \right) & = \langle
    \chi_\omega, \Ind_H^G(\chi_\eta) \rangle \\
    & = \langle \Res_H^G(\chi_\omega), \chi_\eta \rangle \\
    & = \int_H \chi_\omega(y) \overline{\chi_\eta(y)}
    \,\di\mu_H(y) \\
    & = \lvert G.\omega \rvert^{-\nicefrac{1}{2}} \lvert H.\eta \rvert^{-\nicefrac{1}{2}}
    \sum_{\widetilde{\omega} \in G.\omega} \sum_{\widetilde{\eta} \in
      H.\eta} \int_\fh \widetilde{\omega}(Y)_\C \,
    \overline{\widetilde{\eta}(Y)_\C} \,\di\mu_\fh(Y) \\
    & = \lvert G.\omega \rvert^{-\nicefrac{1}{2}} \lvert H.\eta \rvert^{-\nicefrac{1}{2}}
    \IN(\omega,\eta). \qedhere
  \end{align*}
\end{proof}

We now specialise to the situation where we induce the trivial
representation~$\mathbb{1}_H$ from a subgroup $H$ to~$G$.  We stress
that the next result is valid also for $p=2$.

\begin{pro} \label{pro:zeta-algebraic-formula} Let $G$ be a finitely
  generated torsion-free potent pro-$p$ group, with associated
  $\Z_p$\nobreakdash-Lie lattice $\fg$, and $H \le_\mathrm{c} G$ a
  potent subgroup, with associated $\Z_p$\nobreakdash-Lie lattice
  $\fh$.  The zeta function of $\rho = \Ind_H^G(\mathbb{1}_H)$ is
  given by
  \[
  \zeta_\rho(s) = \sum_{\substack{\omega \in \fg^\vee \\
      r_\fh^\fg(\omega) = 0}} \left( \lvert
    \fg : \stab_\fg(\omega) \rvert^{\nicefrac{1}{2}}
  \right)^{-1-s}.
  \]
\end{pro}

\begin{proof}
  First suppose that $p$ is an odd prime.  Applying
  Proposition~\ref{pro:multiplicities} to $\eta = 0$, hence
  $\pi_\eta = \pi_0 = \mathbb{1}_H$ and $H.\eta = \{0\}$, this is a
  simple computation:
  \begin{equation}\label{eq:calculationOrbitMethod}
    \begin{split}   
      \zeta_\rho(s) & = \sum_{\omega \in \fg^\vee} \lvert G.\omega
      \rvert^{-1} \, m \left( \pi_\omega,\Ind_H^G(\pi_0)
      \right) \, \lvert G.\omega \rvert^{-s/2} \\
      & = \sum_{\omega \in \fg^\vee} \frac{\lvert G.\omega \cap
        (r_\fh^\fg)^{-1}(\{0\}) \rvert}{\lvert G.\omega \rvert} \,
      \left( \lvert G.\omega
        \rvert^{\nicefrac{1}{2}} \right)^{-1-s} \\
      & = \sum_{\substack{\omega \in \fg^\vee \\
          r_\fh^\fg(\omega) = 0}} \left( \lvert \fg :
        \stab_\fg(\omega) \rvert^{\nicefrac{1}{2}} \right)^{-1-s}.
    \end{split}     
 \end{equation}
  
  Now suppose that~$p=2$.  In this case the Kirillov orbit method
  involves a choice, which makes it difficult to establish a possible
  analogue of Proposition~\ref{pro:multiplicities}.  However, we
  obtain a canonical correspondence by considering suitable
  equivalence classes of irreducible characters.
  
  We say that $\theta, \widetilde{\theta} \in \Irr(G)$ are
  \emph{$G^2$-equivalent} if they restrict to the same character
  on~$G^2$, i.e.,
  $\theta_{\vert G^2} = \widetilde{\theta}_{\vert G^2}$.  Denote by
  $Z$ the group of linear characters of $G$ that factor through the
  elementary abelian quotient~$G/G^2$.  The group $Z$ acts on
  $\Irr(G)$ by multiplication and, as $G/G^2$ is abelian, the
  $G^2$-equivalence classes are exactly the $Z$-orbits in~$\Irr(G)$.
  For $\omega \in \fg^\vee$ let $\xi_\omega$ denote the sum of all
  $\theta \in \Irr(G)$ that are $G^2$-equivalent to $\chi_\omega$,
  i.e.\ the character associated to the $G^2$-equivalence class
  of~$\chi_\omega$, and let $\sigma_\omega$ denote a representation
  affording~$\xi_\omega$.
  
  For $\omega_1,\omega_2 \in \fg^\vee$ the characters
  $\chi_{\omega_1}$ and $\chi_{\omega_2}$ are $G^2$-equivalent if and
  only if there are $g\in G$ and $\tau \in Z$ such that
  $^g\omega_1 \, \tau = \omega_2$, equivalently if
  $\omega_1, \omega_2$ lie in the same orbit under the indicated
  action of the direct product $G \times Z$ on~$\fg^\vee$.  As
  explained in~\cite{Ja06}, the Kirillov orbit method yields a
  bijective correspondence
  \begin{equation*}
      (G\times Z).\omega \mapsto \xi_\omega
  \end{equation*}
  between the $(G\times Z)$-orbits in $\fg^\vee$ and the characters
  associated to $G^2$-equivalence classes in $\Irr(G)$.  Moreover, the
  formula
   \begin{equation*}
     \xi_\omega(x) =  \lvert G.\omega \rvert^{-\nicefrac{1}{2}}
     \sum_{\widetilde{\omega} \in (G\times Z).\omega} \widetilde{\omega}(\log(x))_\C
   \end{equation*}
   holds for all $x \in G$.  Indeed, both functions are $Z$-invariant
   and hence vanish on $G \smallsetminus G^2$; furthermore, by
   \eqref{eq:formulaKirillovOrbitMethod} the functions agree on~$G^2$.
   Since all characters in a $G^2$-equivalence class have the same
   degree, a slight modification of the argument
   in~\eqref{eq:calculationOrbitMethod} suffices to complete the
   proof: replace the terms $\lvert G.\omega \rvert^{-1}$,
   $\pi_\omega$, and
   $\frac{\lvert G.\omega \cap (r_\fh^\fg)^{-1}(\{0\}) \rvert}{\lvert
     G.\omega \rvert}$
   by $\lvert(G\times Z).\omega \rvert^{-1}$, $\sigma_\omega$, and
   $\frac{\lvert (G\times Z).\omega \cap (r_\fh^\fg)^{-1}(\{0\})
     \rvert}{\lvert (G\times Z).\omega \rvert}$ respectively.
\end{proof}

%%%

\subsection{$\fo$-Lie lattices} \label{sec:LieLattices} Let
$\fo$ be a compact discrete valuation ring of
characteristic~$0$, residue characteristic~$p$, and residue field
cardinality~$q$.  Fix a uniformiser $\pi$ so that the valuation ideal
of~$\fo$ takes the form $\Ip = \pi \fo$.
Let $\ff$ denote the fraction field of~$\fo$, a
finite extension of~$\Q_p$.  We denote by $e(\ff,\Q_p)$ the
absolute ramification index of~$\ff$.

Let $\fg$ be an $\fo$-Lie lattice.  We take interest in the family of
finitely generated torsion-free potent pro-$p$ groups $G$ that arise
as $G = \exp(\pi^r \fg)$ for all sufficiently large $r \in \N_0$.

The codifferent of $\ff$ over $\Q_p$ is the
fractional ideal
\[
\{ x \in \ff \mid \forall y \in \fo \colon
\mathrm{Tr}_{\ff \vert \Q_p}(xy) \in \Z_p \}.
\]
It lies in the kernel of the non-trivial character
\[
\ff \xrightarrow{\mathrm{Tr}_{\ff \vert \Q_p}} \Q_p
\rightarrow \Q_p / \Z_p \xrightarrow{\cong} \mu_{p^\infty}(\C)
\]
of the additive group $\ff$.  The different
$\mathfrak{D}_{\ff \vert \Q_p}$ is the inverse of the
codifferent and thus an ideal of $\fo$, say
$\Ip^\delta$, where $\delta = \delta(\ff) \in \N_0$
satisfies $\delta \ge e(\ff, \Q_p) -1$ with equality if and
only if $e(\ff, \Q_p)$ is prime to~$p$.  Using
$[\pi^{-\delta} \cdot] \colon \ff \rightarrow \ff$,
$a \mapsto \pi^{-\delta}a$, we obtain a map
\[
\Hom_\fo(\fg, \ff) \rightarrow
\Hom(\fg, \Q_p), \quad \psi \mapsto \mathrm{Tr}_{\ff
  \vert \Q_p} \circ [\pi^{-\delta} \cdot] \circ \psi
\]
that induces a non-canonical isomorphism
\[
\fg^\vee = \fg_\fo^\vee =
\Hom_\fo(\fg, \ff/\fo)
\xrightarrow{\cong} \Hom(\fg, \Q_p/\Z_p).
\]

Consider $\omega \in \fg^\vee$ and choose $w \in
\Hom_\fo(\fg, \ff)$ such that $\omega(X) =
w(X) + \fo$ for $X \in \fg$.  The alternating
$\fo$-bilinear form
\[
A_w \colon \fg \times \fg \rightarrow \ff,
\quad A_w(X,Y) = w([X,Y])
\]
induces the alternating $\fo$-bilinear form 
\[
A_\omega \colon \fg \times \fg \rightarrow
\ff/\fo, \quad A_\omega(X,Y) = \omega([X,Y]) =
A_w(X,Y) + \fo.
\]
The $\fo$-submodule $\stab_\fg(\omega)$ is equal to
the radical of $A_\omega$, which can be described in terms of $A_w$ as
\[
\rad(A_\omega) = \{ Y \in \fg \mid \forall X \in \fg
: A_w(X,Y) \in \fo \},
\]
and the index $\lvert \fg : \stab_\fg(\omega) \rvert$ can be expressed
in terms of the invariant factors of the $\fo$-submodule
$\rad(A_\omega)$ of $\fg$.  The latter are closely related to the
Pfaffians of the alternating $\fo$-bilinear form~$A_w$, which are
defined as follows.  Let $n = \dim_\fo \fg$.  For
$0 \le k \le \lfloor \nicefrac{n}{2} \rfloor$, the \emph{degree-$k$
  Pfaffian} of $A_w$ is the fractional ideal $\Pfaff_k(w)$ of $\ff$
generated by the elements
\[
\frac{1}{2^k k!} \sum_{\sigma \in \mathrm{Sym}(2k)}
\mathrm{sgn}(\sigma) \prod_{j=1}^k
A_w(Y_{(2j-1)\sigma},Y_{(2j)\sigma}) = \raisebox{1.2ex}{\text{``}} \sqrt{\det
  \big( (A_w(Y_i,Y_j)_{1 \le i,j \le 2k}) \big)}
\,\raisebox{1.2ex}{\text{''}},
\]
where $Y_1,\ldots,Y_{2k}$ run through spanning sets of
$\fo$-submodules of dimension $2k$ in~$\fg$.  As indicated the
elements generating the degree-$k$ Pfaffian can be regarded as square
roots of suitable determinants (up to an irrelevant choice of sign).

% old version:
% 
% \[
% \sum_{\mathbf{i} = (i_1,\ldots,i_{2k})}
% \mathrm{sgn}(\sigma_\mathbf{i}) \prod_{j=1}^k
% A_w(Y_{i_{2j-1}},Y_{i_{2j}}) = \raisebox{1.2ex}{\text{``}} \sqrt{\det
%   \big( (A_w(Y_i,Y_j)_{1 \le i,j \le 2k}) \big)}
% \,\raisebox{1.2ex}{\text{''}},
% \]
% where $Y_1,\ldots,Y_{2k}$ run through spanning sets of
% $\fo$-submodules of dimension $2k$ in~$\fg$, in each
% case the sum on the left-hand side ranges over all partitions of
% $\{1,\ldots,2k\}$ into $2$-element subsets
% $\{i_1, i_2\}, \ldots, \{i_{2k-1},i_{2k}\}$ subject to
% \[
% i_1 < i_3 < \ldots < i_{2k-1} \quad \text{ and } \quad i_{2j-1} < i_{2j} \text{ for $1
%   \le j \le k$,}
% \]
% and the permutation $\sigma_\mathbf{i}$ is given by
% $\sigma_\mathbf{i}(j) = i_j$ for $1 \le j \le 2k$.  As indicated the
% elements generating the degree-$k$ Pfaffian can be regarded as square
% roots of suitable determinants (up to a choice of sign).

If $\mathbf{Y} = (Y_1,\ldots,Y_n)$ is any $\fo$-basis of
$\fg$, then the structure matrix
$[A_w]_\mathbf{Y} \in \Mat_n(\ff)$ of $A_w$ with respect to
$\mathbf{Y}$ is alternating.  Furthermore, its elementary divisors
$\pi^{\nu_1(w)}, \ldots, \pi^{\nu_n(w)}$, where
$\nu_1(w), \ldots, \nu_n(w) \in \Z \cup \{\infty\}$ with
$\nu_1(w) \le \ldots \le \nu_n(w)$, come in pairs:
$\nu_{2j-1}(w) = \nu_{2j}(w)$ for
$1 \le j \le \lfloor \nicefrac{n}{2} \rfloor$.  Finally, we see that
\[
\Pfaff_k(w) = \Ip^{\sum_{j=1}^k \nu_{2j-1}(w)} \quad \text{for
  $1 \le k \le \lfloor \nicefrac{n}{2} \rfloor$,}
\]
and this leads to
\[
\lvert \fg : \stab_\fg(\omega) \rvert^{\nicefrac{1}{2}} =
q^{\frac{1}{2}\sum_{i=1}^n \max \{-\nu_i(w),0\}} = \left\| \bigcup \{
  \Pfaff_k(w) \mid 0 \le k \le \lfloor \nicefrac{n}{2} \rfloor \}
\right\|_\Ip,
\]
where $\| S \|_\Ip = \max \{ \lvert x \rvert_\Ip
\mid x \in S \}$ for $\varnothing \not = S \subseteq \ff$.

As a consequence of Propositions~\ref{pro:ind-finite-dim-adm-rel-FAb}
and \ref{pro:zeta-algebraic-formula} we obtain the following result;
compare Proposition~\ref{propABC:integrals} in the introduction.

\begin{pro} \label{pro:zeta-integral-formula} Let $\fg$ be an
  $\fo$-Lie lattice with an $\fo$-Lie sublattice $\fh$ such that
  $\lvert \fg : \fh + [\fg,\fg] \rvert < \infty$.  Write
  $n = \dim_\fo \fg$ and $m+1 = \dim_\fo \fg - \dim_\fo \fh$.  Let
  $r \in \N_0$ be such that $G = \exp(\pi^r \fg)$ is a finitely
  generated torsion-free potent pro-$p$ group with potent subgroup
  $H = \exp(\pi^r \fh) \le_\mathrm{c} G$.  Then $G$ is FAb relative to
  $H$ and the zeta function of $\rho = \Ind_H^G(\mathbb{1}_H)$ is
  given by the following integral formulae.

  \smallskip

  \textup{(1)} Writing $W = \{ w \in
  \Hom_\fo(\fg,\ff) \mid w(\fh)
  \subseteq \fo \}$, we have
  \[
  \zeta_\rho(s) = q^{r(m+1)} \int_{w \in W} \left\| \bigcup \{
      \Pfaff_k(w) \mid 0 \le k \le \lfloor \nicefrac{n}{2} \rfloor \}
    \right\|_\Ip^{-1-s} \,\di\mu(w),
  \]
  where $\mu$ denotes the normalised Haar measure satisfying
  $\mu(\Hom_\fo(\fg,\fo)) = 1$

  \smallskip

  \textup{(2)} For simplicity, suppose further that the
  $\fo$-module $\fg$ decomposes as a direct sum
  $\fg = \fk \oplus \fh$.  Interpreting
  $\Hom_\fo(\fk,\fo)$ as the
  $\fo$-points $\fV(\fo)$ of
  $m+1$-dimensional affine space $\fV$ over
  $\Spec(\fo)$, let $\fX$ denote the
  projectivisation $\mathbb{P} \fV$ over
  $\Spec(\fo)$.  For
  $0 \le k \le \lfloor \nicefrac{n}{2} \rfloor$, the map
  $w \mapsto \Pfaff_k(w)$ induces a sheaf of ideals $\calI_k$ on
  $\fX$, and
  \[
  \zeta_\rho(s) = (1- q^{-1}) q^{r(m+1)} \sum_{\ell \in \Z}
  q^{-\ell(m+1)} \int_{\fX(\fo)} \left( \max_{0 \le
      k \le \lfloor n/2 \rfloor} \left\| \pi^{k \ell} \calI_k
    \right\|_\Ip \right)^{-1-s} \,\di\mu_{\fX,\Ip},
  \]
  where the suggestive notation $\Vert \pi^{k \ell} \calI_k \Vert_\Ip$
  is employed in anticipation of a formal definition
  in~\eqref{equ:def-Vert-sheaf} and $\mu_{\fX,\Ip}$ denotes the
  canonical $p$-adic measure on~$\fX(\fo)$; see
  Remark~\ref{rmk:canonicalmeasure}.
 \end{pro}

\begin{proof}
  From $\lvert \fg : \fh + [\fg,\fg] \rvert < \infty$ we infer that
  $\lvert G : H[G,G] \rvert < \infty$ so that $G$ is FAb relative
  to~$H$.  In the described set-up the formula for $\zeta_\rho$
  provided by Proposition~\ref{pro:zeta-algebraic-formula} translates
  directly into the first integral formula.  To derive the second
  formula write $\Hom_\fo(\fk,\ff) \smallsetminus \{0\}$ as a disjoint
  union of the sets
  $\pi^\ell \Hom_\fo(\fk,\fo) \smallsetminus \pi^{\ell+1}
  \Hom_\fo(\fk,\fo)$
  and recall that the canonical $p$-adic measure $\mu_{\fX,\Ip}$ is
  normalised so that
  $\int_{\fX(\fo)} \di\mu_{\fX,\Ip} = q^{-m} \lvert \fX(\fo/\Ip)
  \rvert = (1-q^{-m-1})/(1-q^{-1})$,
  where we use that $\fX$ is the projective $m$-dimensional space.
\end{proof}

\begin{rmk}\label{rmk:assumptions-on-ideals-hold}
  By definition $\Pfaff_0(w) = \fo$ is the degree-$0$ Pfaffian of any
  form $w\in \Hom_\fo(\fg,\ff)$.  This means that the ideal sheaf
  $\calI_0$ is the structure sheaf $\calO_\fX$ of $\fX$.  In addition,
  the assumption $|\fg:\fh + [\fg,\fg]| < \infty$ implies that the
  image ${\rm{pr}}_\fk([\fg,\fg])$ of the projection of $[\fg,\fg]$ to
  the complement $\fk$ is of finite index; i.e.,
  $\pi^\epsilon \fk \subseteq {\rm{pr}}_\fk([\fg,\fg])$ for some
  non-negative integer $\epsilon \in \Z$.  Since
  $\Pfaff_1(w) = w({\rm pr}_\fk([\fg,\fg])) \supseteq \pi^\epsilon
  w(\fk)$,
  we deduce that the ideal sheaf $\calI_1$ contains
  $\pi^\epsilon \calO_\fX$.
\end{rmk}

Next we formulate the two integral formulae more explicitly, subject
to a choice of coordinates.  For simplicity, we suppose that the
$\fo$-module $\fg$ decomposes as a direct sum $\fg = \fk \oplus \fh$.
Fix an $\fo$-basis $\mathbf{Y} = (Y_1,\ldots,Y_n)$ of $\fg$ such that
$Y_1,\ldots,Y_{m+1}$ span the complement $\fk$ of $\fh$ in $\fg$ and
$Y_{m+2},\ldots,Y_{n}$ span $\fh$. Denoting by $(c_{i,j,k})_{i,j,k}$
the structure constants of the $\fo$-Lie lattice $\fg$ with respect to
the basis $\mathbf{Y}$, we have
$[Y_i,Y_j] = \sum_{k=1}^n c_{i,j,k} Y_k$ for $1 \le i,j \le n$.  The
\emph{commutator matrix} of $\fg$ with respect to $\mathbf{Y}$ is the
alternating matrix of linear homogeneous forms
\[
\calR(\mathbf{T}) = \calR(T_1,\ldots,T_n) = \left(
  \sum\nolimits_{k=1}^n c_{i,j,k} T_k \right)_{i,j} \in \Mat_n
(\fo[T_1,\ldots,T_n]).
\]
For $k \in \{ 0, \ldots, \lfloor \nicefrac{n}{2} \rfloor \}$ we denote by $P_k
\subseteq \fo[T_1,\ldots,T_n]$ the collection of degree-$k$
Pfaffians of the matrix $\calR(\mathbf{T})$, i.e.\ square roots
of principal $2k \times 2k$-minors of $\calR(\mathbf{T})$; in
particular, $P_0 = \{1\}$.  Furthermore, we set
\[
F_k = \{ f(T_1,\ldots,T_{m+1},0,\ldots,0) \mid f \in P_k \} \subseteq
\fo[T_1,\ldots,T_{m+1}]
\]
and $F_k(\underline{x}) = \{ f(\underline{x}) \mid f \in F_k \}$ for
$\underline{x} \in \ff^{m+1}$.  Observe that $F_k$ consists
of homogeneous polynomials of degree $k$ so that for $f \in F_k$ and
$(y_1 \!:\! \ldots \!:\! y_{m+1}) \in \mathbb{P}^m(\fo)$ the value
$f(y_1,\ldots,y_{m+1})$ is independent of the particular choice of
homogeneous coordinates up to a unit in~$\fo$.

With this notational set-up, we obtain the following concrete version
of Proposition~\ref{pro:zeta-integral-formula}, based on a choice of
coordinates.

\begin{pro} \label{pro:zeta-integral-formula-coord} Let $\fg$ be an
  $\fo$-Lie lattice with an $\fo$-Lie sublattice $\fh$ such that
  $\lvert \fg : \fh + [\fg,\fg] \rvert < \infty$.  Suppose that the
  $\fo$-module $\fg$ decomposes as a direct sum
  $\fg = \fk \oplus \fh$, and put
  $m+1 = \dim_\fo \fk = \dim_\fo \fg - \dim_\fo \fh$.  Let
  $\mathbf{Y} = (Y_1,\ldots,Y_n)$ be an $\fo$-basis of $\fg$ such that
  $Y_1, \ldots, Y_{m+1}$ (respectively $Y_{m+2}, \ldots, Y_{n}$) form an
  $\fo$-basis of $\fk$ (respectively~$\fh$), and let $F_k$, depending on
  $\mathbf{Y}$, be defined as above.

  Let $r \in \N_0$ such that $G = \exp(\pi^r \fg)$ is a finitely
  generated torsion-free potent pro-$p$ group with potent subgroup
  $H = \exp(\pi^r \fh) \le_\mathrm{c} G$.  Then $G$ is FAb relative to
  $H$ and the zeta function of $\rho = \Ind_H^G(\mathbb{1}_H)$ is
  given by the following formulae:
  \[
  \zeta_\rho(s) = q^{r(m+1)} \int_{\ff^{m+1}}  \left\|
      \bigcup \{ F_k (\underline{x}) \mid 0 \le k \le \lfloor \nicefrac{n}{2}
      \rfloor \} \right\|_\Ip ^{-1-s}
  \,\di\mu(\underline{x}),
  \]
  where $\mu$ denotes the normalised Haar measure satisfying
  $\mu(\fo^{m+1}) = 1$, and
  \begin{multline*}
    \zeta_\rho(s) = (1- q^{-1}) q^{r(m+1)} \\ \cdot \sum_{\ell \in \Z}
    q^{-\ell(m+1)} \int_{\mathbb{P}^m(\fo)} \left( \max_{0 \le k \le
        \lfloor \nicefrac{n}{2} \rfloor} \left\| \pi^{k \ell}
        F_k(\underline{y}) \right\|_\Ip \right)^{-1-s}
    \,\di\mu_{\mathbb{P}^m,\Ip}(\underline{y}),
  \end{multline*}
  where $\mu_{\mathbb{P}^m,\Ip}$ denotes the canonical $p$-adic
  measure on $\mathbb{P}^m(\fo)$; see
  Remark~\ref{rmk:canonicalmeasure}.
\end{pro}

Examples based on this concrete formula are discussed in Section
\ref{sec:examples-orbit-method} below.

\begin{rmk} \label{rmk:connection-previous-methodology} The formulae
  in Propositions~\ref{pro:zeta-integral-formula} and
  \ref{pro:zeta-integral-formula-coord} are similar, but not identical
  to integral formulae that are familiar from the study of ordinary
  zeta functions; compare, for instance, \cite[Lem.~4.1]{Ja06} and
  \cite[Sec.~3.2]{AvKlOnVo13}.  The transparent description of the
  relevant zeta functions as inifinite series of Igusa integrals
  highlights aspects that were not so clearly visible before; see
  Section~\ref{sec:rationality-etc}.  Moreover, the introduction of an
  auxiliary integration variable is avoided, and explicit
  calculations, as in Section~\ref{sec:examples-orbit-method}, benefit
  from the approach.

  Nevertheless, the coordinate-based integral appearing in the first
  formula in Proposition~\ref{pro:zeta-integral-formula-coord} could
  easily be rewritten as an integral over the compact affine set
  $\mathfrak{o}^{m+2}$ as in \cite[Sec.~3.2]{AvKlOnVo13}, at the
  expense of adding an auxiliary variable.  We illustrate this in
  Section~\ref{sec:examples-orbit-method}, just after the proof of
  Proposition~\ref{pro:ind-Br-Gr}.
\end{rmk}

%%%%%

\section{A Denef-type formula for globally induced
  representations} \label{sec:rationality-etc} The discussion in this
section has a strong geometric flavour.  Accordingly, we adopt some
terminology that is commonly used in geometry.  Let $K_0$ be a number
field with ring of integers $O_{K_0}$, and fix a finite set of closed
points $S \subseteq \Spec(O_{K_0})$.  Let
$O_{K_0,S} = \{ a \in K_0 \mid \lvert a \rvert_\Ip \le 1 \text{ for
  closed points } \Ip \not\in S \}$
denote the ring of $S$-integers in~$K_0$.

We consider number fields $K$ that arise as finite extensions
of~$K_0$.  Let $O_S = O_{K,S}$ denote the integral closure of
$O_{K_0,S}$ in~$K$.  For a maximal ideal $\Ip \trianglelefteq O_S$ we
write $\kappa_\Ip = O_S/\Ip$ for the residue field and we denote its
cardinality by $q_\Ip$.  Furthermore, $O_\Ip = O_{K,S,\Ip}$ denotes
the completion of $O_S$ with respect to $\Ip$ and we fix a uniformiser
$\pi_\Ip$ so that $O_\Ip$ has the valuation ideal $\pi_\Ip O_\Ip$.
\[
\begin{array}{cccccclcl}
  O_K & \subseteq & O_S = O_{K,S} & \subseteq & K  %
  & \phantom{xx} & O_\Ip = O_{K,S,\Ip}  \trianglerighteq \pi_\Ip
                   O_\Ip  & \phantom{xx} & \\
  \vert && \vert && \vert &&&& \\
  O_{K_0} & \subseteq & O_{K_0,S} & \subseteq &  K_0  %
  && \kappa_\Ip \cong O_\Ip/\pi_\Ip O_\Ip \cong O_S /\Ip 
                 && q_\Ip = \lvert \kappa_\Ip \rvert
\end{array}
\]

Let $\fg$ be an $O_{K_0,S}$-Lie lattice.  For every finite extension
$K$ of~$K_0$ and every maximal ideal $\Ip \trianglelefteq O_S$ we
consider the $O_\Ip$-Lie lattice
$\fg_{\Ip} = O_\Ip \otimes_{O_{K_0,S}} \fg$ and, for $r \in \N_0$, its
principal congruence
sublattices~$\fg_{\Ip,r} = \pi_\Ip^{\, r} \fg_{\Ip}$.  Observe that,
for any given $K$ and $\Ip$, the Lie lattice $\fg_{\Ip,r}$ is potent
for all sufficiently large integers~$r$ so that
$G_{\Ip,r} = \exp(\fg_{\Ip,r})$ is a finitely generated torsion-free
potent pro-$p$ group; we say that such $r$ are \emph{permissible}
for~$\fg_\Ip$.

With these preparations we formulate the main result of this section,
which corresponds to Theorem~\ref{thm:D} in the introduction.

\begin{thm} \label{thm:main-summary} As in the set-up described above,
  let $\fg$ be an $O_{K_0,S}$-Lie lattice with an $O_{K_0,S}$-Lie
  sublattice $\fh$ such that
  $\lvert \fg : \fh + [\fg,\fg] \rvert < \infty$.  Suppose that $\fh$
  is a direct summand as a submodule of the $O_{K_0,S}$-module~$\fg$,
  and put $m + 1= \dim_{O_{K_0,S}} \fg - \dim_{O_{K_0,S}} \fh$.

  For closed points $\Ip \in \spec(O_{K,S})$, where $K$ ranges over
  finite extensions of~$K_0$, and for positive integers~$r$ that are
  permissible for $\fg_\Ip$ and $\fh_\Ip$, we consider the induced
  representation
  \[
  \rho_{\Ip,r} = \Ind_{H_{\Ip,r}}^{G_{\Ip,r}}(\one_{H_{\Ip,r}})
  \]
  associated to the pro-$p$ groups $G_{\Ip,r} = \exp(\fg_{\Ip,r})$ and
  $H_{\Ip,r}= \exp(\fh_{\Ip,r})$.
  
  \begin{enumerate}[\rm (1)]
  \item For each $\Ip$ there is a complex-valued function $Z_\Ip$ of a
    complex variable $s$ that is rational in $q_\Ip^{-s}$ with integer
    coefficients so that for all permissible~$r$,
   \begin{equation*}
     \zeta_{\rho_{\Ip,r}}(s) = (1-q_\Ip^{-1})q^{r(m+1)}_\Ip Z_\Ip(s).
   \end{equation*}
 \item The real parts of the poles of the functions $Z_{\Ip}$, for all
   $\Ip$, form a finite subset $P_{\fg,\fh} \subseteq \Q$.
 \item There is a finite extension $K_1$ of $K_0$ such that, for all
   closed points $\Ip \in \spec(O_{K,S})$, arising for extensions
   $K \supseteq K_1$, and for all permissible~$r$, the abscissa of
   convergence of the zeta function $\zeta_{\rho_{\Ip,r}}$ satisfies
   \[
   \alpha(\zeta_{\rho_{\Ip,r}}) = \max P_{\fg,\fh}.
   \]
 \item \label{itm:funct-equ} There are an open dense subscheme
   $\spec(O_{K_0,T}) \subseteq \spec(O_{K_0,S})$ and a rational
   function $F \in \Q(Y_1, Y_2, X_1,\dots, X_g)$ such that the
   following holds:

   For every closed point $\Ip_0 \in \Spec(O_{K_0,T})$ there are
   algebraic integers
   $\lambda_1 = \lambda_1(\Ip_0), \dots, \lambda_g = \lambda_g(\Ip_0)
   \in \C^\times$
   so that for every finite extension $K$ of~$K_0$ and every closed
   point $\Ip \in \spec(O_{K,T})$ lying above $\Ip_0$,
   \begin{equation*}
     Z_\Ip(s) = F(q^{-f} ,q^{-fs},\lambda_1^{\, f},\dots,\lambda_g^{\, f}),
   \end{equation*}
   where $q = q_{\Ip_0}$ and $f = [ \kappa_\Ip : \kappa_{\Ip_0}]$
   denotes the inertia degree of $\Ip$ over~$K_0$.

   Furthermore the following functional equation holds:
   \begin{equation*}
     F(q^f, q^{fs}, \lambda_1^{\, -f}, \dots, \lambda_g^{\, -f})
     = -q^{fm} F(q^{-f},q^{-fs}, \lambda_1^f, \dots, \lambda_g^f). 
   \end{equation*}
 \end{enumerate}
\end{thm}

This theorem is a consequence of Theorems~\ref{thm:abscissa}
and~\ref{thm:functionalEquation} below, which hold for a larger class
of integrals of Igusa type described in~\eqref{equ:general-integral}
below.  Indeed, the zeta functions of induced representations are of
the required form due to the formula given in
Proposition~\ref{pro:zeta-integral-formula}.  As discussed in the
introduction, our proof proceeds along the lines
of~\cite{DeMe91,Vo10}.  Our approach is of independent interest,
because the specific shape of our zeta functions allows us to
introduce quite naturally the new concept of \emph{local abscissae of
  convergence}, which leads to a strengthening of known results.  In
addition, the specific form is conceptually adapted to our
applications and also turned out to be useful in calculations of
concrete examples; e.g., see Proposition~\ref{pro:ind-Br-Gr} below.
Compared to other approaches, we work on a different domain of
integration and we modify the zeta functions with an infinite series:
the inner integrals in such a series are always convergent and only
the outer sum can cause divergence.  We provide an almost
self-contained discussion, by streamlining arguments that can already
be found in~\cite{DeMe91,Vo10}.

\begin{rmk} \label{rmk:not-global} The assertions in (1), (2), and (3)
  of Theorem~\ref{thm:main-summary} have local analogues for Lie
  lattices that are defined over the valuation ring of a $p$-adic
  field; i.e.\ a finite extension field of some $\Q_p$.  In fact, the
  proof of the underlying Theorem~\ref{thm:abscissa} given in Section
  \ref{sec:ProofAbscissa} carries over directly if the base field
  $K_0$ is assumed to be a $p$-adic field (instead of a number field).
  Note that there are $\Z_p$-Lie lattices $\fg$ which cannot be
  defined globally; this means, there is no Lie lattice $\fg_0$
  defined over the ring of $S$-integers $O_S$ of a number field and a
  prime ideal $\Ip \trianglelefteq O_S$ such that $O_\Ip \cong \Z_p$
  and $O_\Ip \otimes_{O_S} \fg_0 \cong \fg$.
\end{rmk}

%%%

\subsection{$p$-Adic integrals} \label{sec:p-adic-integrals}
We continue to use the notation set up above and write 
\begin{equation} \label{equ:not-51} O_S = O_{K,S}, \qquad O_\Ip =
  O_{K,S,\Ip}, \qquad \pi = \pi_\Ip
\end{equation}
for short.  Let $\schX$ be a smooth integral projective scheme over
$\spec(O_{K_0,S})$, and let $m$ be the dimension of the generic
fibre~$\schX_{K_0}$.  Let $d \in \N$, and let
$\boldsymbol{\calI} = (\calI_j)_{j=0}^d$ and
$\boldsymbol{\calJ} = (\calJ_j)_{j=0}^d$ be two collections of
coherent sheaves of ideals on~$\schX$.

\begin{assumption}\label{ass:ideals}
  We assume throughout that $\calI_0 = \calJ_0 = \calO_\schX$ is the
  structure sheaf of $\schX$, and
  $\gamma \calO_\schX \subseteq \calI_{j_0}$ for some non-zero
  $\gamma \in O_{K_0,S}$ and some $j_0 \in \{1,\dots,d\}$.
\end{assumption}

A sheaf of ideals $\calI$ on $\schX$ defines a continuous function
$\Vert \calI \Vert_\Ip$ on the compact space $\schX(O_\Ip)$ by
\begin{equation} \label{equ:def-Vert-sheaf}
  \Vert \calI \Vert_\Ip (x) = \max \{ \lvert f(x) \rvert_\Ip \mid f
  \in \calI_x \},
\end{equation}
where $\calI_x = \calI \cdot \calO_{\schX,x}$ denotes the stalk of
$\calI$ at $x \in \schX(O_\Ip)$.  We are interested in weighted
combinations of such functions.  For $\ell \in \Z$ we define
\begin{equation*}
  \left\Vert (\pi^{j\ell} \calI_j)_{j=0}^d \right\Vert_{\Ip} \coloneqq
  % \left\Vert \bigcup \{
  %   \pi^{j \ell} \calI_j \mid 0 \le j \le d \} \right\Vert_\Ip =
  \max_{0 \le j \le  d} q _\Ip^{\, -j\ell} \Vert \calI_j \Vert_\Ip 
\end{equation*}
and likewise for $\boldsymbol{\calJ}$.  

\begin{rmk}\label{rmk:canonicalmeasure}
  On the compact topological space $\schX(O_\Ip)$ we consider the
  \emph{canonical $p$-adic measure} $\mu_{\schX, \Ip}$ defined by
  Batyrev in \cite[Def.~2.6]{Ba99}.  This measure is locally
  defined in terms of differential forms, i.e.\ sections of the
  canonical bundle $\omega_{\schX/O_\Ip}$.  It can be defined without
  assuming the existence of a global gauge form, on which Weil's
  original approach \cite{Weil82} was based.  Here we will mainly use
  the following property of the canonical $p$-adic measure: for every
  point $x \in \schX(\kappa_p)$, the fibre $Y$ over $x$ in
  $\schX(O_\Ip)$ has measure $\mu_{\schX, \Ip}(Y) = q_\Ip^{-m}$; in
  particular,
  $\mu_{\schX, \Ip}(\schX(O_\Ip)) = q_\Ip^{-m} |\schX(\kappa_\Ip)|$ as
  shown in~\cite[Thm.~2.7]{Ba99}.
\end{rmk}

Fix once and for all a positive integer~$c \in \N$.  We consider the
\emph{local zeta function}
\begin{equation} \label{equ:general-integral}
  Z_\Ip(s) = Z_{\schX,\boldsymbol{\calI},\boldsymbol{\calJ},\Ip}(s)
  = \sum_{\ell \in \Z} q _\Ip^{\, -\ell c } \int_{\schX(O_\Ip)}  
  \left\Vert (\pi^{j\ell} \calI_j)_{j=0}^d \right\Vert_\Ip^{-s} 
  \left\Vert (\pi^{j\ell} \calJ_j)_{j=0}^d \right\Vert^{-1}_\Ip
  \di\mu_{\schX, \Ip}. 
\end{equation}
For instance, in the setting of
Proposition~\ref{pro:zeta-integral-formula}, we encounter such zeta
functions for~$c = m+1$.  In such applications the collections
$\boldsymbol{\calI}$ and $\boldsymbol{\calJ}$ may very well contain
zero ideal sheaves; however, these do not contribute to the local zeta
function. It follows from Remark \ref{rmk:assumptions-on-ideals-hold}
that Assumption \ref{ass:ideals} holds in this case. The next lemma
shows that these assumptions guarantee the convergence of the infinite
series in~\eqref{equ:general-integral}.

\begin{lem}
  Under Assumption~$\mathrm{\ref{ass:ideals}}$ the series $Z_\Ip(s)$
  converges absolutely for all $s \in \C$ with
  $\mathrm{Re}(s) > \nicefrac{c}{j_0}$.
\end{lem}

\begin{proof}
  First we carry out the summation over all non-negative integers.
  For $\ell \ge 0$, the functions
  $\left\Vert (\pi^{j\ell} \calI_j)_{j=0}^d \right\Vert_\Ip$ and
  $\left\Vert (\pi^{j\ell} \calJ_j) _{j=0}^d \right\Vert_\Ip$ have
  value $1$ at every $x \in \schX(O_\Ip)$, because
  $\calI_0 = \calJ_0 = \calO_\schX$.  We deduce that
 \begin{align*}
   \sum_{\ell = 0}^{\infty}  q _\Ip^{\, -\ell c } \int_{\schX(O_\Ip)}
   \left\Vert (\pi^{j\ell} \calI_j)_{j=0}^d \right\Vert_\Ip^{-s} 
   \left\Vert (\pi^{j\ell} \calJ_j)_{j=0}^d \right\Vert^{-1}_\Ip
   \di\mu_{\schX, \Ip} 
   &= \sum_{\ell = 0}^{\infty}  q _\Ip^{\, -\ell c}
     \mu_{\schX,\Ip}(\schX(O_\Ip)) \\
   &=  \frac{\mu_{\schX,\Ip}(\schX(O_\Ip))}{1-q_\Ip^{\, -c}},
 \end{align*}
 with absolute convergence independently of~$s \in \C$.

 It remains to sum over all negative integers.  Using
 $\gamma\calO_\schX \subseteq \calI_{j_0} $ to justify the second inequality below,
 we observe that for all $s \in \C$ with
 $\mathrm{Re}(s) > \nicefrac{c}{j_0}$,
 \begin{multline*}
   \sum_{\ell = 1}^{\infty} \left| q _\Ip^{\, \ell c }
     \int_{\schX(O_\Ip)} \left\Vert (\pi^{-j\ell} \calI_j)_{j=0}^d
     \right\Vert_\Ip^{-s} \left\Vert (\pi^{-j\ell} \calJ_j)_{j=0}^d
     \right\Vert^{-1}_\Ip
     \di\mu_{\schX, \Ip} \right| \\
   \le \sum_{\ell = 1}^{\infty} q _\Ip^{\, \ell c }
   \int_{\schX(O_\Ip)} \Vert \pi^{-j_0\ell} \calI_{j_0}
   \Vert_{\Ip}^{-\mathrm{Re}(s)} \di\mu_{\schX, \Ip} 
   \le \sum_{\ell = 1}^{\infty} q _\Ip^{\, \ell (c -\mathrm{Re}(s)j_0)}
   |\gamma|^{ -\mathrm{Re}(s)}_\Ip \mu_{\schX,\Ip}(\schX(O_\Ip)) \\
   = \frac{q_\Ip^{\, c -\mathrm{Re}(s)j_0} |\gamma|^{ -\mathrm{Re}(s)}_\Ip }{1-q_\Ip^{\,
       c -\mathrm{Re}(s)j_0}} \, \mu_{\schX,\Ip}(\schX(O_\Ip)). \qedhere
 \end{multline*}
\end{proof}

The infimum of all $\sigma \in \R$ such that the series
$Z_\Ip(\sigma)$ converges is called the \emph{abscissa of convergence}
$\alpha(Z_\Ip)$ of the local zeta function $Z_\Ip$; as seen above, the
series $Z_\Ip(s)$ converges absolutely to an analytic function on the
half-plane $\{ s \in \C \mid \mathrm{Re}(s)>\alpha(Z_\Ip) \}$.
  
The computation of the integral defining the local zeta function
$Z_\Ip$ is rather simple if the ideal sheaves in
$\boldsymbol{\calI}$ and $\boldsymbol{\calJ}$ are \emph{monomial} in
the following sense.

\begin{dfn} Let $k$ be a field and let $\schY$ be a regular scheme
  over~$\spec(k)$.  An ideal sheaf $\calI$ on $\schY$ is called
  \emph{monomial} if for all $y \in \schY$ the ideal
  $\calI_y \trianglelefteq \calO_{\schY,y}$ is of the form
  $\calI_y = \prod_{i = 1}^r z_i^{c_i} \calO_{\schY,y}$, where
  $z_1,\dots,z_r$ are regular parameters at $y$ and
  $c_1, \dots, c_r \in \N_0$ are suitable exponents; cf.\
  \cite[(3.16)]{Ko07}.
 
  Let $\schX$ be a smooth scheme over $\spec(O_S)$.  Observe that for
  every $a \in \spec(O_S)$, the fibre $\schX_a$ is a regular scheme
  over $\spec(k)$, where $k = \kappa_\Ip$ if
  $a = \Ip$ is a closed point and $k = K$ otherwise.  An
  ideal sheaf $\calI$ on $\schX$ is \emph{monomial}, if it is locally
  principal and for all $a \in \spec(O_S)$ the ideal sheaf
  $\calI \cdot \calO_{\schX_a}$ on the fibre $\schX_a$ over $a$ is
  monomial.
\end{dfn}

Many authors refer to a monomial sheaf of ideals as \emph{a sheaf of
  ideals of a divisor with simple normal crossings}.  Hironaka's
famous resolution of singularities \cite{Hironaka1964} implies that
over a field of characteristic $0$ any ideal sheaf can be transformed
into a monomial one using a suitable sequence of blow-ups.  We will
use the following version of this result, which is explained
in~\cite[Thm.~3.26]{Ko07}.  By a variety over a field $k$ we
mean an integral, separated scheme of finite type over~$\spec(k)$.

\begin{thm}[Hironaka Monomialisation Theorem] \label{thm:Hironaka} Let
  $X$ be a smooth variety over some field $k$ of characteristic $0$
  and let $\calI$ be an ideal sheaf on $X$.  There are a smooth
  variety $Y$ over $k$ and a projective morphism $h\colon Y \to X$
  such that the ideal sheaf $h^*(\calI)$ is monomial and $h$ restricts
  to an isomorphism between the complements of the closed subvarieties
  defined by $h^*(\calI)$ and~$\calI$.
\end{thm}

%%%

\subsection{The abscissa of convergence}
We continue to use the notation set up above.  In this section we
study the abscissa of convergence of the zeta functions
$Z_{\Ip} = Z_{\schX,\boldsymbol{\calI},\boldsymbol{\calJ},\Ip}$,
for finite extensions $K$ of $K_0$ and closed points
$\Ip \in \spec(O_S)$, as defined in~\eqref{equ:general-integral}.  The
main result is the following theorem.

\begin{thm}\label{thm:abscissa}
  In the set-up described above, including
  Assumption~$\mathrm{\ref{ass:ideals}}$, the following statements hold.
  \begin{enumerate}[\rm (1)]
  \item For every finite extension $K$ of $K_0$ and every closed point
    $\Ip \in \spec(O_S)$, there is a rational function
    $F_\Ip \in \Q(Y)$ such that
    \[
    Z_\Ip(s) = F_\Ip(q_\Ip^{-s}).
    \]
  \item The real parts of the poles of the functions $Z_{\Ip}$, for
    all $K$ and $\Ip$ as above, form a finite set
    $P = P_{\schX,\boldsymbol{\calI},\boldsymbol{\calJ}} \subseteq \Q$
    of rational numbers.
  \item There is a finite extension $K_1$ of $K_0$ such that, for all
    $K$ and $\Ip$ as above satisfying $K_1 \subseteq K$, the abscissa of
    convergence of $Z_\Ip$ satisfies $\alpha(Z_\Ip) = \max P$.
  \end{enumerate}
\end{thm}

The proof of the theorem is described in
Section~\ref{sec:ProofAbscissa}.  Here we briefly discuss an
interesting consequence.  Let $K_1$ be a finite extension of $K_0$
with the properties described in part (3) of the theorem.  If a
maximal ideal $\Ip_0$ of $O_{K_0,S}$ is unramified in $K_1$ and admits
a prime divisor $\Ip_1 \trianglelefteq O_{K_1,S}$ of inertia degree $1$
over~$\Ip_0$, then the completions $(K_0)_{\Ip_0}$ and $(K_1)_{\Ip_1}$
coincide and hence $Z_{\Ip_0} = Z_{\Ip_1}$, in particular
$\alpha(Z_{\Ip_0}) = \alpha(Z_{\Ip_1})$.  By the Chebotarev Density
Theorem~\cite[VII~(13.6)]{Ne99} the set of such primes
$\Ip_0 \trianglelefteq O_{K_0,S}$ has Dirichlet density at least
$[K_1:K_0]^{-1}$.

\begin{cor}\label{cor:abscissaDensity}
  The set of closed points $\Ip_0 \in \spec(O_{K_0,S})$ satisfying
 \begin{equation*}
    \alpha(Z_{\Ip_0}) = \max P
 \end{equation*}
 has positive Dirichlet density.
\end{cor}

%%%

\subsubsection{The abscissa of convergence of certain Dirichlet
  series}\label{sec:abscissaAbstract}
We discuss a general result on a family of power series.  Let
$\Q\LS{Q}\PS{t}$ be the ring of formal power series in $t$ over the
field of Laurent series $\Q\LS{Q}$ in an indeterminate~$Q$.  Fix
integers $d,u \in \N$ and two collections
$\boldsymbol{\lambda} = (\lambda_j)_{j=0}^d$,
$\boldsymbol{\beta} = (\beta_j) _{j=0}^d$ of integral linear forms
$\lambda_j, \beta_j \colon \Z^{1+u} \to \Z$ for $0 \le j \le d$.
We say that an integral linear form $\lambda \colon \Z^{1+u} \to \Z$
is \emph{strictly negative} if $\lambda(v) < 0$ for every
$v \in \N_0^{\, 1+u} \smallsetminus \{ 0 \}$.  It is convenient for us
to write the elements of $\Z^{1+u}$ as pairs $(\ell,n)$, where
$\ell \in \Z$ and $n = (n_1, \dots, n_u) \in \Z^u$.  The following
assumptions are analogous to those in Assumption~\ref{ass:ideals}.

\begin{assumption} \label{ass:linear-forms}
  We assume throughout: (i) $\lambda_0$ is strictly negative, (ii)
  $\beta_0 = 0$, and (iii) there are an index $j_0 \in \{1,\dots,d\}$
  and $a \in \N$ such that $\beta_{j_0}(\ell, n) = a \ell$ for all
  $(\ell,n) \in \Z^{1+u}$.
\end{assumption}

For every integer $N \ge 0$ and vectors
$\boldsymbol{\epsilon} = (\epsilon_0,\dots,\epsilon_d),
\boldsymbol{\delta} =(\delta_0,\dots,\delta_d) \in \Z^{d+1}$
with $\delta_0 = 0$ we define a power series
\begin{equation} \label{equ:def-Xi-Q-t} 
  \begin{split}
    \Xi_{N, \boldsymbol{\epsilon},
      \boldsymbol{\delta}}^{\boldsymbol{\lambda},\boldsymbol{\beta}}
    (Q,t)  & = \sum_{\ell \in \Z} \sum_{n \in (\Z_{\ge N})^u} Q^{-
      \min\limits_{0 \le j \le d} \big( \lambda_j(\ell,n) +
      \epsilon_j \big) } \, t^{- \min\limits_{0 \le j
        \le d} \big( \beta_j(\ell,n) + \delta_j \big)} \\
     & \in \Q\LS{Q}\PS{t}.
  \end{split}
\end{equation}
To simplify the notation, we write
$\Xi_{N, \boldsymbol{\epsilon}, \boldsymbol{\delta}}$ in place of
$\Xi_{N, \boldsymbol{\epsilon},
  \boldsymbol{\delta}}^{\boldsymbol{\lambda},\boldsymbol{\beta}}$.
Let us verify that \eqref{equ:def-Xi-Q-t} defines an element of
$\Q\LS{Q}\PS{t} \smallsetminus \Q\LS{Q}[t]$.  The variable $t$ only
occurs with non-negative exponents since $\beta_0 = 0$ and
$\delta_0 = 0$.  Consider summands contributing to the coefficient of
$t^e$, for some fixed exponent $e \ge 0$.  By
Assumption~\ref{ass:linear-forms}~(iii), the relevant summands only
occur for $\ell \ge -\nicefrac{(e+\delta_{j_0})}{a}$.  By
Assumption~\ref{ass:linear-forms}~ (i), $\lambda_0$ is strictly
negative; hence every monomial $Q^k t^e$ occurs only a finite number
of times and, moreover, $k$ is bounded below in terms of~$\ell$ and
thus in terms of~$e$.  Finally, we observe that the coefficients are
all positive; as $\ell$ decreases, we pick up non-zero terms of
increasing degree in~$t$, thus
$\Xi_{N, \boldsymbol{\epsilon}, \boldsymbol{\delta}} \not \in
\Q\LS{Q}[t]$.

\begin{lem}\label{lem:XiSeries}
  In the above setting, including Assumption~\ref{ass:linear-forms},
  the following statements hold.
  \begin{enumerate}[\rm (1)]
  \item \label{eq:XiRational} The power series
    $\Xi_{N, \boldsymbol{\epsilon}, \boldsymbol{\delta}}$ is a
    rational function over~$\Q$, i.e., $\Xi_{N,
      \boldsymbol{\epsilon}, \boldsymbol{\delta}} \in \Q(Q,t)$.
  \item \label{eq:FinitePoles1} There is a finite set
    $\widetilde{P} \subseteq (\Z \times \N_0) \smallsetminus
    \{(0,0)\}$
    such that, for all $N, \boldsymbol{\epsilon}$ and
    $\boldsymbol{\delta}$, the rational function
    $\Xi_{N, \boldsymbol{\epsilon}, \boldsymbol{\delta}}$ can be
    written as the quotient of an element of $\Q[Q,Q^{-1},t]$ by a
    power of $\prod_{(A,B) \in \widetilde{P}} (1- Q^A t^B)$.
  \item \label{eq:FinitePoles2} The real parts of poles of
    $\Xi_{N,\boldsymbol{\epsilon},\boldsymbol{\delta}}(q^{-1},q^{-s})$,
    for all $N,\boldsymbol{\epsilon},\boldsymbol{\delta}$ and
    $q \in \R_{> 1}$, form a non-empty finite subset
    $P_{\boldsymbol{\lambda},\boldsymbol{\beta}} \subseteq \Q$.
  \item \label{eq:MaxAttained} The maximum
    $\max P_{\boldsymbol{\lambda},\boldsymbol{\beta}}$ is a pole for
    every
    $\Xi_{N,\boldsymbol{\epsilon},\boldsymbol{\delta}}(q^{-1},q^{-s})$
    as in~\eqref{eq:FinitePoles2}.
  \item \label{eq:XiInversion} The inversion property
    $\Xi_{1, \mathbf{0}, \mathbf{0}}(Q^{-1},t^{-1}) = (-1)^{u+1}
    \Xi_{0, \mathbf{0}, \mathbf{0}}(Q,t)$ holds.
  \end{enumerate}
\end{lem}

\begin{proof}
  Apart from part (4), we only sketch the proof as the assertions are
  essentially paraphrases of known results; see~\cite[Prop.~2.1]{Vo10}
  and \cite[Prop.~4.5]{AvKlOnVo13}.  Furthermore, we may assume that
  $N = 0$ as
  $\Xi_{N,\boldsymbol{\epsilon},\boldsymbol{\delta}}(Q,t) =
  \Xi_{0,\boldsymbol{\epsilon}',\boldsymbol{\delta}'}(Q,t)$
  for $\epsilon_j' = \epsilon_j + \lambda_j(0,N,\dots,N)$ and
  $\delta_j' = \delta_j + \beta_j(0,N,\dots,N)$, where
  $0 \le j \le d$.  For simplicity we drop $N$ from the notation
  altogether.
  
  Decompose $\Z \times \N_0^{\, u}$ into a finite number of disjoint
  rational cones
  $C(\boldsymbol{\epsilon},\boldsymbol{\delta})^{\pm}_{j,k}$, for
  $0 \le j,k \le d$, such that on each
  $C(\boldsymbol{\epsilon},\boldsymbol{\delta})^{\pm}_{j,k}$ the
  minimum in the exponent of $Q$ is attained by
  $\lambda_j + \epsilon_j$, the minimum in the exponent of $t$ is
  attained by $\beta_k+\delta_k$, and $\ell$ is positive or
  non-positive according to the attached sign.  Introducing new
  variables we may describe each cone with the cone of positive
  solutions of a linear integral inhomogeneous system of equations.
  The generating function of such a cone is known to be rational (see
  \cite[Ch.~I]{St96}) and so we obtain
  \eqref{eq:XiRational}.

  Moreover, the denominator of the generating function of each cone is
  of the form $\prod_{\gamma} (1- X^\gamma)$, where the product runs
  over the finitely many completely fundamental solutions $\gamma$ of
  the corresponding \emph{homogeneous} system.  Making the suitable
  substitutions we obtain a finite set $\widetilde{P}$ with the
  properties stated in~\eqref{eq:FinitePoles1}.  Observe that
  $(0,0) \not \in \widetilde{P}$ as $\Xi(Q,t)$ is a well-defined power
  series, as explained just after the defining
  equation~\eqref{equ:def-Xi-Q-t}.

  Statement~\eqref{eq:FinitePoles2} is a direct consequence of
  \eqref{eq:FinitePoles1}, as
   \begin{equation*}
     P_{\boldsymbol{\lambda},\boldsymbol{\beta}} \subseteq \left\{
       -\tfrac{A}{B} \in \Q  \mid (A,B) \in \widetilde{P} 
       \text{ and } B \neq 0 \right\},
   \end{equation*}
   and the inversion property~\eqref{eq:XiInversion} follows from
   \cite[Prop.~8.3]{St74}.
  
   It remains to justify~\eqref{eq:MaxAttained}.  We give more details
   than provided in~\cite{AvKlOnVo13}, where the corresponding
   discussion appears to be short if not incomplete.  Fix $q>1$.  As a
   first step we show that the abscissae of convergence of
   $\Xi_{\boldsymbol{\epsilon},\boldsymbol{\delta}}(q^{-1},q^{-s})$
   and $\Xi_{\mathbf{0}, \mathbf{0}}(q^{-1},q^{-s})$ agree.  Define
   $\calE = \max \{ \lvert \epsilon_i \rvert \mid 0 \le i \le d \}$
   and
   $\Delta = \max \{ \lvert \delta_i \rvert \mid 0 \le i \le d \}$.
   Fix $(\ell,n) \in \Z \times \N_0^{\, u}$, and choose indices
   $j,k \in \{0, \ldots, d \}$ such that
   \[
   \min_{0 \le i \le d} \beta_i(\ell,n) = \beta_j(\ell,n) \quad
   \text{and} \quad \min_{0 \le i \le d} \beta_i(\ell,n)+\delta_i =
   \beta_k(\ell,n) + \delta_k.
   \]
   Then $\beta_j(\ell,n) \le \beta_k(\ell,n)$, whereas
   $\beta_k(\ell,n) +\delta_k \le \beta_j(\ell,n) + \delta_j$, and we
   conclude
  \begin{equation*}
    -\Delta \le \delta_k \le (\beta_k(\ell,n)+\delta_k) -
    \beta_j(\ell,n) \le \delta_j \le \Delta. 
  \end{equation*}
  A similar inequality relates the minima of $\lambda_i(\ell,n)$ and
  $\lambda_i(\ell,n)+\epsilon_i$.  We conclude that, for every
  $\sigma \in \R$ for which at least one of the two series
  $\Xi_{\boldsymbol{\epsilon},\boldsymbol{\delta}}(q^{-1},q^{-\sigma})$
  and $\Xi_{\mathbf{0}, \mathbf{0}}(q^{-1},q^{-\sigma})$ converges,
  \begin{equation*}
    q^{-\calE - |\sigma | \Delta} \, \Xi_{\mathbf{0}, \mathbf{0}}(q^{-1},q^{-\sigma}) \le
    \Xi_{\boldsymbol{\epsilon},\boldsymbol{\delta}}(q^{-1},q^{-\sigma})
    \le q^{\calE + |\sigma|\Delta} \, 
    \Xi_{\mathbf{0}, \mathbf{0}}(q^{-1},q^{-\sigma}).
  \end{equation*}
  Hence the abscissae of convergence of the two series are equal.
  
  Consider the abscissa of convergence of
  $\Xi_{\mathbf{0}, \mathbf{0}}(q^{-1},q^{-s})$. Pick
  $(A,B) \in \widetilde{P}$ with $B \neq 0$ such that
  $\nicefrac{-A}{B}$ is maximal among all the pairs
  $(A,B) \in \widetilde{P}$.  By~\eqref{eq:FinitePoles1}, the series
  $\Xi_{\mathbf{0}, \mathbf{0}}(q^{-1},q^{-s})$ converges absolutely
  for $\mathrm{Re}(s) > \nicefrac{-A}{B}$ (as for ordinary Dirichlet
  series this can be seen by looking at the Taylor series).
  
  Let $\sigma \in \R$ such that
  $\Xi_{\mathbf{0}, \mathbf{0}}(q^{-1},q^{-\sigma})$ converges.  The
  term $(1-Q^A t^B)$ corresponds to some extremal ray in some cone
  $C(\boldsymbol{0},\boldsymbol{0})^\pm_{j,k}$ in the decomposition
  chosen above.  As $\Xi_{\mathbf{0}, \mathbf{0}}$ is a series in $Q$
  and $t$ with non-negative integral coefficients we may sum over the
  extremal ray to obtain
  $\Xi_{\mathbf{0}, \mathbf{0}}(q^{-1},q^{-\sigma}) \ge
  \sum_{i=0}^\infty q^{(-A-B \sigma)i} = \nicefrac{1}{(1-q^{-A - B
      \sigma})}$.
  This implies $\sigma > \nicefrac{-A}{B}$, and thus
  $\nicefrac{-A}{B}$ is the abscissa of convergence.
\end{proof}

%%%

\subsubsection{Local abscissae of convergence and the proof of
  Theorem~\ref{thm:abscissa}}\label{sec:ProofAbscissa}
Recall the notation set up in Section~\ref{sec:p-adic-integrals}, in
particular~\eqref{equ:not-51} and Assumption~\ref{ass:ideals}.
Consider the generic fibre $\schX_{K_0}$ of the scheme~$\schX$,
equipped with the ideal sheaf $\calL_{K_0}$ which is the product of
all \emph{non-zero} ideals $\calI_{j,K_0}$ and $\calJ_{j,K_0}$.  To
simplify the notation, we assume that all $\calI_{j,K_0}$ and
$\calJ_{j,K_0}$ are non-zero, so that
$\calL_{K_0} = \prod_{j=0}^d \calI_{j,K_0}\calJ_{j,K_0}$; in general,
the ideal sheaves that are zero have to be removed.  By Hironaka's
Theorem~\ref{thm:Hironaka}, there is a monomialisation
\begin{equation*}
 h \colon Y \to \schX_{K_0}
\end{equation*}
of the ideal sheaf $\calL_{K_0}$.  If a product of non-zero ideals is
monomial, then each of the ideals in the product is monomial.  For
$0 \le j \le d$ we conclude that $h^*(\calI_{j,K_0})$ (respectively
$h^*(\calJ_{j,K_0})$) is the sheaf of ideals of an effective divisor $D_j$
(respectively $C_j$) with simple normal crossings; say
\begin{equation*}
 D_j = \sum_{E \in T} N_E^{(j)} E \quad \text{and} \quad
 C_j = \sum_{E \in T} M_E^{(j)} E.
\end{equation*}
Here $T$ denotes the finite set of smooth prime divisors, defined over
$K_0$, occurring in the co-support of $h^*(\calL_{K_0})$.  Let $\disc$
denote the \emph{discrepancy divisor}, that is, the divisor defined by
the image of the canonical map
$h^*(\omega_{\schX_{K_0}}) \otimes \omega_{Y}^\vee \to \calO_Y$, where
$\omega_X$ denotes the canonical bundle on a variety~$X$.  The support
of $\disc$ is contained in~$T$, thus we may write
\begin{equation*}
     \disc = \sum_{E \in T} (\nu_E - 1) E
\end{equation*}
for certain~$\nu_E \in \N$.  These parameters are relevant for our
purposes, as they describe the pullback of the canonical $p$-adic
measure; see~Section~\ref{sec:func-equ}.  For any finite extension $K$
of $K_0$ we obtain a monomialisation $h_K \colon Y_K \to \schX_K$.  As
the divisors in $T$ are smooth, they decompose into a disjoint union
of prime divisors defined over~$K$.  Let $\Ip \in \spec(O_S)$ be a
closed point.  Since $h$ is birational, we may pull back the integral
defining $Z_\Ip$ to $Y$ and we obtain
\begin{equation}\label{eq:zetaIntegral}
  Z_\Ip(s) = \sum_{\ell \in \Z} q_\Ip^{\, -c\ell} \int_{Y(K_\Ip)}
  \left\Vert \left( \pi^{j\ell} h^{-1}(\calI_j) \right)_{j=0}^d
  \right\Vert_\Ip^{\, -s}  \left\Vert \left( \pi^{j\ell}
      h^{-1}(\calJ_j) \right)_{j=0}^d \right\Vert_\Ip^{\, -1} \di h^*(\mu_{\schX, \Ip}).
\end{equation}

We point out two subtleties in connection with the
formula~\eqref{eq:zetaIntegral}.  Firstly, the scheme $\schX$ is
projective and hence $\schX(O_\Ip) = \schX(K_\Ip)$.  Secondly, the
symbol $h^{-1}(\calI_j)$ denotes the sheaf-theoretic inverse image of
the $\calO_\schX$-module~$\calI_j$, which is not an $\calO_Y$-module;
a similar remark applies to $\calJ_j$.

Given an open subset $V \subseteq_\mathrm{o} Y(K_\Ip)$, we define the
\emph{restricted zeta function} $Z_\Ip[V](s)$ by the formula obtained
from~\eqref{eq:zetaIntegral} by restricting the domain of integration to $V$.
We denote by $\alpha_\Ip(V)$ the abscissa of convergence of
$Z_\Ip[V]$.  Observe that
$V \subseteq_\mathrm{o} W \subseteq_\mathrm{o} Y(K_\Ip)$ implies
$\alpha_\Ip(V) \le \alpha_\Ip(W)$.

\begin{dfn}
  Let $K$ be a finite extension of $K_0$ and let $\Ip \in \spec(O_S)$
  be a closed point.  We define the \emph{local abscissa of
    convergence at $y \in Y(K_\Ip)$} by
 \begin{equation*}
   \alpha_\Ip(y) = \inf \{ \alpha_\Ip(V) \mid  y \in V
   \subseteq_\mathrm{o} Y(K_\Ip) \}.
 \end{equation*}
\end{dfn}

\begin{lem} \label{lem:absLocalGlobal} In the above setting, for any
  open compact subset $V \subseteq Y(K_\Ip)$ the abscissa of
  convergence of $Z_\Ip[V]$ is
 \begin{equation*}
  \alpha_\Ip(V) = \sup \{ \alpha_{\Ip}(y) \mid y \in V \}.
 \end{equation*}
\end{lem}

\begin{proof}
  Clearly, $V$ is an open neighbourhood of all of the points that it
  contains, thus
  $\alpha_\Ip(V) \ge \sup \{\alpha_{\Ip}(y) \mid y \in V \}$.
  Conversely, let $\epsilon > 0$.  For every $y \in V$ we find an open
  neighbourhood $W_y$ of $y$ in $V$ such that
  $\alpha_\Ip(W_y) < \alpha_\Ip(y) + \epsilon$.  Since $V$ is compact,
  we find a finite number of points $y_1,\dots, y_n \in V$ such that
  $V = \bigcup \{ W_{y_i} \mid 1 \le i \le n \}$.  We deduce that
  \begin{equation*}
    \alpha_\Ip(V) \le \max \{ \alpha_\Ip(W_{y_i}) \mid 1 \le i \le n \} <
    \sup \{ \alpha_\Ip(y) \mid y \in V \} + \epsilon. 
  \end{equation*}
  As $\epsilon$ tends to~$0$, we obtain
  $\alpha_\Ip(V) \le \sup \{\alpha_{\Ip}(y) \mid y \in V \}$.
\end{proof}

Consider a point $y \in Y(K_\Ip)$.  We use the monomialisation
$ h \colon Y \to \schX_{K_0} $ to compute $Z_\Ip[V]$ for small
coordinate neighbourhoods $V$ of $y$.  Set
$T(y) = \{ E \in T \mid y \in E \}$ and observe that
$u \coloneqq \lvert T(y)\rvert \le m = \dim(Y)$.  Locally at $y$, we
can complement the collection of local equations $g_E$ for
$E \in T(y)$ by $m-u$ further elements to obtain a regular system of
parameters.

Locally at $x \coloneqq h(y) \in \schX(O_\Ip)$, the ideal $\calI_j$,
for $j \in \{ 0, \dots, d\}$, is finitely generated, since $\schX$ is
noetherian.  Moreover, the sheaf $h^{-1}(\calI_j)$ generates the
monomial ideal $h^*(\calI_{j,K_0})$, therefore we obtain
\begin{equation*}
  \Vert h^{-1}(\calI_j) \Vert_\Ip = \lvert w \rvert_\Ip \prod\nolimits_{E \in T(y)} \Big|
  g_E^{N^{(j)}_{E}} \Big|_\Ip
\end{equation*}
locally at $y$ for some regular function $w$ with $w(y) \neq 0$.
Similar considerations apply to the ideals $\calJ_j$.  Therefore, if
we choose $V \cong \Ip^N O_\Ip^{\, m}$ to be a small coordinate
neighbourhood, we obtain
\begin{align*}
  \left\Vert \big( \pi^{j\ell} h^{-1}(\calI_j) \big)_{j=0}^d \right\Vert_\Ip 
  & =  \left\Vert \Big(\pi^{j\ell + \delta_j}   \prod\nolimits_{E \in T(y)}
    g_E^{N^{(j)}_{E}} \Big)_{j=0}^d \right\Vert_\Ip\\
  \left\Vert \big(\pi^{j\ell} h^{-1}(\calJ_j) \big)_{j=0}^d \right\Vert_\Ip
  &=  \left\Vert \Big(\pi^{j\ell + \epsilon_j} \prod\nolimits_{E \in T(y)}
    g_E^{M^{(j)}_{E}} \Big)_{j=0}^d \right\Vert_\Ip.
\end{align*}
for certain integers $\delta_j, \epsilon_j \in \Z$.  The canonical
$\Ip$-adic measure is locally defined by differential forms, thus the
pullback $h^*(\mu_{\schX, \Ip})$ is transformed according to the
discrepancy divisor.  Writing $T(y) = \{E_1,\ldots,E_u\}$, we deduce
that the zeta function $Z_\Ip[V]$ is given by
\begin{align}
  \nonumber Z_\Ip[V](s) & = \sum_{\ell \in \Z} q_\Ip^{-c\ell}
                          \int_{\Ip^N O_\Ip^{\, m}} \, \max_{0 \le j \le
                          d} \Big|
                          \pi^{j\ell + \delta_j}
                          \prod\nolimits_{i=1}^u x_i^{N^{(j)}_{E_i}} \Big|_\Ip^{-s}\\
  \nonumber & \quad \cdot \max_{0 \le j \le d} \Big|  \pi^{j\ell +
              \epsilon_j} \prod\nolimits_{i=1}^u x_i^{M^{(j)}_{E_i}}
              \Big|_\Ip^{-1}
              \left\vert  \pi^{a} \prod\nolimits_{i=1}^u
              x_i^{\nu_{E_i} - 1} \right\vert_\Ip \,\di\mu(\underline{x})\\ 
                        & = q_\Ip^{N(u-m)-a} (1-q^{-1}_\Ip)^u  \,\,
                          \Xi_{N, \boldsymbol{\epsilon},
                          \boldsymbol{\delta}}(q_\Ip^{\, -1},
                          q_\Ip^{\, -s}),  \tag{$\dagger$}\label{eq:localZetaFormula} 
\end{align}
where $\mu$ denotes the normalised Haar measure satisfying
$\mu(O_\Ip^{\, m}) = 1$.  Here
$\Xi_{N, \boldsymbol{\epsilon}, \boldsymbol{\delta}}(Q,t)$ is equal to
\begin{equation*}
  \sum_{\ell \in \Z}\sum_{n \in \Z^u_{\ge N}} 
  Q^{- \left( \min\limits_{0 \le j \le d} (\epsilon_j + (j-c)\ell +
      \sum_{i=1}^u (M_{E_i}^{(j)}-\nu_{E_i}) n_i ) \right)} t^{- \left(
      \min\limits_{0 \le j \le d} ( \delta_j + j\ell + \sum_{i=1}^u N^{(j)}_{E_i} n_j) \right)}
\end{equation*}
and thus a power series of the form studied in
Section~\ref{sec:abscissaAbstract}.  By our general assumptions we
have $N_E^{(0)} = N_E^{(j_0)} = M_E^{(0)} = 0$ for all $E \in T$ and
some $j_0 > 0$; so this series satisfies
Assumption~\ref{ass:linear-forms}.

\begin{pro}\label{prop:localAbscissa}
  \begin{enumerate}[\rm (1)]
  \item\label{eq:existNBH} Every point $y \in Y(K_\Ip)$ has an open
    neighbourhood $V$ such that $\alpha_\Ip(y) = \alpha_\Ip(V)$.
  \item Let $K$ be a finite extension of $K_0$ and, for
    $i \in \{1,2\}$, let $\iota_i \colon K \to K_{\Ip_i}$ be a
    completion.  For $y \in Y(K)$ the local abscissa does not depend
    on the completion, that is
    $\alpha_{\Ip_1}(\iota_1(y)) = \alpha_{\Ip_2}(\iota_2(y))$.
  \item The local abscissa at $y \in Y(K_\Ip)$ depends only on
    $T(y) = \{ E \in T \mid y \in E \}$.
\end{enumerate}
\end{pro}

\begin{proof}
  The coordinate neighbourhoods on which the
  formula~\eqref{eq:localZetaFormula} holds form a neighbourhood base
  at $y$. By Lemma~\ref{lem:XiSeries} the abscissa of convergence does
  not depend on~$N$, hence \eqref{eq:existNBH} holds.  Let $y'$ be
  another point in some completion $K_{\Ip'}$ and assume that
  $T(y') = T(y)$.  On small coordinate neighbourhoods of $y$ and $y'$,
  we have formulae of the form~\eqref{eq:localZetaFormula}.  In these
  formulae, different $N, \boldsymbol{\epsilon}, \boldsymbol{\delta}$
  and $q$ may occur.  However, by Lemma~\ref{lem:XiSeries} these
  parameters do not change the abscissa of convergence and the
  assertions follow.
\end{proof}

\begin{proof}[Proof of Theorem~\ref{thm:abscissa}] For a subset
  $U \subseteq T$ we consider the closed smooth $K_0$-subscheme
  $E_U = \bigcap_{E \in U} E \subseteq Y$ and the open subscheme
  $E^\circ_U = E_U \smallsetminus \bigcup_{E \notin U} E$ of~$E_U$.  The
  closed points $y \in E^\circ_U$ are concisely the ones which satisfy
  $T(y) = U$.  By Proposition~\ref{prop:localAbscissa} the local
  abscissa is constant on all points of $E^\circ_U$ and we denote this
  abscissa by~$\alpha(U)$.

  Since $Y$ is projective, the space $Y(K_\Ip)$ is compact.  Therefore
  it can be covered by a finite number of disjoint open coordinate
  neighbourhoods on which the zeta function is given by a local
  formula as in \eqref{eq:localZetaFormula}.  Now
  Lemma~\ref{lem:XiSeries} implies that $Z_\Ip(s)$ is rational in
  $q_\Ip^{-s}$.  Moreover, the real parts of all poles of $Z_\Ip$ are
  contained in a finite union of finite sets of the form
  $P_{\boldsymbol{\lambda},\boldsymbol{\beta}}$ as in
  Lemma~\ref{lem:XiSeries}.

  Let $\Ip \in \spec(O_S)$ be a closed point.  It follows from
  Lemma~\ref{lem:absLocalGlobal} that the abscissa of convergence of
  $Z_\Ip$ is equal to
  \begin{equation}\label{eq:localAbscissaBySets}
    \alpha_\Ip(Y(K_\Ip)) = \max\{ \alpha(U) \mid U \subseteq T \text{
      and } E^\circ_U(K_\Ip) \neq \varnothing \}. 
  \end{equation}

  Each of the schemes $E_U^\circ$ is of finite type over $K_0$, thus
  there is a finite extension $K_1$ of $K_0$ satisfying
  \begin{equation*}
    E^\circ_U(K_1) \neq \varnothing 
  \end{equation*}
  whenever $E^\circ_U$ is not the empty scheme.  Suppose that
  $K_1 \subseteq K$.  Then \eqref{eq:localAbscissaBySets}
  shows that $\alpha_\Ip(Y(K_\Ip))$ is independent of $K$ and~$\Ip$; thus
  $\alpha(Z_\Ip) = \max P$.
\end{proof}

%%%

\subsection{The functional equation}\label{sec:func-equ}
In this section we sketch the proof of part~\eqref{itm:funct-equ} of
Theorem~\ref{thm:main-summary} leaving out some arguments which
can be found in \cite{DeMe91, Vo10}\footnote{Note added to arXiv preprint: More details are contained in the preceding
arXiv version (v2) of this article.}. We first consider the case
$K =K_0$ and, to simplify the notation, it is convenient to use the
short notation $O_S = O_{K,S}$ etc.\ introduced for~$K$.  Using this
notation we construct in~\eqref{equ:construct-F} the rational function
$F$ featuring in Theorem~\ref{thm:main-summary}.  Only in the last
step we go back to the original general notation to verify the
functional equations for arbitrary extensions $K$ of~$K_0$.

Consider the scheme $\schX$ over $\spec(O_S)$ as in
Section~\ref{sec:p-adic-integrals}. We would like to obtain a
monomialisation of the ideal sheaves $\calI_j$ and $\calJ_j$ over
$\spec(O_S)$.  For fields of positive characteristic a Monomialisation
Theorem is not known and so we should not expect to achieve
monomialisation globally. However, we can obtain a monomialisation on
the generic fibre and expand it to some open set in $\spec(O_S)$.
Indeed, using Hironaka's Monomialisation Theorem and \emph{enlarging
  the set $S$ of places} in a suitable way one obtains a smooth
integral projective scheme $\schY$ over $\spec(O_S)$ and a birational
projective morphism $h \colon \schY \to \schX$ such that all the ideal
sheaves $\widetilde{\calI}_j = h^*(\calI_j)$,
$\widetilde{\calJ}_j= h^*(\calI_j)$ are monomial.
% Enlarging $S$ we can assume that the complement of the union of all
% closed sets defined by the (non-zero) ideal sheaves $\calI_j$ and
% $\calJ_j$ surjects onto $\spec(O_S)$.
%
% \begin{thm}[Generic Monomialisation]
%   Let $\schX$ be a smooth integral projective scheme over $\spec(O_S)$
%   and let $\calI$ be an ideal sheaf on $\schX$.  There are a non-empty
%   open subset $U \subseteq_\mathrm{o} \spec(O_S)$, a smooth integral
%   projective scheme $\schY$ over $U$, and a projective morphism
%   $h\colon \schY \to \schX_U$ such that $h^*(\calI)$ is monomial and
%   $h$ restricts to an isomorphism between the complements of the closed sets
%   defined by~$h^*(\calI)$ and~$\calI$.
% \end{thm}

The fact that $h$ is birational enables us to compute the zeta
integral on $\schY$ instead of $\schX$.  For every closed point
$\Ip \in \spec(O_S)$, the map $h\colon \schY(O_\Ip) \to \schX(O_\Ip)$
is a homeomorphism, ignoring sets of measure~$0$, and consequently
\begin{equation*}
  Z_\Ip(s)  = 
  \sum_{\ell \in \Z} q^{-\ell c }_\Ip \int_{\schY(O_\Ip)} \left\Vert
    \bigl( \pi^{j\ell} \widetilde{\calI}_j \bigr)_{j=0}^d
  \right\Vert_\Ip^{-s} \, \left\Vert \bigl( \pi^{j\ell} \widetilde{\calJ}_j
    \bigr)_{j=0}^d \right\Vert^{-1}_\Ip \, \di h^*(\mu_{\schX, \Ip})
\end{equation*}

For $j \in \{0, \ldots, d \}$ the locally principal ideal sheaves
$\widetilde{\calI}_j$ and $\widetilde{\calJ}_j$, define effective
divisors on $\schY$ which are denoted by $D_j$ and $C_j$ respectively.
% $\schY$ is an integral noetherian separated regular scheme, there
% is, in fact, no difference between Cartier and Weil divisors on
% $\schY$; see \cite[II.6.11]{Ha06}.
Let $T$ denote the set of prime divisors occurring in
$D = \sum_{j=0}^d C_j + D_j$.  As in Section~\ref{sec:ProofAbscissa}
we may write
\[
D_j = \sum_{E \in T} N^{(j)}_{E} E \quad \text{and} \quad C_j =
\sum_{E \in T} M^{(j)}_{E} E, \qquad \text{for
  $j \in \{0, \ldots, d\}$},
\]
where $N^{(j)}_{E}, M_E^{(j)} \in \Z$ are certain non-negative
integers. Note that $M_E^{(0)} = N_E^{(0)} = 0$ for all~$E \in T$.  We
enlarge $S$ such that for every set $U \subseteq T$ of prime divisors,
the intersection $E_U = \bigcap_{E \in U} E$ is a smooth scheme over
$\spec(O_S)$.  For every finite commutative $O_S$-algebra~$L$, we
write $b_U(L) = \lvert E_U(L) \rvert$ for the number of $L$-rational
points of $E_U$.
% We will consider the case where $L$ is a finite extension of a
% residue field $\kappa_\Ip$ of $O_\Ip$.

The pullback $h^*(\mu_{\schX, \Ip})$ of the canonical $\Ip$-adic
measure is $\Vert\calM\Vert_\Ip \, \mu_{\schY, \Ip}$, where $\calM$ is
the locally principal ideal sheaf which defines the discrepancy
divisor $\disc$.
% i.e. $\calM$ is the image of the canonical morphism of
% $\calO_\schY$-modules
% $h^*\omega_\schX \otimes \omega_{\schY}^{\vee} \to \calO_{\schY}$.
The support of $\calM$ is contained in $T$ and so $\calM$ is monomial
and
\begin{equation*}
  \disc = \sum_{E \in T} (\nu_E - 1) E
\end{equation*}
for certain multiplicities $\nu_E -1 \in \N_0$.

\begin{pro}
  In the above situation for almost all closed points
  $\Ip \in \spec(O_S)$ the following formula holds:
 \begin{equation*}
   Z_{\Ip}(s) = q_\Ip^{-m}\sum_{U \subseteq W \subseteq T}
   (-1)^{|W\setminus U|} \: (q_{\Ip} - 1)^{|U|} \: b_W(\kappa_\Ip) \:
   \Xi_U(q^{-1}_\Ip,q_{\Ip}^{-s}),
 \end{equation*}
 where the sum runs over all pairs $(U,W)$ of subsets
 $U \subseteq W \subseteq T$ and
 \begin{equation*}
   \Xi_U(Q,t) = \sum_{\substack{\ell \in \Z \\ (n_E)_E \in \N^{U}}}
   Q^{- \left( \min\limits_{0 \le j \le d} ( (j-c)\ell +  \sum_{E
         \in U} (M^{(j)}_{E} - \nu_{E}) n_E ) \right)} \, t^{- \left(
       \min\limits_{0 \le j \le d} (j\ell + \sum_{E \in U}
       N_E^{(j)}n_E) \right)}.
 \end{equation*}
\end{pro}

\begin{proof}
  By the principle of inclusion and exclusion it suffices to establish
  the equation
  \begin{equation*}
    Z_{\Ip}(s) = q_\Ip^{-m}\sum_{U \subseteq T} \: (q_{\Ip} - 1)^{|U|}
    \: c_U(\kappa_\Ip) \: \Xi_U(q^{-1}_\Ip,q_{\Ip}^{-s}),
  \end{equation*}
  where
  $c_U(L) = \lvert \{a \in \schY(L) \mid \forall E \in T \colon a \in
  E \longleftrightarrow E \in U \} \rvert$.
  To establish this identity one can argue as in the proof of
  \cite[Thm.~2.1]{Vo10}.
\end{proof}

By Lemma~\ref{lem:XiSeries} (1) the formal sum
$\Xi_U(Q,t) \in \Q\LS{q}\PS{t}$ is a rational function in $\Q(Q,t)$.
In addition, Lemma~\ref{lem:XiSeries} (5) implies that the functions
$\Xi$ satisfy the inversion property
\begin{equation}\label{eq:XiFunctional}
  \Xi_U(Q^{-1},t^{-1}) = (-1)^{|U|+1} \sum_{V \subseteq U} \Xi_V(Q,t)
\end{equation}
for all $U \subseteq T$; compare \cite[Cor.~2.2]{Vo10}.

Let $f $ be a positive integer, and let $L_f$ denote the unique
extension of $\kappa_\Ip$ of degree~$f$.  For simplicity we write
$b_U(\Ip,f) = b_U(L_f)$.  We proceed along the lines of~\cite{DeMe91};
see \cite{Da96} for an overview of the relevant results on $\ell$-adic
cohomology.  Write
$\overline{E_U} = E_U \times \overline{\kappa_\Ip}$, where
$\overline{\kappa_\Ip}$ denotes the algebraic closure of~$\kappa_\Ip$.
Then, by the $\ell$-adic Lefschetz fixed point principle, we obtain
\begin{equation*}
  b_U(\Ip,f) = \sum^{2(m-|U|)}_{i=0} (-1)^i \Tr\left(\phi^{f} |
    H^i(\overline{E_U}, \Q_\ell)\right), 
\end{equation*}
where $\phi$ denotes the endomorphism induced by the Frobenius
automorphism.  In particular, we get an expression
\begin{equation*}
    b_U(\Ip, f) = \sum^{2(m-|U|)}_{i=0} (-1)^i \sum_{j =1}^{g(U,i)} \lambda^f_{U,i,j}
\end{equation*}
in terms of the eigenvalues $\lambda_{U,i,j} $, for
$1 \le j \le g(U,i) \coloneqq \dim H^i(\overline{E_U}, \Q_\ell)$, of
$\phi$ acting on the $\ell$-adic cohomology
$H^i(\overline{E_U}, \Q_\ell)$.  Poincar\'e duality in $\ell$-adic
cohomology implies the functional equation of the Weil zeta function
\begin{equation}\label{eq:WeilFE}
  b_U(\Ip,-f) = q_{\Ip}^{-f(m-|U|)} b_U(\Ip,f).
\end{equation}
% By construction, the schemes $E_U$ are smooth and projective (hence
% proper) over $\spec(O_S)$.
By the smooth proper base change theorem \cite[\S6.8]{Da96} the $i$th
$\ell$-adic Betti number is independent of the chosen place~$\Ip$ and
so is the number of Frobenius eigenvalues.

Let $\Upsilon$ be the set of all triples $\upsilon = (U,i,j)$
satisfying $U \subseteq T$, $0 \le i \le m- |U|$ and
$1\le j \le g(U,i)$.  Consider the polynomial ring
$\Q(Y_1,Y_2)[\underline{X}]$
% \coloneqq \Q(Y_1,Y_2)[X_\upsilon \mid \upsilon \in \Upsilon]
over the field of rational functions $\Q(Y_1,Y_2)$, where
$\underline{X}$ denotes the whole collection of independent
variables~$(X_\upsilon)_{\upsilon\in\Upsilon}$.  For every
$U \subseteq T$, we set
\[
b_U(\underline{X}) = \sum^{2(m-|U|)}_{i=0} (-1)^i \sum_{j =1}^{g(U,i)}
X_{U,i,j},
\]
and we define the rational function $F \in \Q(Y_1,Y_2)[\underline{X}]$
as
\begin{equation}\label{equ:construct-F}
  F(Y_1,Y_2,\underline{X})
  = Y_1^m \sum_{U \subseteq W \subseteq T} (-1)^{|W\setminus U|}
  \: (Y_1^{-1} - 1)^{|U|} \: b_W(\underline{X}) \:
  \Xi_U(Y_1,Y_2).
\end{equation}

As explained above, we now revert to the original notation in order to
establish the functional equations for arbitrary finite extension $K$
of~$K_0$.  Let $K$ be a finite extension of $K_0$ of inertia
degree~$f$, and let $\Ip \in \spec(O_S)$ be a closed point over
$\Ip_0 \in \spec(O_{K_0,S})$.  As in the statement of
Theorem~\ref{thm:main-summary}, set $q = q_{\Ip_0}$.  The preceding
discussion yields
\begin{equation*}
  Z_\Ip(s) = F(q^{-f}, q^{-fs}, (\lambda^f_\upsilon)_{\upsilon \in \Upsilon}),
\end{equation*}
where the $\lambda^f_\upsilon$ are the corresponding Frobenius
eigenvalues as above.

It is readily verified that the combinatorial identity
 \begin{equation}\label{eq:combinatorial}
   \sum_{V \subseteq U \subseteq W} (z-1)^{|U|} (-z)^{|W\setminus U|}
   = (-1)^{|W \setminus V|} (z-1)^{|V|}
 \end{equation}
 holds for all finite sets $V \subseteq W$ (where $z$ is an
 indeterminate).  A short calculation as in \cite[p.1197]{Vo10} using
 \eqref{eq:XiFunctional}, \eqref{eq:WeilFE} and
 \eqref{eq:combinatorial} yields the following result which implies
 part~\eqref{itm:funct-equ} of Theorem~\ref{thm:main-summary}.

\begin{thm}\label{thm:functionalEquation}
  In the above setting, the function $F$ satisfies for every positive
  integer $f$ the functional equation
  \begin{equation*}
    F(q^{f}, t^{-f}, (\lambda^{-f}_\upsilon)_{\upsilon \in
      \Upsilon}) = -q^{fm} F(q^{-f}, t^f,
    (\lambda^{f}_\upsilon)_{\upsilon \in
      \Upsilon}), 
  \end{equation*}
  where $t$ is any indeterminate.
\end{thm}

% \begin{proof}
 
%   \begin{align*}
%     F&(q^{f}, t^{-f}, (\lambda^{-f}_\upsilon)_{\upsilon \in
%        \Upsilon}) \\
%      & = q^{fm}\sum_{U \subseteq W \subseteq T} (-1)^{|W\setminus
%        U|} \: (q^{-f} - 1)^{|U|} \: b_W(\Ip,-f) \:
%        \Xi_U(q^{f},t^{-f})\\ 
%      & =
%        q^{fm}\sum_{U \subseteq W \subseteq T} (-1)^{|W\setminus U|} \:
%        (q^{-f} - 1)^{|U|} \: q^{-f(m-|W|)} b_W(\Ip,f) \:
%        \Xi_U(q^{f},t^{-f}) && \text{\tiny by \eqref{eq:WeilFE}} \\ 
%      & = \sum_{V \subseteq U
%        \subseteq W \subseteq T} (-1)^{|W\setminus U|} \: (q^{-f} -
%        1)^{|U|} \: q^{f|W|} b_W(\Ip,f) \: (-1)^{|U|+1}
%        \Xi_V(q^{-f},t^f) && \text{\tiny by \ref{lem:XiFunctional}}\\ 
%      & = - \sum_{V \subseteq U \subseteq W
%        \subseteq T} (-q^f)^{|W\setminus U|} \: (q^{f} - 1)^{|U|} \:
%        b_W(\Ip,f) \: \Xi_V(q^{-f},t^f)\\
%      & = - \sum_{V \subseteq
%        W \subseteq T} (-1)^{|W\setminus V|} \: (q^{f} - 1)^{|V|} \:
%        b_W(\Ip,f) \: \Xi_V(q^{-f},t^f) && \text{\tiny by \ref{lem:combinatorial}}\\ 
%      &  = -q^{fm} F(q^{-f},
%        t^f, (\lambda^{f}_\upsilon)_{\upsilon \in
%        \Upsilon}). \qedhere
%   \end{align*}
% \end{proof}

%%%%%

\section{Groups acting on rooted trees and multiplicity-free~representations} \label{sec:action-on-trees} In this section we
describe a geometric situation which leads to a transparent treatment
of zeta functions associated to certain admissible representations.
In particular, this yields a straightforward description of the zeta
function associated to the induced representation
$\Ind_{P_\xi}^{\GL_{n+1}(\fo)}(\one_{P_\xi})$ from a maximal
$(1,n)$-parabolic subgroup $P_\xi$ to the compact $p$-adic Lie group
$\GL_{n+1}(\fo)$; see Proposition~\ref{prop:formulaInductionTrees} for
details.

Central to the geometric approach is a suitable notion of (weak)
$2$-transitivity.
 
\begin{dfn}
  The action of a group $G$ on a metric space $(X,d)$, from the right
  by isometries, is called \emph{distance-transitive} if for all
  $x_1,x_2,y_1,y_2 \in X$ with $d(x_1,x_2) = d(y_1,y_2)$ there is an
  element $g \in G$ with $x_i g = y_i g$ for $i \in \{1,2\}$.
\end{dfn}

More specifically, if $G$ is a profinite group acting
distance-transitively on a rooted tree~$\calT$ there is a geometric
way to compute the zeta function of the representation of $G$ on the
space $\calC^\infty(\partial\calT,\C)$ of locally constant functions
on the boundary~$\partial\calT$.  The corresponding method was used by
Bartholdi and Grigorchuk~\cite{BarGri2001} for specific groups.  In
\cite{BeHa03}, Bekka and de la Harpe discuss a general approach based
on reproducing kernels; in collaboration with Grigorchuk, they treat
groups acting on rooted trees in detail in an appendix.  We briefly
review the theoretic background and produce a formula for the zeta
function. As an application we compute the zeta functions of induced
representations of certain compact $p$-adic analytic groups that are
\emph{not} torsion-free potent pro-$p$.

%%%

\subsection{Trees and distance-transitive actions}
Let $m = (m_i)_{i\in\N}$ be an integer sequence with $m_i \ge 2$ for
all $i \in \N$.  We define the associated spherically homogeneous
rooted tree $\calT_m$ as follows.  The vertices of level $n\ge 0$ are
finite sequences $w = (w_1,\dots,w_n)$ of length $n$ with
$w_i \in \{0,\dots,m_{i}-1\}$.  We write $\lvert w \rvert = n$ to
indicate that $w$ is of level $n$.  The root of $\calT_m$ is the empty
sequence.  There is an edge between a vertex $v$ of level $n$ and a
vertex $w$ of level $n+1$ if and only if $v$ is a prefix of~$w$.  Let
$L_m(n)$ denote the set of vertices of level $n$ in~$\calT_m$.  For
any two vertices $v$ and $w$ we write $\pre(v,w)$ for the longest
common prefix of $v$ and $w$, regarded as a vertex of~$\calT_m$.

The boundary $\partial\calT_m$ of $\calT_m$ is the space of geodesic
rays in $\calT_m$ starting at the root, that is, the profinite
topological space of infinite sequences $(\xi_i)_{i\in \N}$ with
$\xi_i \in \{0,\dots,m_{i}-1\}$.  For any $n \in \N_0$, we write
$\xi(n)$ for the prefix of $\xi$ of length $n$ and we say that the
geodesic ray $\xi$ passes through $\xi(n)$.  The longest common prefix
of $\xi,\eta \in \partial\calT_m$ is denoted $\pre(\xi,\eta)$, which
is a vertex of $\calT_m$ unless $\xi = \eta$.  There is a natural
metric $d$ on $\partial\calT_m$ given by
\begin{equation*}
  d(\xi,\eta) = \frac{1}{1+|\pre(\xi,\eta)|} \qquad \text{for all
    $\xi, \eta \in \partial\calT_m$ with $\xi \neq \eta$.}
\end{equation*}
The same formula defines also a metric on $L_m(n)$.

Let $\Aut(\calT_m)$ denote the profinite group of rooted automorphisms
of $\calT_m$.

\begin{lem}
  A closed subgroup $G \le_\mathrm{c} \Aut(\calT_m)$ acts distance-
  transitively on $\partial\calT_m$ if and only if it acts
  distance-transitively on each layer $L_m(n)$, $n \in \N_0$.
\end{lem}

\begin{proof}
  First suppose that $G \le_\mathrm{c} \Aut(\calT_m)$ acts
  distance-transitively on the boundary~$\partial\calT_m$.  Let
  $n \in \N_0$ and, for $i \in \{1,2\}$, let $x_i,y_i \in L_m(n)$ such
  that $d(x_1,x_2) = d(y_1,y_2)$.  Choose geodesic rays $\xi_i$
  passing through $x_i$ and $\eta_i$ passing through $y_i$ with
  $d(\xi_1,\xi_2) = d(\eta_1,\eta_2)$.  Then there exists $g \in G$
  such that $\xi_ig = \eta_i$ for $i \in \{1,2\}$, and hence
  $x_ig = y_i$.
 
  Conversely, suppose that $G$ acts distance-transitively on each
  layer.  For $i \in \{1,2\}$, consider
  $\xi_i,\eta_i \in \partial\calT_m $ such that
  $d(\xi_1,\xi_2) = d(\eta_1,\eta_2)$.  For each $n \in \N_0$, the set
  $K_n = \{ g \in G \mid \xi_i(n)g = \eta_i(n) \text{ for } i \in
  \{1,2\}\} \subseteq G$
  is closed and non-empty. As $G$ is compact, we find an element
  $g \in \bigcap_{n} K_n$.  Evidently, $\xi_ig = \eta_i$ for $i \in \{1,2\}$.
\end{proof}

Let $G \le_\mathrm{c} \Aut(\calT_m)$ act distance-transitively on
$\partial\calT_m$ and consider the space
$\calC^\infty(\partial\calT_m,\C)$ of locally constant complex-valued
functions on the boundary.  We study the representation
$\rho_\partial$ of $G$ on $\calC^\infty(\partial\calT_m,\C)$ defined
by
\begin{equation*}
  (\rho_\partial(g).f)(x) = ({}^g f)(x) = f(xg) \qquad \text{ for all
  } g\in G \text { and } f \in
  \calC^\infty(\partial\calT_m,\C).
\end{equation*}
Let $\xi \in \partial\calT_m$ be a point on the boundary and let
$P = P_\xi = \stab_G(\xi)$ denote the stabiliser in $G$; the group $P$
is called a \emph{parabolic subgroup} of~$G$.  As $G$ acts
transitively on the boundary, it follows that $\rho_\partial$ is
isomorphic to the induced representation~$\Ind_P^G(\one_P)$.  We
recall that $(G,P)$ is called a \emph{Gelfand pair} if
$\Ind_P^G(\one_P)$ is multiplicity-free or, equivalently, if the
convolution algebra of $P$-bi-invariant locally constant functions on
$G$ is commutative; compare~\cite[\S 45]{Bu13} or \cite{Gr91}.  The
following theorem, which corresponds to Theorem~\ref{thmABC:E} in the
introduction, is essentially due to Bekka, de la Harpe, and
Grigorchuk; see~\cite[Prop.~10]{BeHa03}.  For completeness we include
a short proof.

\begin{thm}\label{thm:distanceTransitiveFormula}
  Let $G \le_\mathrm{c} \Aut(\calT_m)$ act distance-transitively on
  $\partial\calT_m$ and let $P = P_\xi \le_\mathrm{c} G$ be a
  parabolic subgroup.  Then $(G,P)$ is a Gelfand pair and the
  representation $\rho_\partial$ decomposes as a direct sum of the
  trivial representation and a unique irreducible constituent $\pi_n$
  of dimension $(m_n-1)\prod_{j=1}^{n-1} m_j$ for each $n\ge 1$.  In
  particular, the zeta function of the representation $\rho_\partial$
  is
  \begin{equation*}
    \zeta_{\rho_\partial}(s) = 1 + \sum_{i=1}^\infty (m_i-1)^{-s}
    \prod_{j=1}^{i-1} m_j^{\, -s} 
   \end{equation*}
   with abscissa of convergence $\alpha(\rho_\partial) =  0$.
\end{thm}

\begin{proof}
  For $n \in \N_0$, let $P_n \le_\mathrm{o} G$ be the stabiliser of the
  vertex $\xi(n)$ for the action of $G$ on the layer~$L_m(n)$.
  Observe that
  $\Ind_P^G(\one_P) = \varinjlim \Ind_{P_n}^G(\one_{P_n})$, where
  $\Ind_{P_n}^G(\one_{P_n})$ is simply the permutation representation
  $\C[L_m(n)]$ of $G$ on the vertices of level $n$.
   
  We claim that $\dim_\C \End_{G}(\C[L_m(n)]) = n+1$.  From this it
  follows, for $n \ge1$, that
  $\C[L_m(n)] = \C[L_m(n-1)] \oplus V_{\pi_{n}}$ for some
  $\pi_{n} \in \Irr(G)$ that does not yet occur as a constituent of
  $\C[L_m(n-1)]$ and satisfies
  \begin{equation*}
    \dim_\C \pi_{n} = \lvert L_m(n) \rvert - \lvert L_m(n-1) \rvert = (m_{n}-1)
    \prod\nolimits_{j=1}^{n-1} m_j. 
  \end{equation*}   

   It remains to establish the claim.  Frobenius reciprocity
   implies
   \begin{equation*}
     \dim_\C \End_{G}(\C[L_m(n)]) = \dim_\C \Hom_{P_n}(\one_{P_n},
     \C[L_m(n)]) = \lvert L_m(n)/P_n \rvert. 
   \end{equation*}
   As $G$ acts distance-transitively, the $P_n$-orbits in $L_m(n)$
   correspond bijectively to the possible distances from the
   vertex~$\xi(n)$.  Finally, the number of prefixes of any vertex of
   level $n$ is exactly $n+1$.
\end{proof}

\begin{rmk}
  Theorem~\ref{thm:distanceTransitiveFormula} applies, in particular,
  to $G = \Aut(\calT_m)$.  Variation of the sequence
  $m = (m_i)_{i\in \N}$ allows one to construct induced
  representations with various different zeta functions.  For
  instance, if we apply the theorem to the $d$-regular tree $\calT_m$
  with $m_i=d$ for all $i$ and some fixed integer $d \geq 2$, then the
  zeta function of the representation on the boundary is
  $\zeta_{\rho_\partial}(s) = 1+ \frac{(d-1)^{-s}}{1-d^{-s}}$ which
  admits a meromorphic continuation to the complex plane.  By
  contrast, suppose that the sequence $(m_i)_{i\in\N}$ tends to
  infinity with $i$. In this case the zeta function
  $\zeta_{\rho_\partial}(s)$ cannot be extended analytically and the
  vertical axis ${\rm{Re}}(s) = 0$ is the natural boundary; see
  \cite[Thm.~VI.2.2]{Ma72}.
\end{rmk}
%%%

\subsection{Induction from maximal $(1,n)$-parabolic
    subgroups to $\GL_{n+1}(\Delta)$}

In this section $\fo$ denotes a complete discrete valuation ring with
maximal ideal~$\Ip$ and finite residue field $\kappa = \fo/\Ip$ of
cardinality $q$.  Let $\pi \in \Ip$ denote a uniformiser.  We emphasise
that here $\fo$ may have positive characteristic, e.g.,
$\fo = \F_p\PS{t}$.

Consider a central division algebra $\fd$ of index $d$ over the
fraction field~$\ff$ of $\fo$ so that $\dim_\ff \fd = d^2$.  It is
known that $\fd$ contains a unique maximal $\fo$-order~$\Delta$.  Up
to isomorphism, $\fd$ and $\Delta$ can be described explicitly, in
terms of the index $d$ and a second invariant $r$ satisfying
$1 \le r \le d$ and $\gcd(r,d)=1$; see~\cite[\S14]{Re03} and the more
explicit description given in Section~\ref{sec:schiefkoerper}.  Here
it suffices to fix a uniformiser $\Pi \in \fd$ so that $\Delta$ has
maximal ideal $\IP = \Pi \Delta$.  Recall that every element of
$\Delta$ can be written as a converging power series
$\sum_{k=0}^\infty c_k \Pi^k$ with coefficients coming from any set of
representatives for the elements of the residue field $\Delta / \IP$,
which is a finite field of size~$q^d$.

Fix $n \in \N$, and consider the general linear group
$\GL_{n+1}(\Delta)$ with a maximal parabolic subgroup $P_\xi$ which is
the stabiliser of a point $\xi \in \mathbb{P}^n(\Delta)$ under the right
$\GL_{n+1}$-action.  For instance, if
$\xi_0 = (1 \!:\! 0 \!:\! \ldots \!:\! 0)$, then
\begin{equation*}
   P_{\xi_0}  =  
   \begin{pmatrix}[c|ccc]
     \GL_1(\Delta) & 0 & \cdots &  0\\ \hline
     \Delta & & & \\
     \vdots & & \GL_n(\Delta) & \\
     \Delta & & & \\
   \end{pmatrix} 
   \le \GL_{n+1}(\Delta).
\end{equation*}

\begin{pro}\label{prop:formulaInductionTrees}
  For all $\xi \in \mathbb{P}^n(\Delta)$ the pair
  $(\GL_{n+1}(\Delta),P_\xi)$ is a Gelfand pair and the induced
  representation
  $\rho = \Ind_{P_\xi}^{\GL_{n+1}(\Delta)}(\one_{P_\xi})$ has the zeta
  function
  \begin{equation*}
    \zeta_\rho(s) = 1 + \Big( \frac{q^{dn} - 1}{q^d - 1} \Big)^{-s} \left(
    q^{-ds} + \frac{(q^{d(n+1)}-1)^{-s}}{1-q^{-dns}} \right). 
  \end{equation*}
\end{pro}

\begin{proof}
  Consider the (left) projective $n$-space
  $\mathbb{P}^n(\Delta)$.  Concretely, we describe elements $x$ in
  terms of homogeneous coordinates $x = (x_0 \!:\! \ldots \!:\! x_n)$,
  where $(x_0, \ldots, x_n) \in \Delta^n$ is \emph{primitive}, i.e.,
  has at least one entry not contained in~$\IP$, and
  $(x_0 \!:\! \ldots \!:\! x_n) = (zx_0 \!:\! \ldots \!:\! zx_n )$ for
  all units $z \in \Delta^\times$.

  We construct a rooted tree $\calT$ representing
  $\mathbb{P}^n(\Delta)$ as follows.  The root $\epsilon$ is the
  unique vertex of level $0$ and, for each $k \in \N$, the vertices of
  level $k$ are the points of $\mathbb{P}^n(\Delta / \IP^k)$.  There
  is an edge between the root and every vertex of level~$1$.
  Moreover, a vertex $x\in \mathbb{P}^n(\Delta/\IP^k)$ of level $k$
  and a vertex $y \in \mathbb{P}^n(\Delta/\IP^{k+1})$ of level $k+1$
  are connected by an edge if and only if $x \equiv y \bmod \IP^k$.
  The boundary of $\calT$ is
  $\partial \calT \cong \mathbb{P}^n(\Delta)$ since $\Delta$ is
  complete.  Note that $\calT$ is spherically homogeneous and
  $\calT \cong \calT_m$ for
  $m_1 = \lvert \mathbb{P}^n(\Delta/\IP) \rvert = (q^{d(n+1)}-1)/(q^d-1)$ and
  $m_k = q^{dn}$ for $k\ge 2$.

  The group $\GL_{n+1}(\Delta)$ acts on the tree $\calT$ from the
  right.  We show below that the action on the boundary is
  distance-transitive.  Consequently,
  Theorem~\ref{thm:distanceTransitiveFormula} implies that
  $(\GL_{n+1}(\Delta),P_\xi)$ is a Gelfand pair and
  \begin{align*}
    \zeta_\rho(s) %
    & = 1 + ( \lvert \mathbb{P}^n(\Delta/\IP) \rvert - 1 )^{-s} +
      \lvert \mathbb{P}^n(\Delta/\IP) \rvert^{-s} (q^{dn}-1)^{-s}
      \sum_{k=0}^\infty q^{-dnks} \\ 
    & = 1 + \Big( \frac{q^{d(n+1)}-q^d}{q^d-1} \Big)^{-s} +
      \Big( \frac{q^{d(n+1)}-1}{q^d-1} \Big)^{-s}
      (q^{dn}-1)^{-s} \frac{1}{1-q^{-dns}} \\ 
    & = 1 + \Big( \frac{q^{dn} - 1}{q^d - 1} \Big)^{-s} \Big(
      q^{-ds} + \frac{(q^{d(n+1)}-1)^{-s}}{1-q^{-dns}} \Big). 
  \end{align*}

  It remains to show that $\GL_{n+1}(\Delta)$ acts
  distance-transitively on~$\mathbb{P}^n(\Delta)$.  Since
  $\GL_{n+1}(\Delta)$ acts transitively on $\mathbb{P}^n(\Delta)$, it
  suffices to show that the group $P_0 = P_{\xi_0}$ acts transitively
  on the spheres around
  $\xi_0 = (1 \!:\! 0 \!:\! \ldots \!:\! 0) \in \mathbb{P}^n(\Delta)$.
  Let $x,y \in \mathbb{P}^n(\Delta)$ lie on such a sphere, that is,
  $x \equiv \xi_0 \equiv y \bmod \IP^{k}$ but
  $x \not \equiv \xi_0 \bmod \IP^{k+1}$ and
  $y \not\equiv \xi_0 \bmod \IP^{k+1}$ for some $k\ge 0$.  Note that,
  for $k=0$, the first congruence holds trivially for all $x,y$.

  We need to find $g \in P_0$ such that $x g = y$.  We may write
  \[
  x = (x_0 \!:\! \Pi^k u_1 \!:\! \ldots \!:\! \Pi^k u_n) \qquad \text{and} \qquad
  y = (y_0 \!:\! \Pi^k v_1 \!:\! \ldots \!:\! \Pi^k v_n),
  \]
  where $x_0 \equiv y_0 \equiv 1 \bmod{\IP^k}$ and
  $u = (u_1, \ldots, u_n), v = (v_1, \dots, v_n) \in \Delta^n$ are
  primitive.  There is a matrix $h \in \GL_n(\Delta)$ such that $u h = v$,
  since $\GL_n(\Delta)$ acts transitively on primitive vectors.
  Furthermore, we find $a_1,\dots, a_n \in \Delta$ such that
  $\Pi^k \sum_{i=1}^n u_i a_i = y_0 - x_0$.  Thus
  \begin{equation*}
    g = \begin{pmatrix}[c|ccc]
      1 & 0 & \cdots & 0\\ \hline
      a_1 & & & \\
      \vdots & & h & \\
      a_n & & &\\
    \end{pmatrix} \in P_0
  \end{equation*}
  has the property that~$xg = y$.
\end{proof}

\begin{rmk}
  Distance-transitivity is, in general, a strong requirement and it
  would be desirable to have similar methods available under weaker
  assumptions.  For instance, to compute zeta functions of
  representations induced from other maximal parabolic subgroups to
  $\GL_{n+1}(\fo)$ (or even $\GL_{n+1}(\Delta)$), one would like to
  take advantage of the action of $\GL_{n+1}(\fo)$ on a tree
  constructed from a suitable Grassmannian scheme. However, this
  action is not distance-transitive since the general linear group
  $\GL_{n+1}(\F_q)$ does not act $2$-transitively on the set of
  $r$-dimensional subspaces of $\F_q^{n+1}$ whenever $1<r<n$.  By
  contrast, we can use the Kirillov orbit method to discuss induction
  from maximal parabolic subgroups in
  Section~\ref{sec:ind-maximal-parabolic}, under the additional
  assumption that $\fo$ has characteristic~$0$; compare
  Problem~\ref{problem:function-fields}.  In~\cite{Ki18} a weak form
  of distance-transitivity and a strengthening of our
  Theorem~\ref{thm:distanceTransitiveFormula} are discussed.
\end{rmk}

%%%%%

\section{Examples of induced representations of potent pro-$p$
  groups} \label{sec:examples-orbit-method} Let $\fo$ be the ring of
integers of a $p$-adic field $\ff$, with maximal
ideal~$\Ip \trianglelefteq \fo$ and finite residue field
$\kappa = \fo/\Ip$ of cardinality~$q$. Let $\pi \in \Ip$ denote a
uniformiser.

%%%

\subsection{Induction from a Borel subgroup to
  $\GL_3^r(\fo)$}\label{sec:GL3}
We consider the general linear group $G = \GL_3(\fo)$ and, for
$r \in \N$, its principal congruence subgroup
\begin{equation*}
   G^r = \ker\bigl( \GL_3(\fo) \to \GL_3(\fo/\Ip^r) \bigr).
\end{equation*}
Let $B \le_\mathrm{c} G$ be the Borel subgroup consisting of
upper-triangular matrices, and set $B^r = B \cap G^r$.  Our aim is to
compute the zeta function of the induced representation
$\Ind_{B^r}^{G^r}(\one_{B^r})$.  We remark that the zeta function of
the induced representation $\Ind_B^G(\one_B)$ has already been
computed by Onn and Singla~\cite[Thm.~6.5]{OnSi14}, using direct
representation-theoretic considerations.

Let $\fg = \gl_3(\fo)$ denote the $\fo$-Lie lattice of
$3 \times 3$-matrices and let $\fb$ denote the sublattice of all
upper-triangular matrices.  The groups $G^r$ and $B^r$ are finitely
generated torsion-free potent pro-$p$ groups whenever $r \ge 2e$ for
$p=2$ and $r \ge e (p-2)^{-1}$ for $p>2$, where $e$ denotes the
ramification index of~$\ff$ over~$\Q_p$; compare
\cite[Prop.~2.3]{AvKlOnVo13}.  In this case $\pi^r \fg$ (respectively
$\pi^r \fb$) regarded as a $\Z_p$-Lie lattice, is the Lie lattice
associated to $G^r$ (respectively~$B^r$).

\begin{pro} \label{pro:ind-Br-Gr} In the set-up described above and
  subject to $r \geq 2e$ for $p=2$ and $r \ge e (p-2)^{-1}$ for $p>2$,
  the zeta function of the representation
  $\rho = \Ind_{B^r}^{G^r}(\one_{B^r})$ is
  \begin{equation*}
    \zeta_\rho(s) = q^{3(r-1)} \frac{q^{1-3s}(q - q^{-s})(u(q^s) -
      q u(q^{-s}))}{(1-q^{-s})(1-q^{1-6s})}, 
  \end{equation*}
  where $u(X) = X^2 - 2X + 1 -2X^{-1} + X^{-2} - X^{-3}$.  

  The abscissa of convergence is $\alpha(\rho) = \nicefrac{1}{6}$.
\end{pro}

It would be interesting to find an interpretation of the value
$\nicefrac{1}{6}$ for the abscissa of convergence; compare the
discussion in the introduction.  Note that from the given formula it
is easy to verify that $\zeta_\rho$ satisfies a functional equation,
in accordance with Theorem~\ref{thm:main-summary}.

\begin{proof}[Proof of Proposition~\ref{pro:ind-Br-Gr}]
  We indicate how the result can be obtained, by applying the
  method described in Section~\ref{sec:LieLattices}.  As most steps
  consist of simple computations we leave the verification of some
  details to the reader.

  Consider the $\fo$-basis $(E_{i,j})_{i,j=1}^3$ of $\fg$ which
  consists of all elementary matrices; i.e., the only non-zero entry
  of $E_{i,j}$ is a $1$ in position~$(i,j)$.  The subalgebra $\fk$ of
  strictly lower triangular matrices is a complement of $\fb$ in
  $\fg$, and the matrices $E_{i,j}$ with $j<i$ form a basis of $\fk$.

  Using the commutator relations
  $[E_{i,j},E_{s,t}] = \delta_{j,s} E_{i,t} - \delta_{i,t} E_{s,j}$,
  one can determine the commutator matrix
  $\calR(\mathbf{T}) \in \Mat_9(\Z[T_{i,j} \mid 1 \le i,j \le
  3])$,
  introduced in Section~\ref{sec:LieLattices}, for a set of variables
  $T_{i,j}$ corresponding to the chosen basis of elementary
  matrices~$E_{i,j}$.  Substituting $T_{2,1} = x$, $T_{3,2} = y$,
  $T_{3,1} = z$, and $T_{i,j} = 0$ whenever $i \le j$ we obtain the
  \emph{reduced} commutator matrix
  \begin{equation*}
    \overline{\calR}(x,y,z) = \begin{pmatrix}[ccc|cccccc]
      0 & 0 & 0 & z & 0 & -z & -x & y & 0 \\
      0 & 0 & -z & x & -x & 0 & 0 & 0 & 0 \\
      0 & z & 0 & 0 & y & -y & 0 & 0 & 0 \\ \hline
      -z & -x & 0 & 0 & 0 & \cdots &  & \cdots & 0 \\
      0 & x & -y & 0 & 0 & \cdots &  & \cdots & 0 \\
      z & 0 & y & \vdots & \vdots & \ddots &  &  & \vdots \\
      x & 0 & 0 &  &  &  &  &  &  \\
      -y & 0 & 0 & \vdots & \vdots &  &  & \ddots & \vdots \\
      0 & 0 & 0 & 0 & 0 & \cdots & & \cdots & 0 \\
    \end{pmatrix}.
  \end{equation*}
  An easy (but, if done by hand, lengthy) calculation yields the sets
  $F_k$ of degree-$k$ Pfaffians.  Removing those Pfaffians whose
  modulus cannot dominate the others we obtain
  \begin{align*}
    &\| F_0(x,y,z) \|_\Ip = 1, && \| F_1(x,y,z) \|_\Ip = \|x,y,z\|_\Ip, \\
    &\| F_2(x,y,z) \|_\Ip = \| x^2, y^2, z^2 \|_\Ip, && \| F_3(x,y,z) \|_\Ip = \|x^2y,
                                                        xy^2 \|_\Ip.
  \end{align*}
  Hence Proposition~\ref{pro:zeta-integral-formula-coord} yields the
  formula
  \begin{multline} \label{equ:our-integral-vs-Duke} \zeta_\rho(s) =
    q^{3r} \int_{\ff^3} \| 1, x^2,y^2,z^2,
    x^2y, xy^2\|_\Ip^{-1-s} \,\di \mu(x,y,z) \\
    = q^{3r} + q^{3r}(1-q^{-1}) \sum_{\ell = 1}^\infty q^{3\ell}
    \underbrace{\int_{\mathbb{P}^2(\fo)} \| \pi^{-2\ell}, \pi^{-3\ell}
      \calI_3 \|_\Ip^{-1-s} \,\di \mu_{\mathbb{P}^2,\Ip}}_{(\ast)},
  \end{multline}
  where $\calI_3$ is the sheaf of ideals on $\mathbb{P}^2/\fo$
  generated by the image of $\calO(3)^\vee = \calO(-3)$ after pairing
  with the sections $x^2y$ and $xy^2$ of the $3$rd twisting
  sheaf~$\calO(3)$.

  The fibres of the reduction map
  $\mathbb{P}^2(\fo) \to \mathbb{P}^2(\kappa)$ admit global charts and
  via these charts the contribution to the integral $(\ast)$ from any
  one fibre can be written as an integral over~$\pi\fo^2$.  Thus it
  remains to compute the integral over every fibre.

  Let $(\bar{x}\!:\!\bar{y}\!:\!\bar{z}) \in \mathbb{P}^2(\kappa)$ be
  a point in projective coordinates.  (a) There are $q(q-1)$ points
  satisfying $\bar{x}, \bar{y} \neq 0$.  In these fibres the integral
  has constant value $q^{-2} q^{-3\ell - 3\ell s}$ and working out the
  geometric series, we obtain the overall contribution
  \[
  q^{3r} (1-q^{-1}) \cdot q(q-1) \cdot q^{-2} \sum_{\ell=1}^\infty q^{-3\ell s} =
  q^{3r}  (1-q^{-1})^2 \frac{q^{-3s}}{1-q^{-3s}}
  \]
  to~$\zeta_\rho(s)$.  (b) There are $q$ points satisfying
  $\bar{x} \neq 0 $ and $\bar{y} = 0$.  In these fibres one can compute
  the integral by distinction of the cases
  $\lvert y \rvert_\Ip \ge q^{-\ell}$ and
  $\lvert y \rvert_\Ip < q^{-\ell}$.  Similarly, there are $q$ points
  satisfying $\bar{x} = 0$ and $\bar{y} \neq 0$; these points yield
  the same contributions, because the integral formula is symmetric in
  $x$ and~$y$.  We obtain the overall contribution
  \[
  q^{3r} 2 (1-q^{-1}) \frac{q^{-2s}-q^{-1-5s}}{(1-q^{-3s})(1-q^{-2s})}
  \]
  to~$\zeta_\rho(s)$.  (c) Finally, the integral over the fibre of the
  remaining point $(0\!:\!0\!:\!1)$ is more intricate: the resulting
  $3$-dimensional summation, over $\ell$ and the valuations of $x$ and
  $y$, can be calculated by choosing a suitable cone decomposition of
  the parameter space. For example, we first considered the case where
  $x$ and $y$ have the same valuation and then used the symmetry in
  $x$ and $y$ to treat the remaining cases. Skipping the details, we
  record the overall contribution
  \begin{multline*}
    q^{3r} (1-q^{-1})^3 \left(\tfrac{q^{-1-2s} + q^{-4s} +
        q^{1-6s}}{(1-q^{-2})(1-q^{1-6s})} +
      \tfrac{q^{1-9s}}{(1-q^{-3s})(1-q^{1-6s})} \right. \\
      \left. \ + \tfrac{2q^{-2-2s}}{(1-q^{-2})(1-q^{-1})(1-q^{-2s})} 
  \ + \tfrac{2(q^{-1-4s} + q^{-6s}+
        q^{-1-8s})}{(1-q^{-2})(1-q^{-2s})(1-q^{1-6s})} +
      \tfrac{2q^{1-8s}}{(1-q^{-2s})(1-q^{-3s})(1-q^{1-6s})} \right)
  \end{multline*}
  to~$\zeta_\rho(s)$. Addition and simplification yields the explicit
  formula.
 %
  % \begin{align*}
  %   q^{-3r}\zeta_\rho(s) = 1 & + (1- \frac{1}{q})^2 \frac{q^{-3s}}{1-q^{-3s}} \quad&&\text{ (a) }\\
  %                            & + 2 (1-\frac{1}{q})\frac{q^{-2s}-q^{-1-5s}}{(1-q^{-3s})(1-q^{-2s})} &&\text{ (b) }\\
  %                            & + (1-\frac{1}{q})^3 \left(\frac{q^{-1-2s} + q^{-4s} + q^{1-6s}}{(1-q^{-2})(1-q^{1-6s})} + \frac{q^{1-9s}}{(1-q^{-3s})(1-q^{1-6s})} \right) &&\text{(c1)}\\
  %                            & + (1-\frac{1}{q})^3 \left(\frac{2q^{-2-2s}}{(1-q^{-2})(1-q^{-1})(1-q^{-2s})} + \frac{2(q^{-1-4s} + q^{-6s}+ q^{-1-8s})}{(1-q^{-2})(1-q^{-2s})(1-q^{1-6s})}
  %                              + \frac{2q^{1-8s}}{(1-q^{-2s})(1-q^{-3s})(1-q^{1-6s})}\right)
  %          \end{align*}
%
%
                    %%%& + (1-\frac{1}{q})\frac{(-q^{1-8s} - q^{1-6s} + q^{-9s} + 2q^{-8s} + q^{-7s} +
                    %%% q^{-6s} - q^{-4s} - q^{-1-11s} - q^{-1-9s} - q^{-1-7s} + q^{-1-5s} + q^{-1-4s} - q^{-1-2s})}{(1-q^{-2s})(1-q^{-3s})(1-q^{1-6s})} &&\text{(c)}
  %\end{align*}
\end{proof}

\begin{rmk}
  The referee encouraged us to indicate how the integral
  description~\eqref{equ:our-integral-vs-Duke} compares to the
  methodology from~\cite[Sec.~3.2]{AvKlOnVo13}.  This requires juggling
  two sets of notation; watch out.  Following~\cite{AvKlOnVo13}, we
  would derive from Proposition~\ref{pro:zeta-algebraic-formula} that
  \[
  \zeta_\rho(s) = q^{d_2 r}
  \mathcal{P}_{\overline{\calR},\fo}(s+1) = q^{3r}
  \mathcal{P}_{\overline{\calR},\fo}(s+1),
  \]
  where $d_2 = 3$ is the dimension of the space
  $\{ w \in \Hom_\fo(\fg,\ff) \mid w(\fb) \subseteq \fo \}$ and
  $\mathcal{P}_{\overline{\calR},\fo}(s)$ is a Poincar\'e series
  associated to the $d_1 \times d_1$ matrix $\overline{\calR}(x,y,z)$,
  for $d_1 = 9$, that one defines in analogy
  to~\cite[(3.3)]{AvKlOnVo13}.  From \cite[(3.4) and
  (3.5)]{AvKlOnVo13}, we would deduce that
  \begin{align*}
    \zeta_\rho(s) & = q^{d_2 r} \big( 1 + (1-q^{-1})^{-1}
                    \mathcal{Z}_\fo'(-s-1,\rho s - d_2 - 1)\big) \\
                  & = q^{3 r} \big( 1 + (1-q^{-1})^{-1}
                    \mathcal{Z}_\fo'(-s-1,3 s - 4)\big),
  \end{align*}
  where $\rho = 3$ is defined as in~\cite[(3.6)]{AvKlOnVo13} and
  \begin{align*}
    \mathcal{Z}_\fo'(r,t) %
    & = \int_{(\alpha,x,y,z) \in \Ip \times W(\fo)}
      \lvert \alpha \rvert_{\Ip}^{\, t} \, \prod_{j=1}^\rho \frac{\Vert
      F_j(x,y,z) \cup \alpha^2 F_{j-1}(x,y,z) \Vert_{\Ip}^{\, r}}{\Vert
      F_{j-1}(x,y,z) \Vert_{\Ip}^{\, r}} \,\di \mu(\alpha,x,y,z) \\
    & = \int_{\alpha \in \Ip} \lvert \alpha \rvert_{\Ip}^{\, t}
      \left( \int_{(x,y,z) \in W(\fo)} \Vert \alpha , x^2 y, x y^2
      \Vert_{\Ip}^{\, r} \,\di \mu(x,y,z) \right) \di \mu(\alpha) 
  \end{align*}
  is a slight adaptation of~\cite[(3.5)]{AvKlOnVo13}, given that we
  have already defined the sets $F_k$ of degree-$k$ Pfaffians; the
  region of integration involves
  $W(\fo) = (\fo^{d_2})^* = \fo^3 \smallsetminus \pi \fo^3$, as
  in~\cite{AvKlOnVo13}.  Carrying out the integration over $\alpha$
  would lead to
  \begin{align*}
    \zeta_\rho(s) %
    & = q^{3r} + q^{3r} \sum_{l=1}^\infty q^{-3ls} \int_{(x,y,z) \in
      W(\fo)} \Vert \pi^l , x^2 y, x y^2 \Vert_{\Ip}^{-1-s} \,\di
      \mu(x,y,z) \\
    & = q^{3r} + q^{3r} \sum_{l=1}^\infty q^{3l}  \int_{(x,y,z) \in
      W(\fo)} \Vert \pi^{-2l}, \pi^{-3l} x^2 y, \pi^{-3l} x y^2 \Vert_{\Ip}^{-1-s} \,\di
      \mu(x,y,z), 
  \end{align*}
  where $\mu(\fo^3) = 1$.  This is
  indeed~\eqref{equ:our-integral-vs-Duke}, written in affine
  coordinates.

  The concrete example illustrates how one can translate quite
  generally between our approach and the methodology used
  in~\cite{AvKlOnVo13} and elsewhere.
 \end{rmk}

%%%

\subsection{Induction from a Borel subgroup to
  $\Un_3^r(\fo)$}\label{sec:U3}

Let $\fe$ be a quadratic extension field of $\ff$
and let $\fo_\fe$ denote the ring of integers of
$\fe$.  Denote the non-trivial Galois automorphism of
$\fe$ over $\ff$ by $\sigma$.

Let $\Un_3$ be the unitary group scheme over $\fo$ associated to the
extension $\fo_\fe$ and the non-degenerate hermitian matrix
$W = \left(\begin{smallmatrix}
    0 & 0 & 1 \\
    0 & 1 & 0 \\
    1 & 0 & 0 \\
  \end{smallmatrix}\right).$
We study the group
\[
G = \Un_3(\fo) = \{ g \in \GL_3(\fo_\fe) \mid
\sigma(g)^\mathrm{tr} \, W \, g = W \}
\]
of $\fo$-rational points and its principal congruence subgroups
$G^r = G \cap \GL_3^r(\fo_\fe)$ as in
Section~\ref{sec:GL3}. Further we write $B$ and
$B^r = B \cap \GL_3^r(\fo_\fe)$ to denote the Borel subgroups
of upper-triangular matrices in $G$ and $G^r$ respectively.  Consider
the $\fo$-Lie lattice
\[
\fu_3 = \{ Z \in \Mat_3(\fo_\fe) \mid \sigma(Z)^\mathrm{tr}
\, W + WZ = 0 \}.
\]
The groups $G^r$ and $B^r$ are finitely generated
torsion-free potent pro-$p$ groups whenever $r \ge 2e$ for $p=2$ and
$r \ge e (p-2)^{-1}$ for $p>2$, where $e$ denotes the ramification
index of~$\ff$ over~$\Q_p$.  In this case $\pi^r \fu_3$, regarded
as a $\Z_p$-Lie lattice, is the Lie lattice associated to~$G^r$.

\begin{pro} \label{pro:ind-Br-Ur} Suppose that $p$ is odd and that
  $\fe$ is unramified over $\ff$.  In the set-up described above and
  subject to $r \ge e (p-2)^{-1}$, the zeta function of the
  representation $\rho = \Ind_{B^r}^{G^r}(\one_{B^r})$
  is
 \begin{equation*}
   \zeta_\rho(s) = q^{3(r-1)} \frac{q^{1-2s}(q-q^{-s}) (q^{-s}u(q^s) +
     q u(q^{-s}) )}{(1-q^{1-6s})}, 
  \end{equation*}
  where $u(X) = X^2 + X + X^{-2}$.  The abscissa of convergence is
  $\alpha(\rho) = \nicefrac{1}{6}$.
\end{pro}

As before, it would be interesting to find an interpretation for the
value $\nicefrac{1}{6}$; it is highly suggestive that the abscissa
agrees with the one found in Proposition~\ref{pro:ind-Br-Gr}.  From
the given formula it is easy to verify that $\zeta_\rho$ satisfies a
functional equation, in accordance with
Theorem~\ref{thm:main-summary}.  A similar formula holds in case $\fe$
is ramified over $\ff$, however the resulting zeta function does not
admit a functional equation.  In the ramified case the denominator is
still $(1-q^{1-6s})$, whereas the numerator is given by a more
complicated formula.

\begin{proof}[Proof of Proposition~\ref{pro:ind-Br-Ur}]
  Since $p$ is odd, the canonical map
  $\fo_\fe \otimes_\fo \fu_3 \to \gl_3(\fo_\fe)$ is
  an isomorphism of $\fo_\fe$-Lie lattices.  Since
  $\fe$ is unramified over $\ff$, we find
  $\delta \in \fo^\times_\fe \smallsetminus \fo$ with
  $\delta^2 \in \fo^\times$.  Then $1,\delta$ form an $\fo$-basis for
  $\fo_\fe$ and $\sigma(\delta) = - \delta$.

  Let $\fk \le \fu_3$ be the Lie sublattice of strictly lower
  triangular matrices; clearly $\fk \oplus \fb = \fu_3$,  and the matrices
  \begin{equation*}
    A = \left(\begin{smallmatrix}
      0 & 0 & 0 \\
      \delta & 0 & 0 \\
      0 & \delta & 0 \\
    \end{smallmatrix}\right), \quad %
    B = \left(\begin{smallmatrix}
      0 & 0 & 0 \\
      1 & 0 & 0 \\
      0 & -1 & 0 \\
    \end{smallmatrix}\right), \quad %
    C = \left(\begin{smallmatrix}
      0 & 0 & 0 \\
      0& 0 & 0 \\
      \delta & 0 & 0 \\
    \end{smallmatrix}\right)
  \end{equation*}
  form an $\fo$-basis for~$\fk$.  Consider an $\fo$-linear map
  $\omega\colon \fu_3 \to \ff$ that vanishes on $\fb$, and
  put $u=\omega(A)$, $v=\omega(B)$ and $w=\omega(C)$.  We need to find
  the symplectic minors of the symplectic form $\tilde \omega$ on
  $\fu_3$, given by $\tilde{\omega}(X,Y) = \omega([X,Y])$ for
  $X,Y \in \fu_3$.

  The symplectic minors are invariant under base change. As
  $\fo_\fe \otimes_\fo \fu_3 \cong \gl_3(\fo_\fe)$
  we simply need to compare the basis $A,B,C$ with the basis
  $E_{2,1}, E_{3,2}, E_{3,1}$ used in the proof of
  Proposition~\ref{pro:ind-Br-Gr}.  In the notation used there, we
  obtain $x = \frac{1}{2}(v + \delta^{-1}u)$,
  $y = \frac{1}{2}(\delta^{-1}u-v)$, and $z= \delta^{-1}w$.  Observe
  that
  $\lvert v + \delta^{-1}u \rvert_\Ip = \max \{ \lvert u \rvert_\Ip,
  \lvert v \rvert_\Ip \}$,
  Because $u, v \in \fo$ and $1, \delta$ are $\fo$-linearly
  independent.  We finally get the formula
  \begin{equation*}
    \zeta_\rho(s) = q^{3r} \int_{\ff^3} \| 1, u^2,v^2,w^2,
    u^3, v^3 \|_\Ip^{-1-s} \, \di\mu (u,v,w). 
  \end{equation*}
  Now the explicit formula can be computed as in the proof of
  Proposition~\ref{pro:ind-Br-Gr}.
\end{proof}

%%%

\subsection{Induction from maximal parabolic subgroups to
  $\GL_n^r(\fo)$}\label{sec:ind-maximal-parabolic}
Let $n,t \in \N$ such that $n = t + t'$ with
$1 \le t \le t' \coloneqq n-t$.  We consider the general linear group
$G = \GL_n(\fo)$ and, for $r \in \N$, its principal congruence
subgroup $G^r = \ker (\GL_n(\fo) \to \GL_n(\fo/\Ip^r))$.  We are
interested in the maximal parabolic subgroup of type $(t,n-t)$,
defined by
\begin{align*}
  H = H_{n,t} & = \bigl\{ (g_{ij}) \in \GL_n(\fo) \mid g_{ij} = 0
                \text{ for } i \le t < j \bigr\} \\
              & =    
                \begin{pmatrix}[ccc|ccc]
                  &&& 0 & \cdots & 0 \\
                  & \GL_t(\fo)&  & \vdots & \ddots & \vdots \\
                  &&& 0 & \cdots & 0 \\ \hline
                  \fo & \cdots & \fo & & & \\
                  \vdots & \ddots & \vdots & & \GL_{n-t}(\fo) & \\
                  \fo & \cdots & \fo & & &
                \end{pmatrix},
\end{align*}
and we set $H^r = H \cap G^r$.  The groups $G^r$ and $H^r$ are
finitely generated torsion-free potent pro-$p$ groups whenever
$r \ge 2e$ for $p=2$ and $r \ge e (p-2)^{-1}$ for $p>2$, where $e$
denotes the ramification index of~$\ff$ over~$\Q_p$.  Our aim is to
compute the zeta function of the induced representation
$\Ind_{H^r}^{G^r}(\one_{H^r})$ for such~$r$.

For $m \in \N_0$, define $\calV_q(m) = \prod_{j=1}^m(1-q^{-j})$, viz.\
the volume of $\GL_m(\fo)$ with respect to the additive Haar measure
on $\Mat_{m,m}(\fo)$.  For a subset
$J = \{x_1,\dots, x_{|J|}\} \subseteq \{1, \ldots, t\}$ with
$x_1 < \ldots < x_{|J|}$ we write
\[
V_{n,t}(J) =
\frac{\calV_q(t)\calV_q(n-t)}{\calV_q(t-x_{|J|})\calV_q(n-t-x_{|J|})}
\prod_{j=1}^{|J|} \calV_q(x_j -
x_{j-1})^{-1},
\]
with the convention $x_0 = 0$.

\begin{thm}\label{thm:InductionMaxPara}
  In the set-up described above and subject to $r \ge 2e$ if $p=2$ and
  $r \geq e(p-2)^{-1}$ if $p > 2$, the zeta function of the
  representation $\rho = \Ind_{H^r}^{G^r}(\one_{H^r})$ is
 \begin{equation} \label{equ:max-para-ind}
   \zeta_{\rho}(s) = q^{r t (n-t)} \sum_{J \subseteq \{1,\ldots,t\} } V_{n,t}(J) \prod_{j \in J}
   \frac{ q^{-j(n-j)s}}{1-q^{-j(n-j)s}}.
 \end{equation}
 The abscissa of convergence is
 $\alpha(\rho) = 0$.
\end{thm}

\begin{rmk}
  Even though the formula in~\eqref{equ:max-para-ind} seems
  complicated, it is rather easy to evaluate the occurring terms for
  small values of~$t$.  For instance, for $t=1$ we obtain
  \begin{equation*}
    \zeta_{\rho}(s) = q^{r(n-1)} \left( 1 + \big( 1- q^{-(n-1)} \big)
      \frac{q^{-(n-1)s}}{1-q^{-(n-1)s}} \right) =
    q^{r(n-1)} \frac{1-q^{-(n-1) - (n-1)s}}{1-q^{-(n-1)s}}. 
  \end{equation*}
  This formula will be generalised in Section~\ref{sec:schiefkoerper}.
  It should also be compared with the formula obtained in
  Proposition~\ref{prop:formulaInductionTrees}.
\end{rmk}

The proof of Theorem~\ref{thm:InductionMaxPara} requires some
preparations.  For $N, k \in \N_0$, a \emph{composition} of length $k$
of $N$ is a $k$-tuple $\mathbf{u} = (u_1,\dots,u_k)$ of positive
integers such that $N = u_1 + \ldots + u_k$; we put
$N(\mathbf{u}) = N$ and $\lambda(\mathbf{u}) = k$.  In particular, the
empty tuple is a composition of $0$ of length~$0$.  We denote by
$\comp(N)$ the set of all compositions of~$N$.

Set $\fg = \gl_n(\fo)$, and let
\[
\fh = \{ (X_{ij}) \in \fg \mid X_{ij} = 0 \text{ for } i \le t
< j\}
\]
denote the parabolic $\fo$-Lie sublattice of type $(t,n-t)$, related
to~$H$.  We identify the quotient $\fg/\fh$ with the $\fo$-lattice
$\Mat_{t,t'}(\fo)$ of $t \times t'$-matrices, by projecting onto the
upper right $t \times t'$-block.  Furthermore, we identify the
$\ff$-dual space $(\fg/\fh)^* = \Hom_\fo(\fg/\fh, \ff)$ with the
$\ff$-vector space $\Mat_{t', t}(\ff)$ via the trace form
\[
\underbrace{\Mat_{t',t}(\ff)}_{\cong (\fg/\fh)^*}
\times \underbrace{\Mat_{t,t'}(\fo)}_{\cong \fg/\fh} \to
\ff, \qquad (M,X) \mapsto \Tr(MX),
\]
and we identify the Pontryagin dual
$(\fg/\fh)^\vee \cong \Hom_{\fo}(\fg/\fh,\ff/\fo)$ with the space
$\Mat_{t', t}(\ff/\fo)$; compare Section~\ref{sec:LieLattices}.
Indeed, writing $\overline{M} \in \Mat_{t', t}(\ff/\fo)$ for the image
of $M \in \Mat_{t', t}(\ff)$, we obtain an isomorphism
\begin{align*}
  \Mat_{t',t}(\ff/\fo) %
  & \to (\fg/\fh)^\vee, \quad \overline{M} \mapsto
    \omega_{\overline{M}}, \qquad \text{where} \\  
  & \omega_{\overline{M}} \colon \fg/\fh \to \ff/\fo, \quad
    \omega_{\overline{M}}(X) = \Tr(MX) + \fo \qquad \text{for $M \in
    \Mat_{t', t}(\ff)$.}
\end{align*}

In order to apply Proposition~\ref{pro:zeta-algebraic-formula}, we
need to calculate $\lvert \fg : \stab_\fg (\omega) \rvert$ for every
form $\omega \in (\fg/\fh)^\vee$.  We make use of the
Levi subgroup $L \le_\mathrm{c} H$.  This is the group of block
diagonal matrices
\begin{equation*}
    L = \bigl\{ (g_{ij}) \in H \mid g_{ij} = 0 \text{ for } j \le t <
    i \bigr\} \cong \GL_t(\fo) \times \GL_{n-t}(\fo), 
\end{equation*}
and accordingly we write elements of $L$ as pairs
$(g,h) \in \GL_t(\fo) \times \GL_{n-t}(\fo)$.   The action of $L$ on
$(\fg/\fh)^\vee \cong \Mat_{t',t}(\ff/\fo)$ via the co-adjoint
representation is described explicitly by
\[
(g,h).\overline{M} = \overline{h M g^{-1}} \qquad \text{for $(g,h) \in L$ and
$M \in \Mat_{t',t}(\ff)$;}
\]
indeed, for $X \in \Mat_{t,t'}(\fo) \cong \fg/\fh$ we obverse that
\begin{equation*}
  \bigl((g,h) . \omega_{\overline{M}}\bigr)(X) = \Tr(Mg^{-1}Xh) = \Tr(h M g^{-1} X)
  = \omega_{\overline{h M g^{-1}}}(X). 
\end{equation*} 
  
Next we determine an explicit set of orbit representatives
for the $L$-orbits in $\Mat_{t',t}(\ff/\fo)$.
To this end we introduce the parameter set $\Xi = \Xi_{n,t}$ consisting
of all pairs $\xi = (\mathbf{u},\boldsymbol{\gamma})$, where
$\mathbf{u}$ is a composition satisfying $N(\mathbf{u}) \le t$ and
$\boldsymbol{\gamma} = (\gamma_1, \ldots,
\gamma_{\lambda(\mathbf{u})})$
is a strictly increasing sequence of negative integers so that
$\gamma_1 < \ldots < \gamma_{\lambda(\mathbf{u})} < 0$.  For
$\xi = (\mathbf{u},\boldsymbol{\gamma}) \in \Xi$ we define
\begin{equation*}
  \alpha_i(\xi) = \begin{cases}
    \gamma_k  & \text{if  $\sum_{j=1}^{k-1} u_j < i \le \sum_{j=1}^{k}
      u_j$,} \\
    0  &\text{if  $N(\mathbf{u}) < i$,}
  \end{cases}
  \qquad \text{where $1 \le i \le t$,}
\end{equation*}
and we associate to $\xi$ the matrices 
\begin{equation*}
   M_\xi = \begin{pmatrix}
               \pi^{\alpha_1(\xi)} & & & \\
               & \pi^{\alpha_2(\xi)} & & \\
               & & \ddots & \\
               & & & \pi^{\alpha_t(\xi)} \\ \hline
               0 & \cdots & & 0 \\
               \vdots &  & & \vdots\\
               0 & \cdots & & 0
           \end{pmatrix}
           \in \Mat_{t',t}(\ff) 
\end{equation*}
and $\overline{M_\xi} \in \Mat_{t',t}(\ff/\fo)$.  The elementary
divisor theorem implies that
$\calM = \calM_{n,t} = \{ \overline{M_\xi}\mid \xi \in \Xi \}$ is a
system of representatives of the $L$-orbits in~$\Mat_{t',t}(\ff/\fo)$.

Since $\lvert \fg : \stab_\fg(\omega) \rvert$ is constant on
the co-adjoint orbit of a form $\omega$ with regard to~$L$, it
suffices to determine, for $\xi \in \Xi$, the index
$\lvert \fg : \stab_\fg(\omega_{\overline{M_\xi}}) \rvert$ and the size
of the $L$-orbit of~$\overline{M_\xi}$.

\begin{lem}\label{lem:indexFormula}
  Let $\xi = (\mathbf{u},\boldsymbol{\gamma}) \in \Xi$.  The
  stabiliser of $\omega_{\overline{M_\xi}}$ in $\fg$ satisfies
 \begin{equation*}
   \left|\fg : \stab_\fg(\omega_{\overline{M_\xi}})\right|^{\nicefrac{1}{2}} = 
   \prod_{i=1}^{\lambda(\mathbf{u})} q^{- u_i \gamma_i(n -u_i -2 \sum_{j=1}^{i-1}
     u_j)}.
 \end{equation*}
\end{lem}

\begin{lem}\label{lem:orbitFormula}
  Let $\xi = (\mathbf{u},\boldsymbol{\gamma}) \in \Xi$ and put
  $N = N(\mathbf{u})$.  The $L$-orbit of $\overline{M_\xi}$ has the
  cardinality
 \begin{equation*}
   \lvert L. \overline{M_\xi} \rvert = \frac{\calV_q(t) \calV_q(n-t)}{\calV_q(t-N)
     \calV_q((n-t)-N)} 
   \prod_{i=1}^{\lambda(\mathbf{u})} \Bigl(\calV_q(u_i)^{-1} q^{-(n-N)\gamma_i
     u_i + \sum_{j=1}^{i-1} (\gamma_i-\gamma_j)u_iu_j} \Bigr). 
 \end{equation*}
\end{lem}

We postpone the proofs of the lemmata and explain first how
Theorem~\ref{thm:InductionMaxPara} can be deduced.

\begin{proof}[Proof of Theorem \ref{thm:InductionMaxPara}]
  For any composition $\mathbf{u}$, we set
  \[
  W_\mathbf{u} = \frac{\calV_q(t)
    \calV_q(n-t)}{\calV_q(t-N(\mathbf{u}))
    \calV_q((n-t)-N(\mathbf{u}))} \prod_{i=1}^{\lambda(\mathbf{u})}
  \calV_q(u_i)^{-1}.
  \]
  It follows from Proposition~\ref{pro:zeta-algebraic-formula} and
  from Lemmata~\ref{lem:indexFormula} and \ref{lem:orbitFormula} that
  $\rho = \Ind_{H^r}^{G^r}(\one_{B^r})$ satisfies
  \begin{align*}
    \zeta_\rho(s)  %
    & = q^{r t (n-t)} \sum_{\omega \in (\fg/\Ip)^\vee} \left( \lvert
      \fg : \stab_\fg(\omega) \rvert^{\nicefrac{1}{2}}
      \right)^{-1-s} \\
    &= q^{r t (n-t)} \!\!\sum_{\xi
      =(\mathbf{u},\boldsymbol{\gamma}) \in \Xi}  \left| L .
      \overline{M_\xi}\right| \prod_{i=1}^{\lambda(\mathbf{u})} q^{
      u_i \gamma_i(n - u_i - 2 \sum_{j=1}^{i-1} u_j) (1+s)}\\ 
    &= q^{r t (n-t)} \!\!\sum_{\xi
      =(\mathbf{u},\boldsymbol{\gamma}) \in \Xi} 
      W_{\mathbf{u}}\prod_{i=1}^{\lambda(\mathbf{u})}
      \Bigl(q^{N(\mathbf{u}) \gamma_i u_i - \gamma_i u^2_i -
      \sum_{j<i} (\gamma_i+\gamma_j)u_iu_j}  q^{u_i
      \gamma_i(n - u_i - 2 \sum_{j=1}^{i-1} u_j) s}\Bigr) \\ 
    &= q^{r t (n-t)} \!\!\sum_{\xi
      =(\mathbf{u},\boldsymbol{\gamma}) \in \Xi} 
      W_{\mathbf{u}}\prod_{i=1}^{\lambda(\mathbf{u})}
      q^{u_i \gamma_i(n - u_i - 2 \sum_{j=1}^{i-1}
      u_j) s}. 
  \end{align*}
  Let $\xi = (\mathbf{u},\boldsymbol{\gamma}) \in \Xi$. Since
  $\gamma_1 < \gamma_2 < \ldots < \gamma_{\lambda(u)} < 0$ is strictly
  increasing, we may write
  $\gamma_i = - \sum_{k = i}^{\lambda(\mathbf{u})} \beta_k$ for
  certain positive integers
  $\beta_1,\dots,\beta_{\lambda(\mathbf{u})} \in \N$.  Using this
  reparametrisation, we obtain
   \begin{align*}
     \zeta_\rho(s) 
     &= q^{r t (n-t)} \!\!\sum_{\xi =(\mathbf{u},\boldsymbol{\gamma})
       \in \Xi} W_{\mathbf{u}}
       \prod_{i=1}^{\lambda(\mathbf{u})} q^{u_i \gamma_i ( n - u_i -
       2 \sum_{j=1}^{i-1} u_j ) s}\\
     &= q^{r t (n-t)} \!\!\sum_{\substack{N \in \{0,\ldots,t \} \\  \mathbf{u} \in
     \comp(N)} } W_{\mathbf{u}}
     \sum_{\substack{\beta_1, \beta_2, \ldots, \\ \beta_{\lambda(\mathbf{u})}
     \in \N}}  \,
     \prod_{i=1}^{\lambda(\mathbf{u})} \Bigl( q^{- u_i
     \sum_{k=i}^{\lambda(\mathbf{u})}\beta_k (n - u_i - 2 \sum_{j<i}
     u_j) s} \Bigr) \\
     & = q^{r t (n-t)} \!\!\sum_{\substack{N \in \{ 0,\dots,t \}  \\
     \mathbf{u} \in \comp(N)} }  W_{\mathbf{u}}
     \sum_{\beta_1=1}^\infty \cdots \sum_{\beta_{\lambda(\mathbf{u})}=1}^\infty 
     \Bigl( q^{-\sum_{k=1}^{\lambda(\mathbf{u})} \beta_k
     \sum_{i=1}^{k}  u_i (n - u_i - 2 \sum_{j=1}^{i-1} u_j) s} \Bigr)\\
     &= q^{r t (n-t)} \!\!\sum_{\substack{N \in \{ 0,\ldots,t \} \\
     \mathbf{u} \in \comp(N)} }  W_{\mathbf{u}}
     \prod_{k=1}^{\lambda(\mathbf{u})} \frac{ q^{-\sum_{i=1}^{k}
     u_i(n-u_i-2\sum_{j=1}^{i-1} u_j) s} }{1-q^{- \sum_{i=1}^{k}
     u_i(n-u_i-2\sum_{j=1}^{i-1} u_j) s}}, 
  \end{align*}
  where the last step is derived using the geometric series.
  Finally, we write $x_k = \sum_{i=1}^k u_i$ for $k \in \{ 1, \ldots,
  \lambda(\mathbf{u}) \}$.  Then every composition $\mathbf{u}$ yields
  a set 
  \[
  J = J_\mathbf{u} = \{x_1, x_2, \dots, x_{\lambda(u)}\} \subseteq
  \{1, \ldots, t \}
  \]
  and every such set $J$ corresponds uniquely to a composition of the
  number $N = \max(J)$.  This proves the theorem since
  $\sum_{i=1}^{k} u_i(n-u_i-2\sum_{j=1}^{i-1} u_j) = x_k(n-x_k)$ and
  $W_{\mathbf{u}} = V_{n,t}(J)$.
\end{proof}

\begin{proof}[Proof of Lemma \ref{lem:indexFormula}]
  Let $\xi = (\mathbf{u},\boldsymbol{\gamma}) \in \Xi$.  As explained in
  Section~\ref{sec:LieLattices}, the index
  $\lvert\fg : \stab_\fg(\omega_{\overline{M_\xi}})\rvert$ can be
  computed via the Pfaffians of the symplectic form
  \[
  A_\xi \colon \fg \times \fg \to \ff, \qquad
  A_\xi(X,Y) = w_{M_\xi}([X,Y] + \fh) \coloneqq \Tr(M_\xi \widetilde{[X,Y]}),
  \]
  where $\widetilde{Z} \in \Mat_{t,t'}(\fo)$ denotes the projection of
  $Z \in \fg$ onto its upper right $t \times t'$-block.  The
  elementary $n\times n$-matrices $E_{i,j}$, with entry $1$ in
  position $(i,j)$ and entries $0$ elsewhere, form an $\fo$-basis
  of~$\fg$.  A short calculation yields
 \begin{equation*}
   A_\xi(E_{i_1,j_1},E_{i_2,j_2}) = \begin{cases}
     \pi^{\alpha_{i_1}(\xi)}  &\text{ if } j_1 = i_2 \text{, } j_2 =
     i_1 + t \text{ and }  i_1 \le t \\ 
     -\pi^{\alpha_{i_2}(\xi)} &\text{ if } j_2 = i_1 \text{, } j_1 =
     i_2 + t \text{ and } i_2 \le  t \\ 
     0 & \text{ otherwise }
   \end{cases}.
 \end{equation*}
 Thus $\fg$, equipped with the form~$A_\xi$, decomposes into an
 orthogonal direct sum of subspaces of three different types.

 \smallskip

 \noindent \textsl{Type 1.} For $i>t$ and $j \notin \{t+1,\ldots,2t\}$,
 the space $\fo E_{i,j}$ is contained in the radical of $A_\xi$, this means,
 $A_\xi(E_{i,j},X) = 0$ for all $X \in \fg$.

 \smallskip

 \noindent \textsl{Type 2.} For $1\le i \le t$ and $j > 2t$, the
 free $\fo$-module of rank~$2$ spanned by $E_{i,j}$ and $E_{j, i+t}$
 is orthogonal to all other elementary matrices.  The form $A_\xi$
 restricted to this space has the Pfaffian $\Ip^{\alpha_i(\xi)}$.
 For each $i$ with $1\le i \le t$ there are $(n-2t)$ possible
 choices for a corresponding index~$j$.

 \smallskip

 \noindent \textsl{Type 3.} For $1 \le i,j \le t$ the $\fo$-module
 of rank~$3$ spanned by $E_{i,j}$,$E_{j, i+t}$, and $E_{i+t,j+t}$ is
 orthogonal to all other elementary matrices.  The form $A_\xi$
 restricted to this space has
 the degree-$1$ Pfaffian $\Ip^{\min(\alpha_i(\xi),\alpha_j(\xi))}$.
       
 \smallskip

 The sequence of numbers $\alpha_i(\xi)$, $i \in \{1, \ldots, t \}$,
 is increasing.  Thus we can determine the maximal Pfaffian of $A_\xi$,
 and hence 
 \begin{equation*}
   \lvert\fg :
   \stab_\fg(\omega_{\overline{M_\xi}})\rvert^{\nicefrac{1}{2}} =
   q^{-\sum_{i=1}^t \alpha_i(\xi) (n- 2i + 1 ) } =
   q^{-\sum_{i=1}^{\lambda(\mathbf{u})} \gamma_i u_i(n- u_i
     -2\sum_{j<i}u_j)}. \qedhere
 \end{equation*}
\end{proof}

\begin{proof}[Proof of Lemma \ref{lem:orbitFormula}]
  Let $\xi = (\mathbf{u},\boldsymbol{\gamma}) \in \Xi$ and put
  $N = N(\mathbf{u})$.  We compute the stabiliser
  $\Stab_L(\overline{M_\xi})$ of
  $\overline{M_\xi} \in \Mat_{t',t}(\ff/\fo)$ and its volume with
  respect to the normalised additive Haar measure on
  $\Mat_{t,t}(\fo) \times \Mat_{t',t'}(\fo)$.  For simplicity we
  denote by $\vol$ the normalised additive Haar measure on any
  implicitly given $\fo$-lattice.
 
  Suppose $(g,h) \in L = \GL_t(\fo) \times \GL_{t'}(\fo)$ lies in the
  stabiliser of~$\overline{M_\xi}$.  Write
 \begin{equation*}
     g = \begin{pmatrix}
            A & C \\
            U & V \\
         \end{pmatrix} \quad \text{ and } \quad
     h = \begin{pmatrix}
            B & X \\
            D & Y \\
         \end{pmatrix} 
 \end{equation*}
 with 
 \begin{align*}
   & A,B \in \Mat_{N,N}(\fo), && C, U^\mathrm{tr} \in \Mat_{N, t-N}(\fo), &&
                                                                   D, X^\mathrm{tr} \in \Mat_{t'-N, N}(\fo), \\
   & V \in \Mat_{t-N, t-N}(\fo), && Y \in \Mat_{t'-N,t'-N}(\fo).
 \end{align*}
 Let $M_\xi^\circ$ denote the upper left $N\times N$-block of the
 matrix $M_\xi$.  The assumption
 $\overline{hM_\xi} = \overline{M_\xi g}$ is equivalent to the
 following three conditions
 \begin{enumerate}[ (i) ]
 \item\label{cond1} $BM_\xi^\circ - M_\xi^\circ A \equiv_\fo  0 $,
 \item\label{cond2} $DM_\xi^\circ \equiv_\fo 0$,
 \item\label{cond3} $M_\xi^\circ C \equiv_\fo 0$.
 \end{enumerate}
 Writing $D = (d_{ij})$, we observe that
 condition~\eqref{cond2} is equivalent to
 $d_{ij} \in \Ip^{-\alpha_j(\xi)}$ for $1 \le i \le t'-N$ and
 $1 \le j \le N$.  Thus the set $Z_2$ of all matrices
 $D \in \Mat_{t'-N, N}(\fo)$ that satisfy
 condition~\eqref{cond2} has volume
 \begin{equation*}
   \vol(Z_2) = \prod\nolimits_{j=1}^{N} q^{\alpha_j(\xi)(t'-N)} =
   \prod\nolimits_{i=1}^{\lambda(\mathbf{u})} q^{\gamma_i u_i (t'-N)}. 
 \end{equation*}
 
 Similarly, writing $C = (c_{ij})$, we see that
 condition~\eqref{cond3} means:
 $c_{ij} \in \Ip^{-\alpha_i(\xi)}$ for $1 \le i \le N$ and
 $ j \le t-N$.  Thus the volume of the set $Z_3$ of all matrices
 $C \in \Mat_{N, t-N}(\fo)$ that satisfy
 condition~\eqref{cond3} is
 \begin{equation*}
   \vol(Z_3) = \prod\nolimits_{i=1}^{N} q^{\alpha_i(\xi)(t-N)} =
   \prod\nolimits_{i=1}^{\lambda(\mathbf{u})} q^{\gamma_i u_i (t-N)}. 
 \end{equation*}
 Since $\alpha_i(\xi) < 0$ for $1 \le i \le N$, the matrices $C$ and
 $D$ vanish modulo~$\pi$.  Consequently, the square matrices
 $A,B,V, Y$ are invertible over~$\fo$.
 
 Now consider condition~\eqref{cond1}.  Writing
 $A = (a_{ij})$ and $B = (b_{ij})$, we see that
 condition~\eqref{cond1} is equivalent to
 \begin{equation} \label{equ:cond-iii}
   a_{ij} \pi^{\alpha_i(\xi)} \equiv_\fo b_{ij} \pi^{\alpha_j(\xi)}
   \qquad \text{ for } 1 \le i, j \le N. 
 \end{equation}
 Let $Z_1$ denote the set of all pairs
 $(A,B) \in \GL_N(\fo) \times \GL_N(\fo)$
 satisfying~\eqref{equ:cond-iii}.  We observe that, if
 $ (A,B) \in Z_1$, then $A$ is contained in the group 
 \begin{align*}
   G(\xi) & = \{ (z_{ij}) \in \GL_N(\fo) \mid z_{ij} \in
            \pi^{\alpha_j(\xi)-\alpha_i(\xi)} \fo \text{ for } i \le
            j\} \\
          & = 
            \begin{pmatrix}
              \GL_{u_1}(\fo) & \Ip^{\gamma_2-\gamma_1} & \cdots &
              \Ip^{\gamma_{\lambda(\mathbf{u})} - \gamma_1} \\
              \fo &  \GL_{u_2}(\fo) & \ddots& \vdots \\
              \vdots & \ddots & \ddots & \Ip^{\gamma_{\lambda(\mathbf{u})} -
                \gamma_{\lambda(u)-1}}\\
              \fo & \cdots & \fo & \GL_{u_{\lambda(\mathbf{u})}}(\fo)\\
            \end{pmatrix}.
 \end{align*}
 Moreover, for any $A \in G(\xi)$, there are matrices $B$ such that
 $(A,B) \in Z_1$ and each entry $b_{ij} \in \fo$ is uniquely
 determined modulo~$\pi^{-\alpha_j(\xi)}$.  We deduce that
 \begin{equation*}
   \vol(Z_1) = \vol(G(\xi)) \prod_{k=1}^N q^{\alpha_k(\xi)N} =
   \prod_{i=1}^{\lambda(\mathbf{u})}\Bigl( \vol(\GL_{u_i}(\fo))
   q^{N\gamma_iu_i} \prod_{j=i+1}^{\lambda(\mathbf{u})} q^{(\gamma_i -
     \gamma_j)u_iu_j}\Bigr).
 \end{equation*}
 Recalling that $\vol(\GL_{m}(\fo)) = \calV_q(m)$ for $m \in \N$, we
 combine the three volume computations to conclude that
 \begin{align*}
   \vol(\Stab_L(\overline{M_\xi})) %
   & = \vol(Z_1) \vol(Z_2) \vol(Z_3) \vol(\GL_{t-N}(\fo))
     \vol(\GL_{t'-N}(\fo)) \\ 
   & = \calV_q(t-N) \calV_q(t'-N) \prod_{i=1}^{\lambda(\mathbf{u})}
     \Bigl( \calV_q(u_i) q^{(t+t'-N) u_i\gamma_i +
     \sum_{j=i+1}^{\lambda(\mathbf{u})} (\gamma_i - \gamma_j)u_iu_j}
     \Bigr). 
 \end{align*}
 The claim now follows from the observation that the orbit length can
 be expressed in terms of the volume as follows:
 \begin{equation*}
   \lvert L . \overline{M_\xi} \rvert = \lvert L : \Stab_L(\overline{M_\xi}) \rvert =
   \frac{\calV_q(t) \calV_q(n-t)}{\vol(\Stab_L(\overline{M_\xi}))}. \qedhere
 \end{equation*}
\end{proof}

%%%

\subsection{Induction from maximal $(1,n)$-parabolic subgroups to
  $\GL_{n+1}^{dr}(\Delta)$} \label{sec:schiefkoerper} Let $\fo$ be a
compact discrete valuation ring of characteristic~$0$, residue
characteristic~$p$ and residue field cardinality~$q$.  Fix a
uniformiser $\pi$ so that the valuation ideal of $\fo$ takes the form
$\Ip = \pi \fo$.  Let $\ff$ denote the fraction field of $\fo$, a
finite extension of~$\Q_p$, and consider a central division algebra
$\fd$ of index $d$ over~$\ff$ so that $\dim_\ff \fd = d^2$.  It is
known that $\fd$ contains a unique maximal $\fo$-order~$\Delta$.  Up
to isomorphism, $\fd$ and $\Delta$ can be described explicitly, in
terms of the index $d$ and a second invariant $h$ satisfying
$1 \le h \le d$ and $\gcd(h,d)=1$; see~\cite[\S14]{Re03}.

Fix $n \in \N$ and set $G = \GL_{n+1}(\Delta)$.  For $r \in \N$, we
take interest in the principal congruence subgroup
\begin{equation*}
  G^{dr} = \ker\bigl( \GL_{n+1}(\Delta) \to \GL_{n+1}(\Delta/\pi^r\Delta) \bigr).
\end{equation*}
The maximal parabolic subgroup
\begin{equation*}
  H  =  
  \begin{pmatrix}[c|ccc]
    \GL_1(\Delta) & 0 & \cdots & 0\\ \hline
    \Delta & & & \\
    \vdots & & \GL_n(\Delta) & \\
    \Delta & & & \\
  \end{pmatrix}
  \le \GL_{n+1}(\Delta)
\end{equation*}
gives rise to a subgroup
$H^{dr} = H \cap G^{dr} \le_\mathrm{c} G^{dr}$.  The groups $G^{dr}$
and $H^{dr}$ are finitely generated torsion-free potent pro-$p$ groups
whenever $r \ge 2e$ for $p=2$ and $r \ge e (p-2)^{-1}$ for $p>2$,
where $e$ denotes the ramification index of~$\ff$ over~$\Q_p$; compare
\cite[Prop.~2.3]{AvKlOnVo13}.
  
\begin{pro} \label{pro:ind-max-par-div-ring} In the set-up described
  above and subject to $r \geq 2e$ for $p=2$ and $r \ge e (p-2)^{-1}$
  for $p>2$, the zeta function of the induced representation
  $\rho = \Ind_{H^{dr}}^{G^{dr}}(\one_{H^{dr}})$ is
  \begin{equation*}
    \zeta_\rho(s) =  q^{rnd^2} \frac{1  -
      q^{-dn(1+s)}}{1 - q^{-dns}}.      
  \end{equation*}
\end{pro}

\begin{proof}
  The division algebra $\fd$ contains a splitting field $\fe$ that is
  unramified of degree $d$ over $\ff$.  Thus $\fe = \ff(\xi)$, where
  $\xi$ is a primitive $(q^d-1)$th root of unity, and $\fe \vert \ff$
  is a cyclic Galois extension.  The ring of integers in $\fe$ is
  $\fo[\xi]$.  Furthermore there exists a uniformiser $\Pi \in \fd$ so
  that
\[
\Pi^d = \pi \qquad \text{and} \qquad \Pi x = \sigma(x) \Pi \qquad
\text{for all $x \in \fe$,}
\]
where $\sigma$ is the generator of $\mathrm{Gal}(\fe \vert
\ff)$ satisfying $\sigma(\xi) = \xi^{q^h}$.

In this way we obtain an $\ff$-basis $(\xi^i \Pi^j)_{0
  \le i,j \le d-1}$ for $\fd$ that is at the same time an
$\fo$-basis for $\Delta$:
\[
\Delta = \sum\nolimits_{i=0}^{d-1} \sum\nolimits_{j=0}^{d-1}
\fo\, \xi^i \Pi^j.
\]
The reduced trace $\tr_{\fd \vert \ff}(x)$ of an
element $x = \sum_{i=0}^{d-1} \sum_{j=0}^{d-1} a_{i,j}\, \xi^i \Pi^j$
satisfies
\[
\tr_{\fd \vert \ff}(x) = \Tr_{\fe \vert
  \ff} \left(\sum\nolimits_{i=0}^{d-1} a_{i,0}\, \xi^i
\right).
\]
Observe that the symmetric $\ff$-bilinear trace form
$\fe \times \fe \rightarrow \ff$, $(x,y)
\mapsto \Tr_{\fe \vert \ff}(xy)$ restricts to a
symmetric $\fo$-bilinear form $\fo[\xi] \times
\fo[\xi] \rightarrow \fo$ that is non-degenerate,
i.e.\ non-degenerate modulo $\Ip$ at the level of residue
fields.  Consequently, the $\ff$-bilinear pairing
\[
\Delta \times \fd \rightarrow \ff, \quad (x,y)
\mapsto \tr_{\fd \vert \ff}(xy)
\]
induces isomorphisms
\[
\Hom_\fo(\Delta,\ff) \cong \fd \qquad \text{and} \qquad
\Hom_\fo(\Delta,\ff/\fo) \cong \fd / \Pi^{-d+1} \Delta.
\]
For $m = m_0 d + m_1$ with $m_0 \ge 0$ and $0 \le m_1 \le d-1$, we
take a closer look at the related $\fo$-bilinear form
\[
\beta_m \colon \Delta \times \Delta \rightarrow \ff, \quad (x,y)
\mapsto \tr_{\fd \vert \ff}(x y \Pi^{-m}) = \pi^{-m_0} \tr_{\fd \vert
  \ff}(x y \Pi^{-m_1}).
\]
Its structure matrix with respect to the $\fo$-basis
$(\xi^i \Pi^j)_{0 \le i,j \le d-1}$ of $\Delta$, ordered as
\[
(1,\xi, \ldots, \xi^{d-1}, \Pi, \xi \Pi, \ldots, \xi^{d-1}
\Pi, \ldots \ldots, \Pi^{d-1}, \xi \Pi^{d-1}, \ldots, \xi^{d-1}
\Pi^{d-1}),
\]
takes the shape
\[
B_m = \pi^{-m_0}
\begin{pmatrix}[cccc|cccc]
  & & & A_0 & & & & \\
  & & A_1 & & & & & \\
  & \cdots & & & & & & \\
  A_{m_1} & & & & & & & \\ \hline
  & & & & & & & \pi A_{m_1+1} \\
  & & & & & & \pi A_{m_1+2} & \\  
  & & & & & \cdots & & \\
  & & & & \pi A_{d-1} & & & \\
\end{pmatrix}
\in \Mat_{d^2}(\fo),
\]
where $A_i \in \Mat_d(\fo)$ denotes the structure matrix of
the $\fo$-bilinear form 
\[
\fo[\xi] \times \fo[\xi] \rightarrow \fo, \quad (x,y) \mapsto \Tr_{\fe
  \vert \ff}(x \sigma^i(y))
\]
with respect to the $\fo$-basis $(1,\xi, \ldots, \xi^{d-1})$.  Since
each of the latter forms is non-degenerate, we infer that the
elementary divisors of $B_m$ are $\pi^{-m_0}$, with multiplicity
$(m_1+1)d$, and $\pi^{-m_0+1}$, with multiplicity $(d-m_1-1)d$.

\smallskip

We consider the $\fo$-Lie lattice $\fg = \gl_{n+1}(\Delta)$.  The
elementary matrices $E_{ij}$, with $1 \le i,j \le {n+1}$, form a $\Delta$-basis
of the left $\Delta$-module $\fg$.  Let $\fh$ denote the maximal
parabolic $\fo$-Lie sublattice of $\fg$ that is spanned as a
$\Delta$-submodule by those $E_{ij}$ satisfying $(i,j) = (1,1)$ or
$i \ge 2$.  For $m \in \N_0$ with $m \ge d-1$, we take interest in the
$\fo$-linear form
\[
w_m \colon \fg \rightarrow \ff, \quad z = \sum_{1 \le
  i,j \le n+1} z_{ij} E_{ij} \mapsto \tr_{\fd \vert
  \ff}(z_{1,n+1} \Pi^{-m})
\] 
and the induced $\fo$-linear form $\omega_m \colon
\fg \rightarrow \ff/\fo$, $z \mapsto w_m(z)
+ \fo$.  Clearly, $w_m(\fh) = \{0\} \subseteq
\fo$.  Furthermore, the co-adjoint action of the Levi
subgroup
\[
L = 
   \begin{pmatrix}[c|ccc]
     \GL_1(\Delta) & 0 & \cdots &  0\\ \hline
     0 & & & \\
     \vdots & & \GL_n(\Delta) & \\
     0 & & & \\
   \end{pmatrix} 
   \le \GL_{n+1}(\Delta),
\]
associated with $\fh$, on $\Hom_\fo(\fg,\ff)$ maps
$W = \{ w \in \Hom_\fo(\fg,\ff) \mid w(\fh) \subseteq \fo \}$ to
itself.  We observe that, modulo $\Hom_\fo(\fg,\fo)$, the $L$-orbits
on $W$ are parametrised by the elements $w_m$, $m \ge d-1$.
Furthermore, the volume of $W_m = L.w_m + \Hom_\fo(\fg,\fo)$ is equal
to
\[
\mu(W_m) = 
\begin{cases}
  (q^{dn}-1) q^{dn(m-d)} & \text{if $m \ge d$,} \\
  1 & \text{if $m=d-1$,}
\end{cases}
\]
where $\mu$ denotes the Haar measure normalised so that
$\mu(\Hom_\fo(\fg,\fo)) = 1$.

We choose a total order $\prec$ on index pairs $\{ (k,l) \mid 1 \le
k,l \le n+1 \}$ such that
\begin{align*}
  (1,1) & \prec (1,n+1) \prec (n+1,n+1) \\
        & \prec\; (1,2) \prec (2,n+1) \;\prec\; (1,3) \prec (3,n+1) \;\prec\;
          \ldots \;\prec\; (1,n)
          \prec (n,n+1) \\ 
        & \prec\; \text{any $(k,l)$ with $k \ne n+1$ and $l \ne 1$.}
\end{align*}
Let $\mathbf{Y}$ denote the $\fo$-basis of $\fg$
consisting of the basis elements
\[ 
\xi^i \Pi^j \, E_{kl}, \quad \text{where $0 \le i,j \le d-1$ and
  $1 \le k,l \le n+1$,}
\]
ordered according to
\[
(i,j,k,l) < (i',j',k',l') \quad \Leftrightarrow \quad
\begin{cases}
  & \text{$(k,l) \prec (k',l')$, or} \\
  & \text{$(k,l)=(k',l')$ and $j<j'$, or} \\
  & \text{$(k,l)=(k',l')$ and $j=j'$ and $i<i'$.}
\end{cases}
\]
Evaluating the commutator matrix $\calR(\mathbf{T})$ of the
$\fo$-Lie lattice $\fg$ with respect to $\mathbf{Y}$
at the point $\underline{x}_m$ corresponding to $w_m$, we obtain
\[
\calR(\underline{x}_m) =
\begin{pmatrix}[ccccc|cc]
\calB_0 &&&&&&\\
 & \calB_1 &&&&&\\
 & & \calB_2 &&&&\\
 & & & \ddots &&&\\
 & & & & \calB_{n-1} &&\\ \hline
 & & & & & \phantom{\ddots} \\
 & & & & & & \phantom{\ddots}
\end{pmatrix}
\in \Mat_{(n+1)^2d^2}(\fo),
\]
where
\[
\calB_0 =
\begin{pmatrix}
  0 & B_m & 0 \\
  -B_m^\mathrm{tr} & 0 & B_m \\
  0 & -B_m^\mathrm{tr} & 0 
\end{pmatrix}
\in \Mat_{3d^2}(\fo)
\]
and
\[
\calB_1 = \ldots = \calB_{n-1} =
\begin{pmatrix}
  0 & B_m \\
  -B_m^\mathrm{tr} & 0
\end{pmatrix}
\in \Mat_{2d^2}(\fo)
\]
each have elementary divisors $\pi^{-m_0}$, with multiplicity
$2(m_1+1)d$, and $\pi^{-m_0+1}$, with multiplicity $2(d-m_1-1)d$, the
remaining $d^2$ elementary divisors of $\calB_0$ being
$\pi^\infty = 0$.  This gives
\begin{align*}
  \lvert \fg : \stab_\fg(\omega_m) \rvert^{\nicefrac{1}{2}} & =
  \left\| \bigcup \{ \Pfaff_k(w_m) \mid 0 \le k \le \lfloor
    \nicefrac{d^2(n+1)^2}{2} \rfloor \} \right\|_\Ip \\
  & = \left\| \Pfaff_{nd^2}(w) \right\|_\Ip \\
  & = \lvert (\pi^{-m_0})^{(m_1+1)d n} (\pi^{-m_0+1})^{(d-m_1-1)d
    n} \rvert_\Ip \\
  & = q^{nd(m-(d-1))}.
\end{align*}

Thus, by Proposition \ref{pro:zeta-integral-formula}, the zeta function of
$\rho = \Ind_{H^{dr}}^{G^{dr}}(\mathbb{1}_{H^{dr}})$ is given by
\begin{align*}
  \zeta_\rho(s) & = q^{rnd^2} \left( \sum_{m=d-1}^\infty \mu(W_m)
    \left\| \bigcup \{ \Pfaff_k(w_m) \mid 0 \le k \le \lfloor
      \nicefrac{d^2(n+1)^2}{2} \rfloor \} \right\|_\Ip^{\, -1-s} \right) \\
  & = q^{rnd^2} \left( 1 + \sum_{m=d}^\infty \left(
      \left(q^{dn}-1\right) q^{dn(m-d)} \right)
    \left( q^{nd(m-(d-1))} \right)^{-1-s} \right) \\
%   & = q^{rnd^2} \left( 1 + \left(1 - q^{-dn}\right)
%     \sum_{m'=1}^\infty \left( q^{-dns} \right)^{m'} \right) \\
  & = q^{rnd^2} \left( 1 + \left(1 - q^{-dn}\right)
    \frac{q^{-dns}}{1 - q^{-dns}} \right) \\
  & =  q^{rnd^2} \frac{1  -
      q^{-dn(1+s)}}{1 - q^{-dns}}. \qedhere
\end{align*}
\end{proof}

%%%%%
\addtocontents{toc}{\setcounter{tocdepth}{0}} % do not include section
                                % in ToC
\section*{Acknowledgements}
\addtocontents{toc}{\setcounter{tocdepth}{1}} % go back to normal

We thank the anonymous referee for many careful and detailed comments.
In particular, we are grateful for the encouragement to clarify
whether the main result of~\cite{GoJaKl14} extends to zeta functions
of induced representations; this led us to prove and include
Proposition~\ref{pro:vanishing-result}.

%%%%% 

\end{document}